\documentclass[11pt]{article}

\usepackage[T1]{fontenc}
\usepackage[utf8]{inputenc}
\usepackage[margin=1in]{geometry}
\usepackage{lmodern}
\usepackage{microtype}
\usepackage{amsmath,amssymb,amsfonts,mathtools}
\usepackage{amsthm}
\usepackage{booktabs}
\usepackage{graphicx}
\usepackage{subcaption}
\usepackage{float}


\usepackage{array}
\usepackage{longtable}
\usepackage{mathrsfs}
\usepackage{enumitem}
\usepackage{xcolor}
\usepackage{hyperref}
\usepackage[capitalize,nameinlink]{cleveref}

\hypersetup{
  colorlinks=true,
  linkcolor=blue!55!black,
  citecolor=blue!55!black,
  urlcolor=blue!55!black
}

\graphicspath{{paper/figures/}}
\setlist[enumerate]{leftmargin=.5in}
\setlist[itemize]{leftmargin=.5in}

\newtheorem{theorem}{Theorem}
\newtheorem{proposition}{Proposition}
\newtheorem{lemma}{Lemma}

\newtheorem{assumption}{Assumption}
\newtheorem*{theorem*}{Theorem}
\newtheorem*{proposition*}{Proposition}
\theoremstyle{remark}
\newtheorem{remark}{Remark}

\newenvironment{keywords}{\par\vspace{0.5em}\noindent\textbf{Keywords.}\ }{\par}
\newenvironment{MSCcodes}{\par\vspace{0.25em}\noindent\textbf{MSC codes.}\ }{\par\vspace{1em}}

\title{Controlled McKean--Vlasov Contagion with State-Dependent Killing\thanks{This version replaces the original combined preprint \emph{Default Contagion, Matrix Approximation, and Control in Sparse Financial Networks}, \href{https://arxiv.org/abs/2605.24833}{arXiv:2605.24833v1}, and contains the controlled McKean--Vlasov contagion, killed-HJB, and steep-killing part of that work.}}

\author{Aoxin Zhang\thanks{School of Mathematical Sciences, Beijing Normal University, Beijing, China (\url{202211130012@mail.bnu.edu.cn}). ORCID: \url{https://orcid.org/0009-0006-3647-9545}.}
\and Yingzhe Wang\thanks{School of Mathematical Sciences, Beijing Normal University, Beijing, China (\url{wangyz@bnu.edu.cn}). Corresponding author. ORCID: \url{https://orcid.org/0000-0003-2113-3307}.}}
\date{}

\hypersetup{
  pdftitle={Controlled McKean--Vlasov Contagion with State-Dependent Killing},
  pdfauthor={Aoxin Zhang and Yingzhe Wang}
}

\begin{document}

\maketitle

\begin{abstract}
We study controlled McKean--Vlasov contagion with state-dependent killing, common noise, loss feedback, and interacting populations. The main result is a comparison principle for the two-population killed-particle HJB on a decomposed state space of alive sub-probability measures and cemetery masses. The proof combines a Wasserstein smooth-gauge comparison argument with a killing-jump absorption estimate for mass transfer into the cemetery state. We also establish a multi-population mean-field limit, an explicit first-order particle convergence rate, conditional propagation of chaos, controlled well-posedness, and a steep-killing bridge to absorbing-boundary default. Finite-particle convergence tests and a two-population HJB feedback experiment illustrate the theory.
\end{abstract}

\begin{keywords}
McKean--Vlasov equations, default contagion, state-dependent killing, common noise, optimal control, HJB equations, propagation of chaos
\end{keywords}

\begin{MSCcodes}
60H30, 60K35, 93E20, 49L20, 91G40, 91G80
\end{MSCcodes}

\section{Introduction}
\label{sec:introduction}

Default contagion models connect a bank's distance to distress with the loss state of the surrounding financial system. In large systems this interaction is naturally expressed through McKean--Vlasov dynamics: individual institutions carry idiosyncratic noise and common shocks, while the empirical loss process feeds back into the drift of surviving institutions. A regularized default mechanism based on state-dependent killing is especially useful in this setting. It retains the economic interpretation that weaker institutions default more intensely, it leaves the alive-state dynamics continuous between killing events, and it creates a tractable interface with stochastic control and dynamic programming. The present paper develops this interface for multi-population contagion systems with common noise and regulatory intervention.

The central analytical issue is the value function. A controlled contagion model with killing is naturally formulated on a state variable consisting of alive distributions and cemetery masses. The alive distributions evolve under diffusion, drift, common noise, and loss feedback, while killing transfers mass from the alive state into the cemetery coordinate. This transfer is a first-order jump term in the loss coordinate and must be handled together with Wasserstein derivatives in the alive-measure variables. Proving comparison for the corresponding HJB equation is therefore the step that turns the controlled McKean--Vlasov formulation into a well-posed stochastic control problem. Our main technical contribution is a killed-particle comparison principle in the two-population case. The proof works on the decomposed state space of alive sub-probability measures and cemetery masses, applies a Wasserstein smooth-gauge/Ishii argument to the alive coordinates, and absorbs the killing jump through a bounded Lipschitz estimate for the mass-transfer term.

The paper makes three contributions. First, we prove a comparison principle and an HJB characterization for the two-population controlled contagion problem with common noise and bounded state-dependent killing. The cemetery masses are treated as genuine state variables rather than as hidden components of probability measures on an enlarged state space. This formulation isolates the killing jump and yields the finite-dimensional marginal-value formula used by the numerical feedback computation. Second, we establish the mean-field foundation for the finite-particle model. The result covers multiple populations, common noise, bounded state-dependent killing, loss feedback, compact controls, and a quantitative first-order particle error under explicit empirical-process inputs. It also gives conditional same-type propagation of chaos after conditioning on the common shock. Third, we prove a steep-killing bridge. Under a localized boundary-regularity package, the bounded killing formulation selects the absorbing-boundary contagion model in a steep-intensity limit. This result clarifies how the regularized model used for control relates to the hard-default convention used in the systemic-risk literature.

The closest points of contact are the McKean--Vlasov contagion models of Hambly, Ledger, and S\o jmark~\cite{HamblyLedgerSojmark2019}, Hambly and S\o jmark~\cite{HamblySojmark2019}, Feinstein and S\o jmark~\cite{FeinsteinSojmark2023}, and the controlled killed-contagion framework of Hambly and Jettkant~\cite{HamblyJettkant2023}. Those works establish the stochastic-analysis foundation for default contagion with absorbing boundaries, heterogeneous impact, and control. The present paper contributes a multi-population killed-particle HJB comparison principle with common noise and a steep-killing bridge from bounded killing to absorbing-boundary default. The comparison proof uses the Wasserstein-space viscosity framework of Cosso, Gozzi, Kharroubi, Pham, and Rosestolato~\cite{CossoGozziKharroubiPhamRosestolato2024}, but the cemetery mass coordinate and the killing-jump absorption estimate are specific to the default-contagion setting.

Finite-network matrix approximation, sparse-network diagnostics, and node-level graph-pressure information are outside the scope of the present paper. The present paper keeps the focus on the stochastic-analysis and control theory: the finite-type model, the mean-field limit, the killed HJB comparison theorem, the steep-killing limit, and the numerical HJB feedback loop.

The paper is organized as follows. Section~\ref{sec:model} defines the finite-type killed contagion model and the central control objective. Section~\ref{sec:mean-field-theory} states the mean-field limit, quantitative convergence rate, and propagation-of-chaos result. Section~\ref{sec:control-theory} gives controlled well-posedness, the killed HJB characterization, and the steep-killing bridge. Sections~\ref{sec:convergence-numerics} and~\ref{sec:hjb-numerics} give the particle convergence check and the finite-dimensional HJB feedback diagnostic. Section~\ref{sec:main-analytical-proofs} contains the main analytical proofs, with the long primitive verification of the steep-killing bridge and the HJB discretization details placed in Appendices~\ref{app:steep-killing-primitive} and~\ref{app:numerical-calibration}. Section~\ref{sec:discussion-conclusion} concludes.

\paragraph{Related literature.}
The paper belongs to the interacting-particle literature on systemic default contagion. Hambly, Ledger, and S\o jmark~\cite{HamblyLedgerSojmark2019} and Hambly and S\o jmark~\cite{HamblySojmark2019} develop distance-to-default systems in which default losses feed back into the dynamics of surviving institutions. Feinstein and S\o jmark~\cite{FeinsteinSojmark2023} study heterogeneous impact and exposure in contagious McKean--Vlasov systems, while related mean-field and systemic-risk formulations appear in Fouque and Ichiba~\cite{FouqueIchiba2013}, Garnier, Papanicolaou, and Yang~\cite{GarnierPapanicolaouYang2013}, and Carmona, Fouque, and Sun~\cite{CarmonaFouqueSun2015}. Our mean-field theorem follows this line, and the propagation-of-chaos component uses standard compactness and empirical-process ideas from the McKean--Vlasov particle literature~\cite{Sznitman1991,Meleard1996}. It is stated for a regularized killed system with multiple populations, common noise, loss feedback, and compact regulatory controls. The steep-killing bridge also relates the bounded-intensity regularization to the absorbing and hitting-time formulations studied in networked integrate-and-fire and systemic-risk models~\cite{DelarueInglisRubenthalerTanre2015,NadtochiyShkolnikov2019}.

The control part is related to McKean--Vlasov control and mean-field games under common noise. The general probabilistic and dynamic-programming background is developed in Carmona and Delarue~\cite{CarmonaDelarue2018}, Fabbri, Gozzi, and \v{S}wi\k{e}ch~\cite{FabbriGozziSwiech2017}, and the weak-formulation and common-noise mean-field-game literature~\cite{CarmonaLacker2015,Lacker2015,CarmonaDelarueLacker2016,CardaliaguetDelarueLasryLions2019}. Hambly and Jettkant~\cite{HamblyJettkant2023} give a controlled McKean--Vlasov contagion model with killing and an HJB characterization. The present paper focuses on the comparison-principle side of the killed problem, where the state variable is decomposed into alive measures and cemetery masses. This decomposition makes the default jump explicit and requires a separate absorption estimate in the doubled-variable comparison argument.

The viscosity comparison argument draws on recent HJB theory on Wasserstein spaces, especially Cosso, Gozzi, Kharroubi, Pham, and Rosestolato~\cite{CossoGozziKharroubiPhamRosestolato2024}. Their smooth-gauge method supplies the alive-measure component of the proof. The new ingredient here is the killing term: alive mass is removed at a state-dependent rate and inserted into a finite-dimensional loss coordinate. The proof therefore combines Wasserstein estimates for alive distributions with finite-dimensional penalties for cemetery masses. This structure is specific to killed default contagion and is also the reason that the theorem is currently stated for two populations.

\section{Model}
\label{sec:model}

\subsection{Finite-type state-dependent killing model}

Consider a system of $N$ financial institutions. Institution $i$ has type $g(i)\in\{1,\ldots,K\}$, and the state variable $X_i(t)$ represents its safety margin from the default region. The baseline kill intensity of type $k$ is
\begin{equation}
    \lambda_{0,k}(x)=\lambda_{\mathrm{base},k}\exp\{\gamma_k(x_{b,k}-x)\},
\end{equation}
where $x_{b,k}$ is the risk boundary and $\gamma_k$ is the state-sensitivity coefficient. Regulatory control $a$ reduces the kill intensity in exponential form:
\begin{equation}
    \lambda_k(x,a)=\lambda_{0,k}(x)\exp\{-\eta_k a\}.
\end{equation}
For rigorous results we work with the corresponding bounded Lipschitz regularization of this intensity on the relevant compact state-control domain; the convention is stated before Theorem~\ref{thm:multiclass-mkv-limit} below. The discrete-time kill probability used in numerical experiments is
\begin{equation}
    p_k(x,a;\Delta t)=1-\exp\{-\lambda_k(x,a)\Delta t\}.
\end{equation}
This regularized killing mechanism preserves the state dependence of default risk and avoids the numerical instability of hard absorbing boundaries in finite-sample simulations.

The cumulative loss of type $k$ and the aggregate loss are defined by
\begin{equation}
    L_t^{k,N}=\frac{1}{N_k}\sum_{i:g(i)=k}\mathbf 1_{\{\tau_i\leq t\}},
    \qquad
    L_t^N=\sum_{k=1}^K \pi_k L_t^{k,N},
\end{equation}
where $\pi_k=N_k/N$. The state dynamics of the finite-type matrix model are
\begin{equation}
    dX_i^k(t)
    =
    \left[\beta_k+\kappa_k\bigl(m_t^{k,N}-X_i^k(t)\bigr)\right]dt
    +\sigma_k dW_i^k(t)+\sigma_0 dW^0(t)
    -\sum_{\ell=1}^K \Gamma_{k\ell}\,dL_t^{\ell,N}.
    \label{eq:finite_type_model}
\end{equation}
The matrix element $\Gamma_{k\ell}\geq0$ denotes the average shock from newly realized losses in type $\ell$ to type $k$. The row sum $\sum_\ell\Gamma_{k\ell}$ measures the external exposure of the target type, the column sum $\sum_k\Gamma_{k\ell}$ measures the spillover impact of the source type, and the spectral radius $\rho(\Gamma)$ provides a scalar summary of aggregate feedback strength.

\subsection{Control objective}

The continuous-time central control problem is
\begin{equation}
    \inf_{a\in\mathcal A}
    \mathbb E\left[\Phi(L_T^1,\ldots,L_T^K)
    +\int_0^T\sum_{k=1}^K \pi_k c_k(a_t^k)\,dt
    \right],
    \label{eq:central-control-objective}
\end{equation}
where $\Phi$ is the terminal system-loss penalty, $c_k$ is the type-$k$ running intervention cost, and $\mathcal A$ denotes compact-valued type-level controls adapted to the common-noise filtration. The theory below treats common-noise adapted open-loop controls and relaxed controls. The numerical section additionally solves a finite-dimensional projection HJB for $K=2$ and validates the stored feedback in a forward particle system.

\section{Mean-field limit and quantitative convergence}
\label{sec:mean-field-theory}

This section gives the mean-field foundation of the regularized default contagion model in the main text. We keep only theorem statements, proof sketches, and the explanations needed for the numerical design. Detailed proof is given in Section~\ref{sec:mkv-proofs}.

\paragraph{Convention for the theoretical results.}
The economic specification in \eqref{eq:finite_type_model} uses an exponential state-dependent kill intensity. The rigorous results below use its bounded Lipschitz regularization: for some $\Lambda<\infty$ and $L_\lambda<\infty$, $0\le \lambda_k(x,a)\le\Lambda$ and $|\lambda_k(x,a)-\lambda_k(y,a)|\le L_\lambda |x-y|$ on the relevant compact state-control domain. On any grid $[x_{\min},x_{\max}]\times[0,\bar a]$, the exponential specification satisfies this after capping above the maximum attained grid intensity. Proposition~\ref{prop:killing-to-absorbing-boundary} gives the corresponding steep-killing bridge to absorbing-boundary contagion under the localized boundary-regularity package verified in Appendix~\ref{app:steep-killing-primitive}.

\begin{theorem}[Multi-population McKean--Vlasov limit]\label{thm:multiclass-mkv-limit}
Fix $T>0$ and let $N_*:=\min_k N_k\to\infty$. Assume independent within-type initial states with uniformly finite first moments, Lipschitz drift, bounded Lipschitz kill intensity, bounded contagion matrix, and compact common-noise adapted type-level controls. Then the type-level surviving empirical measures and loss processes of the finite-particle system converge on $[0,T]$ to the unique multi-population McKean--Vlasov limit $(\mu_t^1,\ldots,\mu_t^K,L_t^1,\ldots,L_t^K)$ associated with the same controlled killed-contagion dynamics. More precisely, $\mu_t^{k,N}\Rightarrow\mu_t^k$ and $L_t^{k,N}\to L_t^k$ uniformly in probability for each type $k$.
\end{theorem}

The proof uses synchronous coupling: the finite particles and the limiting particles share initial states, idiosyncratic noises, common noise, and Poisson thinning random measures. The common-noise adapted control convention preserves conditional independence inside each type once the common shock history is fixed. The drift error, killing error, and loss-feedback error are controlled by Lipschitz conditions, and the estimate closes by Gronwall's inequality. This result gives the finite-type matrix experiments a rigorous large-system limit foundation. The uniqueness of the limiting controlled system is established independently in Proposition~\ref{prop:controlled-multiclass-wellposedness}, which proves well-posedness under the same bounded-killing and Lipschitz regularity.

\begin{theorem}[Quantitative convergence rate under explicit empirical-process inputs]\label{thm:quantitative-rate}
Under the conditions of Theorem~\ref{thm:multiclass-mkv-limit}, assume also that the system has a uniform second-moment bound, that the empirical measures of the conditionally independent limiting particles satisfy the standard bounded-Lipschitz quadratic fluctuation estimate, and that the synchronous Poisson-thinning coupling satisfies the corresponding metric-level quadratic bounded-Lipschitz stability estimate. This last estimate means that the bounded-Lipschitz distance between the finite surviving empirical measure and the paired limiting auxiliary empirical measure is controlled in mean square by the state, measure, and loss coupling errors, up to the usual $N_k^{-1}$ martingale fluctuation term. Then there exists a constant $C_T<\infty$ such that, for each $k$,
\begin{equation}
    \sup_{t\in[0,T]}\mathbb E d_{\mathrm{BL}}^2(\mu_t^{k,N},\mu_t^k)
    +\sup_{t\in[0,T]}\mathbb E|L_t^{k,N}-L_t^k|^2
    \le \frac{C_T}{N_*}.
    \label{eq:quantitative-rate-main}
\end{equation}
Therefore, the bounded-Lipschitz empirical-measure mean-square error and the loss-process mean-square error are $O(N_*^{-1})$, and the corresponding first-order errors are $O(N_*^{-1/2})$. The estimate is obtained by first working conditionally on the common-noise path and then integrating over the common noise; the displayed expectation is the resulting unconditional bound.
\end{theorem}

This rate follows by tracking constants in the synchronous coupling argument for Theorem~\ref{thm:multiclass-mkv-limit} and by estimating the loss process through its Doob--Meyer martingale decomposition. Common noise changes only the conditional distribution and the constants; it does not change the empirical fluctuation order driven by the minimum type sample size. We state the quadratic bounded-Lipschitz empirical fluctuation input and the synchronous-thinning metric stability input explicitly because the killed-particle setting requires metric-level control of the bounded-Lipschitz supremum. Under the localized finite-entropy/Donsker condition recorded in Lemma~\ref{lem:sufficient-quadratic-inputs} and Remark~\ref{rem:localized-entropy-verification}, these inputs hold with constants independent of the type sample size, and the displayed estimate is an unconditional consequence after integrating over the common noise. The empirical one-dimensional $W_2^2$ diagnostics reported below are numerical stability diagnostics consistent with the same first-order $N_*^{-1/2}$ particle scale; the theorem itself is stated in the bounded-Lipschitz metric, where the alive/dead mismatch created by killing is controlled at the mean-square level.

\begin{remark}[Primitive sufficient condition for Theorem~\ref{thm:quantitative-rate}]
The abstract inputs in Theorem~\ref{thm:quantitative-rate} are verified under the localized finite-entropy condition for the killed-path bounded-Lipschitz class and the compact-state localization recorded in Lemma~\ref{lem:sufficient-quadratic-inputs} and Remark~\ref{rem:localized-entropy-verification}. In that primitive regime, the estimate in \eqref{eq:quantitative-rate-main} follows from the bounded-killing, Lipschitz, and finite-entropy assumptions rather than from an additional rate postulate.
\end{remark}

\begin{proposition}[Conditional same-type propagation of chaos]\label{prop:propagation-of-chaos-same-type}
Under the conditions of Theorem~\ref{thm:multiclass-mkv-limit}, fix a type $k$, a time $t\in[0,T]$, and any fixed finite index set $\{i_1,\ldots,i_m\}\subset\{i:g(i)=k\}$. Let $\widehat X_i^{k,N}(t)$ denote the lifted state that equals $X_i^{k,N}(t)$ while the particle is alive and equals the cemetery state $\partial$ after default, and define $\widehat\mu_t^k=\mu_t^k+L_t^k\delta_\partial$. Then, conditionally on the common-noise history,
\[
    \mathcal L\!\left(\widehat X_{i_1}^{k,N}(t),\ldots,\widehat X_{i_m}^{k,N}(t)\,\middle|\,\mathcal F_t^{W^0}\right)
    \Rightarrow (\widehat\mu_t^k)^{\otimes m}
\]
in probability. Equivalently, the finite collection becomes asymptotically independent after conditioning on the common noise. If $\sigma_0=0$, the random measure $\widehat\mu_t^k$ is deterministic, and this conditional statement reduces to the standard unconditional propagation of chaos.
\end{proposition}

This proposition follows directly from Theorem~\ref{thm:multiclass-mkv-limit}, but it has an interpretive role in our setting: the finite-type matrix model averages individual heterogeneity within the same type, and its validity relies on this conditional exchangeability.

\section{Controlled contagion systems and optimal control characterization}
\label{sec:control-theory}

This section gives the well-posedness of the controlled model, the optimal control characterization for two populations, and the relation between the regularized killing mechanism and the hard absorbing boundary. Proposition~\ref{prop:killing-to-absorbing-boundary} is a singular-limit statement and is separate from the bounded-killing framework used in the main limit and control theorems. The controlled well-posedness and killed-HJB comparison proofs are given in Section~\ref{sec:control-proofs}; the longer primitive verification for the steep-killing bridge is given in Appendix~\ref{app:steep-killing-primitive}.

\begin{proposition}[Well-posedness of the controlled multi-population contagion system]\label{prop:controlled-multiclass-wellposedness}
Assume that the drift is Lipschitz in the state and measure variables with linear growth, the kill intensity is bounded and Lipschitz in the state variable, the control sets are compact, the contagion matrix is bounded, the diffusion coefficients are bounded, and the initial states have the required first or second moments. Then, for any common-noise adapted admissible type-level control, the controlled multi-population McKean--Vlasov contagion system has a unique solution. The loss process is a càdlàg nondecreasing process, and the solution is stable with respect to the initial distribution and the control flow.
\end{proposition}

This proposition also supplies the uniqueness input invoked in Theorem~\ref{thm:multiclass-mkv-limit}. The proof constructs a map from measure flows and loss flows to themselves. The map is closed by a contraction argument on short time intervals, and a concatenation argument then gives a global solution on $[0,T]$. This proposition ensures that the optimization object in \eqref{eq:central-control-objective} is well defined.

\begin{theorem}[Two-population optimal control and conditional common-noise HJB characterization]\label{thm:two-pop-optimal-control-hjb-common-noise}\label{thm:killed-two-pop-comparison-principle}
Consider the controlled contagion system with $K=2$, compact controls $[0,\bar a]^2$, bounded Lipschitz killing, well-posed controlled martingale problem, nondegenerate alive-state idiosyncratic diffusions, continuous convex running costs, and bounded uniformly continuous terminal loss penalty. Then the common-noise central control problem admits an optimal relaxed control and its value is characterized by the decomposed killed-state HJB. Section~\ref{sec:control-proofs} proves comparison on the state space of alive sub-probability measures and cemetery masses, identified with $\mathcal M_{\le1,2}(\mathbb R)^2\times[0,1]^2$ subject to $|\nu^k|+L^k=1$. The alive measures are handled by the Wasserstein smooth-gauge/Ishii argument, the cemetery masses by finite-dimensional doubled-variable penalties, and the killing jump by a bounded Lipschitz absorption estimate. Under the standard chattering and measurable-selection conditions stated in the proof section, Section~\ref{sec:control-proofs}, the relaxed value agrees with the strict-control value when the Hamiltonian minimizer has a progressively measurable strict selector.

The viscosity characterization does not require classical differentiability. If a smooth verification is available locally, and for the quadratic running cost $c_k(a)=\frac12 r_k a^2$ in the population-weighted objective $\sum_k\pi_k c_k(a^k)$, the Hamiltonian minimizer satisfies the diagnostic condition
\[
    a_k^*(t)=\Pi_{[0,\bar a_k]}
    \left(
    \frac{\eta_k}{\pi_k r_k}
    \int_{\mathbb R}\lambda_k(x,a_k^*(t))\,\mathcal M_k(t,x,\boldsymbol\mu_t)\,\mu_t^k(dx)
    \right),
\]
where $\mathcal M_k$ is the complete marginal value of suppressing a type-$k$ killing event, including the direct cemetery-jump value and induced loss-feedback value. In the finite-dimensional projected HJB used numerically, this reduces formally to the pointwise fixed-point condition
\[
    a_k^*(t,x_1,x_2,L_1,L_2)=\Pi_{[0,\bar a_k]}
    \left(\frac{\eta_k\lambda_k(x_k,a_k^*)\,\mathcal M_k(t,x_1,x_2,L_1,L_2)}{\pi_k r_k}\right).
\]
\end{theorem}

Under the stated comparison, relaxed-control, and selector inputs, this theorem places the rule-based controls used in this paper within a stochastic control framework. In the numerical section, we solve the finite-dimensional projection HJB for $K=2$ and compare the resulting feedback control with rule-based controls. This projection is the computational counterpart of the measure-valued characterization. The comparison principle is proved for the bounded regularized killing specification with $K=2$; extension to general $K$ requires a $K$-fold product Wasserstein smooth-gauge construction and a complete doubled-variable maximum argument for all cross-type terms. The numerical HJB benchmark is correspondingly kept at $K=2$, a natural core--periphery partition in which systemically important institutions and other banks form the two interacting populations.

\begin{proposition}[Steep killing and absorbing-boundary default]\label{thm:primitive-steep-bridge}\label{prop:killing-to-absorbing-boundary}
The regularized killing mechanism used in the paper has a steep-intensity limit that selects an absorbing-boundary model. Appendix~\ref{app:steep-killing-primitive} verifies the bridge under a localized one-dimensional boundary-regularity package: compact localization, bounded Lipschitz frozen coefficients, nondegenerate idiosyncratic diffusion near the boundary, uniform killed-diffusion boundary estimates, and an absorbing no-atom modulus. The proof combines a Volterra short-time contraction, concatenation, local stability of the absorbing fixed point, and uniform convergence of the driven steep-killing maps on compact localized classes.

If $\lambda_0^{(n)}(x)=\lambda_{\mathrm{base}}\exp(n(x_b-x)^+)$, the selected limit is absorbing-boundary default together with baseline Cox killing above the boundary. For a pure hard absorbing boundary with no baseline Cox killing above the boundary, one can instead use $\widetilde\lambda_0^{(n)}(x)=\lambda_{\mathrm{base}}\{\exp(n(x_b-x)^+)-1\}$, or let the above-boundary baseline intensity vanish with $n$.
\end{proposition}

\section{Particle convergence diagnostics}
\label{sec:convergence-numerics}

Figure~\ref{fig:theorem2-n-convergence} tests the quantitative convergence rate in Theorem~\ref{thm:quantitative-rate}. The horizontal axis is the minimum particle number across types, $N_*$, and the vertical axes report terminal loss MSE and a one-dimensional empirical $W_2^2$ diagnostic of the lifted empirical measure. The theorem is stated in the bounded-Lipschitz metric; the empirical $W_2^2$ panel is used only as a numerical stability diagnostic for the same first-order particle scale. The observed curves follow the $C_T/N_*$ reference order, and Table~\ref{tab:theorem2_convergence_constants} reports log--log slopes and fitted constants. Slopes near $-1$ at the MSE level are consistent with the $N_*^{-1/2}$ RMSE rate in Theorem~\ref{thm:quantitative-rate}.

\begin{figure}[H]
    \centering
    \includegraphics[width=0.78\textwidth]{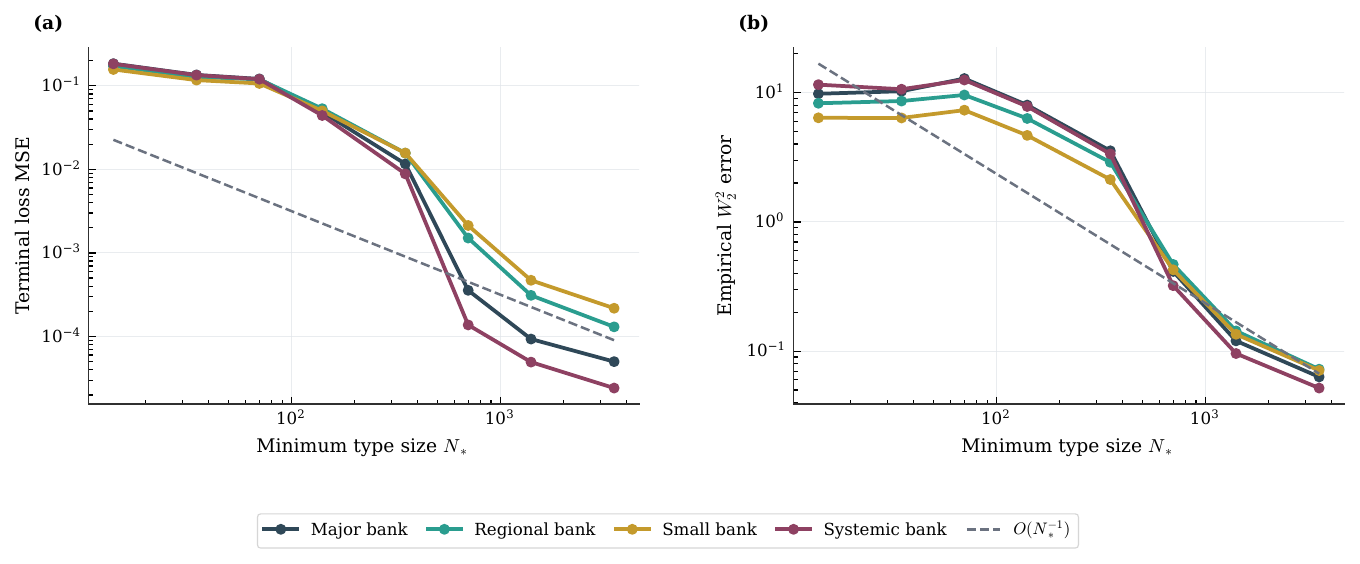}
    \caption{Particle-number convergence validation for Theorem~\ref{thm:quantitative-rate}.}
    \label{fig:theorem2-n-convergence}
\end{figure}

\begin{table}[H]
\centering
\caption{Log--log convergence regressions for Theorem~\ref{thm:quantitative-rate}.}
\label{tab:theorem2_convergence_constants}
\small
\begin{tabular}{llrrrr}
\toprule
Type & Metric & Slope & SE & $R^2$ & $C_T$ \\
\midrule
major bank & $E|L_T^{k,N}-L_T^k|^2$ & -1.730 & 0.234 & 0.901 & 3.399 \\
major bank & $E W_2^2(\widehat\mu_T^{k,N},\widehat\mu_T^k)$ & -1.053 & 0.187 & 0.840 & 554.6 \\
regional bank & $E|L_T^{k,N}-L_T^k|^2$ & -1.455 & 0.183 & 0.913 & 3.766 \\
regional bank & $E W_2^2(\widehat\mu_T^{k,N},\widehat\mu_T^k)$ & -0.974 & 0.164 & 0.855 & 470.8 \\
small bank & $E|L_T^{k,N}-L_T^k|^2$ & -1.329 & 0.160 & 0.920 & 3.648 \\
small bank & $E W_2^2(\widehat\mu_T^{k,N},\widehat\mu_T^k)$ & -0.922 & 0.150 & 0.864 & 369.6 \\
systemically important bank & $E|L_T^{k,N}-L_T^k|^2$ & -1.902 & 0.266 & 0.895 & 3.149 \\
systemically important bank & $E W_2^2(\widehat\mu_T^{k,N},\widehat\mu_T^k)$ & -1.123 & 0.188 & 0.856 & 527.4 \\
\bottomrule
\end{tabular}
\end{table}

\section{Two-population HJB feedback and sensitivity diagnostics}
\label{sec:hjb-numerics}

Theorem~\ref{thm:two-pop-optimal-control-hjb-common-noise} gives the HJB characterization and smooth-verification first-order condition for the controlled McKean--Vlasov system with $K=2$ under common noise. For a finite numerical check, we project the state to $(x_1,x_2,L_1,L_2)$, use controls in $[0,\bar a]^2$, solve the resulting finite-difference HJB by backward induction, and validate the stored feedback in a forward $K=2$ core--periphery particle system. The diffusion terms in the two alive-state coordinates are handled by semi-implicit steps, the common-noise cross term is retained explicitly, and the killing term is treated as a one-way semi-Lagrangian jump in the loss coordinates. This construction is a finite-dimensional diagnostic for the Hamiltonian structure; it is not used as an additional theorem.

The main policy comparison is based on direct forward validation. In the representative two-population system, no control and the representative graph-pressure threshold rule both have mean terminal loss about $0.362$, 95\% loss about $0.375$, and cascade probability $1.000$. Type-uniform control reduces the 95\% loss to $0.211$ and cascade probability to $0.455$ at cost $0.0143$. The HJB feedback further reduces the 95\% loss to $0.202$ and cascade probability to $0.095$ at cost $0.0162$. In a larger forward system with $N=100{,}000$ and 200 replications, the stored $N_x=160,N_L=80$ feedback gives mean terminal loss $0.1569$, 95\% loss $0.1638$, cascade probability $0.000$, and mean control cost $0.0087$. This forward evaluation is the primary numerical conclusion: value-level discretization error is driven mainly by loss-coordinate projection and killing-jump interpolation, so raw grid values $V_0$ are treated as discretization diagnostics rather than as stable policy values.

The multi-grid diagnostic records this discretization issue. Along the coupled grid sequence $N_L=N_x/2$, $V_0$ falls from $0.1713$ to $0.0459$, the initial control falls from $(1.625,1.750)$ to $(0.625,0.625)$, and the average control cost falls from $0.0680$ to $0.0087$. The Richardson-type diagnostic gives an exponent $p\simeq0.476$, indicating that loss-coordinate projection and semi-Lagrangian interpolation dominate the four-dimensional feedback error. Asymmetric-refinement and fixed-loss-slice checks support the same interpretation; these grid-level diagnostics are kept outside the main text so that the numerical section focuses on direct forward validation.

Figure~\ref{fig:hjb-policy-heatmaps} shows representative control slices at a fixed loss state. Control is concentrated in high-pressure regions, with the core-type intervention increasing earlier as system pressure rises. This pattern is consistent with the marginal-benefit term in the Hamiltonian: when a type and state region have larger marginal contribution to future loss, the feedback tilts intervention resources toward that region. A separate recalibration of graph-pressure threshold control in a large-scale $K=2$ forward system reaches 95\% loss $0.2025$ and cascade probability $0.205$ at cost $0.0300$, while the HJB feedback reaches 95\% loss $0.2034$ and cascade probability $0.120$ at cost $0.0166$. The comparison is included only to benchmark the low-dimensional feedback; node-level pressure information in sparse graphs is outside the scope of the present paper.

\begin{figure}[!htbp]
    \centering
    \includegraphics[width=0.82\textwidth]{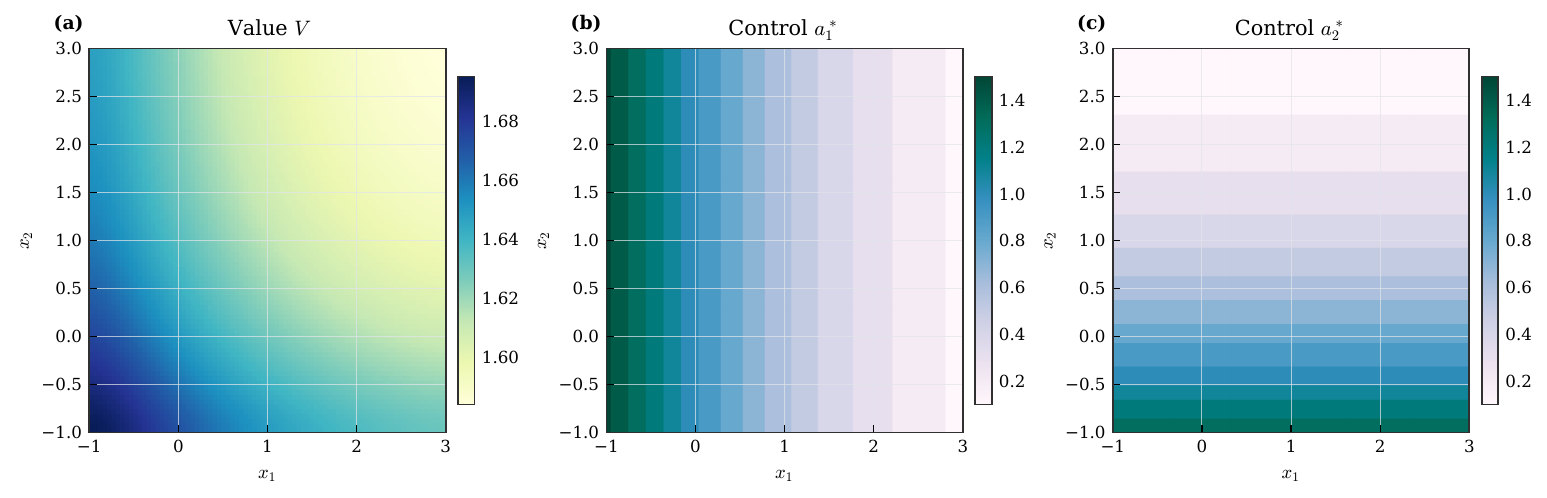}
    \caption{$(x_1,x_2)$ slices of $K=2$ HJB feedback control at a fixed loss state.}
    \label{fig:hjb-policy-heatmaps}
\end{figure}

Killing-parameter sensitivity is summarized at the exposure multiplier used for the high-pressure regime. Increasing the killing-function slope $\gamma$ or the baseline kill intensity $\lambda_{\mathrm{base}}$ raises mean and tail losses until cascade probabilities saturate. At $\gamma=5$, cascade probabilities remain close to zero for $\lambda_{\mathrm{base}}\leq0.05$ and jump to $1.000$ at $0.10$, with 95\% loss rising from $0.101$ to $0.259$. For $\gamma=10$, saturation already occurs once $\lambda_{\mathrm{base}}\geq0.03$, and for $\gamma\geq15$ all reported baseline intensities generate cascade probability $1.000$. Thus the effect of the soft killing mechanism is visible below the cascade threshold and becomes saturated after the system enters the high-pressure region.

\section{Proofs of the main analytical results}
\label{sec:main-analytical-proofs}

\subsection{Mean-field limit, quantitative rate, and propagation of chaos}
\label{sec:mkv-proofs}

\begin{lemma}[Conditional empirical LLN for killed diffusion paths]
\label{lem:conditional-killed-empirical-lln}
Fix a type \(k\). Given the common-noise path \(W^0\) and a common-noise adapted type-level control, the limiting killed particles
 \((\bar X_i^k,\bar\xi_i^k,\bar\tau_i^k)_{i\ge1}\)
are conditionally independent and identically distributed. Assume that the conditional law of \(\bar\tau_i^k\) has no fixed-time atoms on \([0,T]\), which is satisfied for the bounded Cox killing construction used here. Then
 \(\mathbb E\sup_{t\le T} d_{\mathrm{BL}}(\widetilde\mu_t^{k,N},\mu_t^k) \longrightarrow0,\)
and
 \(\mathbb E\sup_{t\le T} |\widetilde L_t^{k,N}-L_t^k| \longrightarrow0 .\)
If the conditional empirical process of the killed paths satisfies the standard quadratic bounded-Lipschitz fluctuation bound, then the same argument gives the corresponding \(N_k^{-1}\) mean-square estimate used in Theorem~\ref{thm:quantitative-rate}.
\end{lemma}

\begin{proof}
Condition on \(\mathcal F_T^{W^0}\). The common-noise adapted control is then a deterministic path, and the remaining randomness comes only from the independent initial states, idiosyncratic Brownian motions, and Poisson thinning measures. Hence the limiting killed paths are i.i.d. on the path space \(D([0,T];\mathsf E)\), with \(\mathsf E=\mathbb R\cup\{\partial\}\). The bounded-Lipschitz unit ball on this Polish path space has a countable determining subclass. The ordinary Glivenko--Cantelli theorem for this countable determining class, followed by monotone approximation of the time supremum over a dense grid and the no-fixed-atom property of the killing times, gives the two uniform convergence statements after removing the conditioning. The quadratic version is the standard empirical-process strengthening of the same conditional LLN and is stated separately in Theorem~\ref{thm:quantitative-rate} because it is the finite-sample input needed for the \(N_k^{-1}\) rate.
\end{proof}

\begin{lemma}[Metric-level quadratic bounded-Lipschitz stability]
\label{lem:metric-level-bl-stability}
Under the synchronous coupling used in Theorem~\ref{thm:multiclass-mkv-limit}, suppose the system has a uniform second-moment bound and the bounded-Lipschitz empirical process of the paired killed particles satisfies the usual separability and quadratic fluctuation estimate on the bounded-Lipschitz unit ball. Then, for the auxiliary empirical surviving measure \(\widetilde\mu_t^{k,N}\),
\[
\begin{aligned}
&\mathbb E\sup_{r\le t}
 d_{\mathrm{BL}}^2(\mu_r^{k,N},\widetilde\mu_r^{k,N}) \\
&\qquad\le
C\int_0^t
\left(
E_k(s)+\sum_{\ell=1}^KD_\ell(s)+\sum_{\ell=1}^KV_\ell(s)
\right)ds
+\frac{C}{N_k},
\end{aligned}
\]
where
\[
E_k(t)=\mathbb E|X_i^{k,N}(t)-\bar X_i^k(t)|^2,
\quad
D_k(t)=\mathbb E d_{\mathrm{BL}}^2(\mu_t^{k,N},\mu_t^k),
\quad
V_k(t)=\mathbb E|L_t^{k,N}-L_t^k|^2.
\]
\end{lemma}

\begin{proof}
For each bounded Lipschitz test function, the synchronous Poisson thinning construction bounds the squared difference of the paired empirical integrals by the squared state-coupling error, the squared loss-coupling error, and the martingale fluctuation of order \(N_k^{-1}\). To pass from a single test function to \(d_{\mathrm{BL}}\), take a countable dense determining subclass of the bounded-Lipschitz unit ball, use the assumed quadratic empirical-process bound on this class, and then pass to the full supremum by separability. This is precisely the metric-level form of the synchronous stability estimate; it is recorded as a lemma so that Theorem~\ref{thm:quantitative-rate} does not rely on an invalid interchange between a pointwise test-function estimate and the bounded-Lipschitz supremum.

\end{proof}

\begin{lemma}[Sufficient finite-entropy verification of the quadratic empirical inputs]
\label{lem:sufficient-quadratic-inputs}
Fix a type \(k\) and work conditionally on \(\mathcal F_T^{W^0}\). Suppose, after localization to a compact alive-state interval and adjoining the cemetery point, that the bounded-Lipschitz test class used to metrize the lifted killed paths admits a countable envelope-normalized subclass \(\mathcal F_k\) with entropy integral
 \(\int_0^1 \sqrt{\log N(\varepsilon,\mathcal F_k,L^2(Q))}\,d\varepsilon <\infty\)
uniformly over all probability laws \(Q\) generated by the conditional limiting killed paths in the model class. Assume also the bounded killing intensity, the Lipschitz state dependence of \(\lambda_k\), and the uniform second-moment bound used in Theorem~\ref{thm:quantitative-rate}. Then the auxiliary empirical measure satisfies
\[
  \mathbb E\sup_{t\le T}d_{\mathrm{BL}}^2(\widetilde\mu_t^{k,N},\mu_t^k)
  \le \frac{C}{N_k},
  \qquad
  \mathbb E\sup_{t\le T}|\widetilde L_t^{k,N}-L_t^k|^2
  \le \frac{C}{N_k},
\]
and the metric-level synchronous-thinning estimate in Lemma~\ref{lem:metric-level-bl-stability} holds with the same order. The constant depends on the entropy integral, the localization radius, \(T\), \(L_\lambda\), \(\Lambda\), and the moment bound, but not on \(N_k\).
\end{lemma}

\begin{proof}
Conditional on \(\mathcal F_T^{W^0}\), the limiting killed paths are i.i.d. random elements of the localized Polish path space. The assumed uniform entropy integral makes the bounded-Lipschitz determining class \(\mathcal F_k\) conditionally Donsker and gives the maximal inequality
\[
  \mathbb E\sup_{f\in\mathcal F_k}
  \left|\frac1{N_k}\sum_{i=1}^{N_k}\bigl(f(\bar X_i^k,\bar\xi_i^k)-\mathbb E[f(\bar X_i^k,\bar\xi_i^k)\mid\mathcal F_T^{W^0}]\bigr)\right|^2
  \le \frac{C}{N_k}.
\]
Applying this estimate to the time-indexed alive-measure and loss-coordinate tests gives the two displayed auxiliary empirical bounds. For the paired finite and limiting systems, write the difference of empirical integrals as the sum of an alive-state displacement term, a cemetery-state mismatch term, and a centered compensator martingale. The displacement term is controlled by the squared state-coupling error. The cemetery mismatch is controlled by synchronous Poisson thinning:
\[
  \mathbb E\sup_{r\le t}|\xi_i^{k,N}(r)-\bar\xi_i^k(r)|^2
  \le C\int_0^t \mathbb E|X_i^{k,N}(s)-\bar X_i^k(s)|^2\,ds .
\]
The centered compensator martingale has quadratic variation bounded by \(C/N_k\), because the intensity is bounded by \(\Lambda\). The same entropy maximal inequality transfers these pointwise estimates to the bounded-Lipschitz supremum. This yields Lemma~\ref{lem:metric-level-bl-stability} under the stated finite-entropy condition.
\end{proof}

\begin{remark}[Concrete localized entropy verification]
\label{rem:localized-entropy-verification}
In the compact alive-state version used for the bounded-killing theory, the finite-entropy condition in Lemma~\ref{lem:sufficient-quadratic-inputs} can be checked from standard bracketing entropy bounds. After adjoining the isolated cemetery point, the lifted killed-path space is the union of a compact one-dimensional diffusion path class with one cemetery jump coordinate. The bounded drift, bounded diffusion coefficient, and bounded killing rate give a uniform modulus-in-probability for the alive diffusion paths, and the cemetery coordinate is a monotone indicator with no fixed-time atom. The bounded-Lipschitz unit ball restricted to this localized class has polynomial bracketing entropy under the laws generated by the model; hence the Dudley--bracketing integral is finite by the entropy integral criterion of van der Vaart and Wellner~\cite[Theorem~2.7.11]{VanDerVaartWellner1996}. Therefore, on compact state grids or after localization before removing the localization error, the quadratic empirical input in Theorem~\ref{thm:quantitative-rate} follows from the stated bounded-Lipschitz and bounded-killing primitives.
\end{remark}

\begin{remark}[Empirical-process input for the conditional rate]
The quadratic empirical inputs in Theorem~\ref{thm:quantitative-rate} are verified by Lemma~\ref{lem:sufficient-quadratic-inputs} under a localized finite-entropy condition, and are otherwise stated explicitly because the killed-particle state space is path-valued and contains a cemetery state. They are standard once one works conditionally on the common-noise path. Conditional on \(\mathcal F_T^{W^0}\), the limiting killed paths are i.i.d. random elements of the Polish space \(D([0,T];\mathsf E)\), with \(\mathsf E=\mathbb R\cup\{\partial\}\). The bounded-Lipschitz unit ball on this space admits a countable determining subclass, and the usual separability and symmetrization arguments for empirical processes give the \(N_k^{-1}\) second-moment fluctuation scale for bounded Lipschitz test functions; see, for example, van der Vaart and Wellner~\cite{VanDerVaartWellner1996}. The additional metric-level stability term in Lemma~\ref{lem:metric-level-bl-stability} is the synchronous Poisson-thinning analogue of the same argument: the alive-state displacement, cemetery-state mismatch, and compensator martingale terms are controlled before taking the bounded-Lipschitz supremum. Thus Theorem~\ref{thm:quantitative-rate} should be read as a conditional quantitative propagation estimate under these standard empirical-process and thinning-stability inputs, not as a pointwise-test-function estimate upgraded without justification.
\end{remark}

\begin{theorem*}[Detailed form of Theorem~\ref{thm:multiclass-mkv-limit}: McKean--Vlasov limit for the multi-class bank particle system]
Fix $T>0$. The system contains $K$ classes of banks. Class $k$ contains $N_k$ particles, and write
 \(N_*=\min_{1\le k\le K}N_k .\)
For $k=1,\dots,K$ and $i=1,\dots,N_k$, the finite-particle system satisfies
\[
dX_i^{k,N}(t)
=
b_k\bigl(X_i^{k,N}(t),\mu_t^{1,N},\dots,\mu_t^{K,N},a_t^k\bigr)\,dt
+
\sigma_k\,dW_i^k(t)
+
\sigma_0\,dW^0(t)
-
\sum_{\ell=1}^K\Gamma_{k\ell}\,dL_t^{\ell,N}.
\]
Here $W^0$ is the common noise, $W_i^k$ is the idiosyncratic noise, and $\Gamma=(\Gamma_{k\ell})_{1\le k,\ell\le K}$ is the contagion matrix. The surviving empirical measure and loss process of class $k$ are defined by
\[
\mu_t^{k,N}
=
\frac1{N_k}\sum_{i=1}^{N_k}
\xi_i^{k,N}(t)\delta_{X_i^{k,N}(t)},
\qquad
L_t^{k,N}
=
\frac1{N_k}\sum_{i=1}^{N_k}
\bigl(1-\xi_i^{k,N}(t)\bigr),
\]
where
 \(\xi_i^{k,N}(t)=\mathbf 1_{\{\tau_i^{k,N}>t\}} .\)
The killing time is given by Poisson thinning. Choose $\Lambda>0$ such that
 \(0\le \lambda_k(x,a)\le \Lambda .\)
Let $\mathcal N_i^k(ds,du)$ be an independent Poisson random measure on $[0,\infty)\times[0,\Lambda]$ with intensity $ds\,du$, and define
\[
\tau_i^{k,N}
=
\inf\left\{
t:
\int_0^t\int_0^\Lambda
\mathbf 1_{\{u\le \lambda_k(X_i^{k,N}(s-),a_s^k)\}}
\mathcal N_i^k(ds,du)
\ge 1
\right\}.
\]

Assume that the initial states are independent and identically distributed within each class and have uniformly finite first moments. Assume that the drift function $b_k$ is Lipschitz in the state variable and measure variables: there exists a constant $L_b>0$ such that
\[
\left|
b_k(x,\nu^1,\dots,\nu^K,a)
-
b_k(y,\eta^1,\dots,\eta^K,a)
\right|
\le
L_b
\left(
|x-y|
+
\sum_{\ell=1}^K d_{\mathrm{BL}}(\nu^\ell,\eta^\ell)
\right),
\]
where $d_{\mathrm{BL}}$ denotes the bounded Lipschitz distance on finite measures,
\[
d_{\mathrm{BL}}(\nu,\eta)
=
\sup_{\|f\|_\infty\le 1,\ \mathrm{Lip}(f)\le 1}
\left|
\int f\,d\nu-\int f\,d\eta
\right|.
\]
Assume also that the kill intensity $\lambda_k$ is Lipschitz in the state variable, namely
 \(|\lambda_k(x,a)-\lambda_k(y,a)| \le L_\lambda |x-y|,\)
and that the contagion matrix $\Gamma$ is bounded. The type-level control $a_t=(a_t^1,\ldots,a_t^K)$ is progressively measurable with respect to the common-noise filtration $\mathbb F^{W^0}$ and takes values in a compact set. This convention is the open-loop mean-field control convention used in the limit theorem; node-level graph-feedback controls are treated later as finite-network rules. Assume that both the finite-particle system and the limiting system below are well posed.

The limiting McKean--Vlasov system is defined as follows. Given the common noise $W^0$, the limiting particle $\bar X_i^k$ of class $k$ satisfies
\[
d\bar X_i^k(t)
=
b_k\bigl(\bar X_i^k(t),\mu_t^1,\dots,\mu_t^K,a_t^k\bigr)\,dt
+
\sigma_k\,dW_i^k(t)
+
\sigma_0\,dW^0(t)
-
\sum_{\ell=1}^K\Gamma_{k\ell}\,dL_t^\ell .
\]
The limiting surviving measure is the conditional sub-probability measure
\[
\mu_t^k
=
\mathbb E\!\left[
\bar\xi_i^k(t)\delta_{\bar X_i^k(t)}
\,\big|\,
\mathcal F_t^{W^0}
\right],
\]
and the limiting loss process is
 \(L_t^k = 1- \mathbb E\!\left[ \bar\xi_i^k(t) \,\big|\, \mathcal F_t^{W^0} \right].\)
Equivalently,
\[
L_t^k
=
\int_0^t
\mathbb E\!\left[
\lambda_k(\bar X_i^k(s),a_s^k)\bar\xi_i^k(s)
\,\big|\,
\mathcal F_s^{W^0}
\right]ds .
\]
Then, as $N_*\to\infty$, for each $k=1,\dots,K$,
\[
\sup_{t\le T}
d_{\mathrm{BL}}(\mu_t^{k,N},\mu_t^k)
\longrightarrow 0
\qquad\text{in probability},
\]
and
 \(\sup_{t\le T} |L_t^{k,N}-L_t^k| \longrightarrow 0 \qquad\text{in probability}.\)
Thus the empirical surviving measures of the finite-particle system converge weakly to the coupled McKean--Vlasov limit, and the loss processes converge uniformly in probability on $[0,T]$.
\end{theorem*}

\begin{proof}
Construct the finite-particle system and the limiting system on the same probability space. For each $i$ and $k$, let $X_i^{k,N}$ and $\bar X_i^k$ use the same initial state, the same idiosyncratic Brownian motion $W_i^k$, the same common Brownian motion $W^0$, and the same Poisson random measure $\mathcal N_i^k$. Define the limiting killing time by
\[
\bar\tau_i^k
=
\inf\left\{
t:
\int_0^t\int_0^\Lambda
\mathbf 1_{\{u\le \lambda_k(\bar X_i^k(s-),a_s^k)\}}
\mathcal N_i^k(ds,du)
\ge 1
\right\},
\]
and write
 \(\bar\xi_i^k(t)=\mathbf 1_{\{\bar\tau_i^k>t\}} .\)
This synchronous Poisson thinning construction ensures that the difference between the killing indicators of the finite particle and the limiting particle comes only from the difference between the two kill-intensity paths
 \(\lambda_k(X_i^{k,N}(t),a_t^k), \qquad \lambda_k(\bar X_i^k(t),a_t^k).\)

Define the coupling error
 \(e_i^k(t)=X_i^{k,N}(t)-\bar X_i^k(t).\)
Since the finite particle and the limiting particle use the same Brownian motions, the diffusion terms cancel. Hence
\[
\begin{aligned}
e_i^k(t)
&=
\int_0^t
\Big[
b_k(X_i^{k,N}(s),\mu_s^{1,N},\dots,\mu_s^{K,N},a_s^k) \\
&\qquad\qquad
-
b_k(\bar X_i^k(s),\mu_s^1,\dots,\mu_s^K,a_s^k)
\Big]ds
-
\sum_{\ell=1}^K
\Gamma_{k\ell}
\bigl(L_t^{\ell,N}-L_t^\ell\bigr).
\end{aligned}
\]
By the Lipschitz property of $b_k$ and the boundedness of $\Gamma$, there exists a constant $C>0$, independent of $\mathbf N=(N_1,\dots,N_K)$, such that
\[
\sup_{r\le t}|e_i^k(r)|
\le
C
\int_0^t
\left(
|e_i^k(s)|
+
\sum_{\ell=1}^K
d_{\mathrm{BL}}(\mu_s^{\ell,N},\mu_s^\ell)
\right)ds
+
C\sum_{\ell=1}^K
\sup_{r\le t}|L_r^{\ell,N}-L_r^\ell|.
\]
Write
\[
u^k(t)
=
\mathbb E\sup_{r\le t}|e_i^k(r)|,
\qquad
m^k(t)
=
\mathbb E\sup_{r\le t}
d_{\mathrm{BL}}(\mu_r^{k,N},\mu_r^k),
\]
and
 \(v^k(t) = \mathbb E\sup_{r\le t}|L_r^{k,N}-L_r^k|.\)
Taking expectations gives
\[
u^k(t)
\le
C
\int_0^t
\left(
u^k(s)+\sum_{\ell=1}^K m^\ell(s)
\right)ds
+
C\sum_{\ell=1}^K v^\ell(t).
\]

We next estimate the empirical measure error. Introduce the auxiliary empirical measure generated by the limiting particles,
 \(\widetilde\mu_t^{k,N} = \frac1{N_k}\sum_{i=1}^{N_k} \bar\xi_i^k(t)\delta_{\bar X_i^k(t)}.\)
By the triangle inequality,
\[
d_{\mathrm{BL}}(\mu_t^{k,N},\mu_t^k)
\le
d_{\mathrm{BL}}(\mu_t^{k,N},\widetilde\mu_t^{k,N})
+
d_{\mathrm{BL}}(\widetilde\mu_t^{k,N},\mu_t^k).
\]
For any test function $f$ with $\|f\|_\infty\le 1$ and $\mathrm{Lip}(f)\le 1$,
\[
\begin{aligned}
\left|
\langle f,\mu_t^{k,N}-\widetilde\mu_t^{k,N}\rangle
\right|
&\le
\frac1{N_k}\sum_{i=1}^{N_k}
\left|
\xi_i^{k,N}(t)f(X_i^{k,N}(t))
-
\bar\xi_i^k(t)f(\bar X_i^k(t))
\right| \\
&\le
\frac1{N_k}\sum_{i=1}^{N_k}
|X_i^{k,N}(t)-\bar X_i^k(t)|
+
\frac1{N_k}\sum_{i=1}^{N_k}
|\xi_i^{k,N}(t)-\bar\xi_i^k(t)|.
\end{aligned}
\]
The second term only uses $\|f\|_\infty\le 1$. Taking the supremum over test functions, then the supremum over $r\le t$, and finally expectation, yields
\[
\mathbb E\sup_{r\le t}
d_{\mathrm{BL}}(\mu_r^{k,N},\widetilde\mu_r^{k,N})
\le
u^k(t)
+
\mathbb E\sup_{r\le t}
\frac1{N_k}\sum_{i=1}^{N_k}
|\xi_i^{k,N}(r)-\bar\xi_i^k(r)|.
\]
Synchronous Poisson thinning gives
\[
\mathbb E\sup_{r\le t}
|\xi_i^{k,N}(r)-\bar\xi_i^k(r)|
\le
\mathbb E\int_0^t
\left|
\lambda_k(X_i^{k,N}(s),a_s^k)
-
\lambda_k(\bar X_i^k(s),a_s^k)
\right|ds .
\]
By the Lipschitz property of $\lambda_k$,
 \(\mathbb E\sup_{r\le t} |\xi_i^{k,N}(r)-\bar\xi_i^k(r)| \le L_\lambda \int_0^t u^k(s)\,ds .\)
Therefore,
\[
\mathbb E\sup_{r\le t}
d_{\mathrm{BL}}(\mu_r^{k,N},\widetilde\mu_r^{k,N})
\le
u^k(t)
+
C\int_0^t u^k(s)\,ds .
\]

On the other hand, conditional on the common noise $W^0$, the limiting particles
 \(\{(\bar X_i^k,\bar\xi_i^k)\}_{i=1}^{N_k}\)
are conditionally independent and identically distributed. Lemma~\ref{lem:conditional-killed-empirical-lln} gives
\[
\alpha_{N_k}^k(T)
:=
\mathbb E\sup_{t\le T}
d_{\mathrm{BL}}(\widetilde\mu_t^{k,N},\mu_t^k)
\longrightarrow 0
\qquad (N_k\to\infty).
\]
Hence
 \(m^k(t) \le u^k(t) + C\int_0^t u^k(s)\,ds + \alpha_{N_k}^k(T).\)

We now estimate the loss process. Define the auxiliary empirical loss process of the limiting particles by
 \(\widetilde L_t^{k,N} = \frac1{N_k}\sum_{i=1}^{N_k} (1-\bar\xi_i^k(t)).\)
Then
 \(|L_t^{k,N}-L_t^k| \le |L_t^{k,N}-\widetilde L_t^{k,N}| + |\widetilde L_t^{k,N}-L_t^k|.\)
The first term is controlled by the synchronous killing coupling:
\[
\begin{aligned}
\mathbb E\sup_{r\le t}
|L_r^{k,N}-\widetilde L_r^{k,N}|
&\le
\mathbb E\sup_{r\le t}
\frac1{N_k}\sum_{i=1}^{N_k}
|\xi_i^{k,N}(r)-\bar\xi_i^k(r)| \\
&\le
C\int_0^t u^k(s)\,ds .
\end{aligned}
\]
The second term is controlled by Lemma~\ref{lem:conditional-killed-empirical-lln}. Given the common noise $W^0$, the killing times
 \(\{\bar\tau_i^k\}_{i=1}^{N_k}\)
are conditionally independent and identically distributed with conditional distribution function $L_t^k$. Hence
\[
\beta_{N_k}^k(T)
:=
\mathbb E\sup_{t\le T}
|\widetilde L_t^{k,N}-L_t^k|
\longrightarrow 0
\qquad (N_k\to\infty).
\]
Thus
 \(v^k(t) \le C\int_0^t u^k(s)\,ds + \beta_{N_k}^k(T).\)

Let
\[
U(t)=\sum_{k=1}^K u^k(t),
\qquad
M(t)=\sum_{k=1}^K m^k(t),
\qquad
V(t)=\sum_{k=1}^K v^k(t).
\]
The estimates above give
 \(U(t) \le C\int_0^t\bigl(U(s)+M(s)\bigr)ds + CV(t),\)
 \(M(t) \le CU(t) + C\int_0^t U(s)\,ds + A_{\mathbf N}(T),\)
and
 \(V(t) \le C\int_0^t U(s)\,ds + B_{\mathbf N}(T),\)
where
\[
A_{\mathbf N}(T)=\sum_{k=1}^K\alpha_{N_k}^k(T),
\qquad
B_{\mathbf N}(T)=\sum_{k=1}^K\beta_{N_k}^k(T).
\]
As $N_*\to\infty$,
 \(A_{\mathbf N}(T)\to 0, \qquad B_{\mathbf N}(T)\to 0.\)
Substituting the estimates for $M$ and $V$ into the estimate for $U$ yields
 \(U(t) \le C\int_0^t U(s)\,ds + C A_{\mathbf N}(T) + C B_{\mathbf N}(T).\)
By Gronwall's inequality,
 \(U(T) \le C_T\bigl(A_{\mathbf N}(T)+B_{\mathbf N}(T)\bigr) \longrightarrow 0.\)
Substituting this bound back gives
 \(M(T)\longrightarrow 0, \qquad V(T)\longrightarrow 0.\)
Therefore, for each $k=1,\dots,K$,
 \(\mathbb E\sup_{t\le T} d_{\mathrm{BL}}(\mu_t^{k,N},\mu_t^k) \longrightarrow 0,\)
and
 \(\mathbb E\sup_{t\le T} |L_t^{k,N}-L_t^k| \longrightarrow 0.\)
Markov's inequality gives
\[
\sup_{t\le T}
d_{\mathrm{BL}}(\mu_t^{k,N},\mu_t^k)
\longrightarrow 0
\qquad\text{in probability},
\]
and
 \(\sup_{t\le T} |L_t^{k,N}-L_t^k| \longrightarrow 0 \qquad\text{in probability}.\)
Since $d_{\mathrm{BL}}$ metrizes weak convergence of finite measures, the empirical surviving measures of the finite-particle system converge weakly to the McKean--Vlasov limit, and the loss processes converge uniformly in probability on $[0,T]$. The proof is complete.
\end{proof}

\begin{theorem*}[Detailed form of Theorem~\ref{thm:quantitative-rate}: Conditional quantitative convergence rate under strengthened conditions]
Under the conditions of Theorem~\ref{thm:multiclass-mkv-limit}, assume also that the system has a uniform second-moment bound and that the quadratic versions of the two empirical inputs in Lemmas~\ref{lem:conditional-killed-empirical-lln} and~\ref{lem:metric-level-bl-stability} hold. Equivalently, the auxiliary empirical surviving measure
 \(\widetilde\mu_t^{k,N} = \frac1{N_k}\sum_{i=1}^{N_k}\bar\xi_i^k(t)\delta_{\bar X_i^k(t)}\)
satisfies
\[
\sup_{t\in[0,T]}
\mathbb E d_{\mathrm{BL}}^2(\widetilde\mu_t^{k,N},\mu_t^k)
\le \frac{C_{\mathrm{emp}}(T)}{N_k},
\qquad k=1,\ldots,K,
\]
and the paired finite and limiting auxiliary systems satisfy the metric-level stability estimate of Lemma~\ref{lem:metric-level-bl-stability}. This formulation records the empirical-process step explicitly: the rate is stated in the bounded-Lipschitz metric, and the supremum over bounded-Lipschitz test functions is handled by a separability/quadratic-fluctuation input rather than by a pointwise test-function estimate.

Then there exists a constant
 \(C_T=C\bigl(T,L_b,L_\lambda,\|\Gamma\|,\Lambda,C_{\mathrm{emp}}(T)\bigr),\)
such that, for each $k=1,\dots,K$,
 \(\sup_{t\in[0,T]} \mathbb E d_{\mathrm{BL}}^2(\mu_t^{k,N},\mu_t^k) \le \frac{C_T}{N_*},\)
and
 \(\sup_{t\in[0,T]} \mathbb E|L_t^{k,N}-L_t^k|^2 \le \frac{C_T}{N_*}.\)
Thus the first-order bounded-Lipschitz empirical-measure error and the first-order loss-process error are both $O(N_*^{-1/2})$, uniformly over $t\in[0,T]$. The constant has Gronwall-type growth and can be written as
 \(C_T\le C_0\exp\{C_1T(1+\|\Gamma\|)^2\},\)
where $C_0,C_1$ do not depend on $\mathbf N=(N_1,\dots,N_K)$.
\end{theorem*}

\begin{proof}
Use the synchronous coupling from the proof of Theorem~\ref{thm:multiclass-mkv-limit}. For each $i,k$, the finite particle $X_i^{k,N}$ and the limiting particle $\bar X_i^k$ use the same initial state, the same idiosyncratic Brownian motion, the same common Brownian motion, and the same Poisson random measure. Let
 \(e_i^k(t)=X_i^{k,N}(t)-\bar X_i^k(t).\)
Define
\[
E_k(t):=\mathbb E|e_i^k(t)|^2,
\qquad
D_k(t):=\mathbb E d_{\mathrm{BL}}^2(\mu_t^{k,N},\mu_t^k),
\qquad
V_k(t):=\mathbb E|L_t^{k,N}-L_t^k|^2 .
\]
Since the Brownian terms are fully synchronized,
\[
\begin{aligned}
e_i^k(t)
&=
\int_0^t
\Big[
 b_k(X_i^{k,N}(s),\boldsymbol\mu_s^N,a_s^k)
 -b_k(\bar X_i^k(s),\boldsymbol\mu_s,a_s^k)
\Big]ds \\
&\qquad-
\sum_{\ell=1}^K\Gamma_{k\ell}
\bigl(L_t^{\ell,N}-L_t^\ell\bigr),
\end{aligned}
\]
where $\boldsymbol\mu_s^N=(\mu_s^{1,N},\ldots,\mu_s^{K,N})$ and $\boldsymbol\mu_s=(\mu_s^1,\ldots,\mu_s^K)$. By the Lipschitz property of $b_k$, the boundedness of $\Gamma$, and the definition of $d_{\mathrm{BL}}$,
\[
E_k(t)
\le
C\int_0^t\left(E_k(s)+\sum_{\ell=1}^KD_\ell(s)\right)ds
+C\sum_{\ell=1}^KV_\ell(t).
\]

We next estimate the surviving-measure error. By the triangle inequality,
\[
d_{\mathrm{BL}}(\mu_t^{k,N},\mu_t^k)
\le
d_{\mathrm{BL}}(\mu_t^{k,N},\widetilde\mu_t^{k,N})
+d_{\mathrm{BL}}(\widetilde\mu_t^{k,N},\mu_t^k).
\]
The assumed quadratic synchronous-thinning stability gives
\[
\mathbb E d_{\mathrm{BL}}^2(\mu_t^{k,N},\widetilde\mu_t^{k,N})
\le
C\int_0^t\left(E_k(s)+\sum_{\ell=1}^KD_\ell(s)+\sum_{\ell=1}^KV_\ell(s)\right)ds+\frac{C}{N_k}.
\]
The auxiliary empirical fluctuation estimate gives
\[
\mathbb E d_{\mathrm{BL}}^2(\widetilde\mu_t^{k,N},\mu_t^k)
\le
\frac{C_{\mathrm{emp}}(T)}{N_k}.
\]
Therefore,
\[
D_k(t)
\le
C\int_0^t\left(E_k(s)+\sum_{\ell=1}^KD_\ell(s)+\sum_{\ell=1}^KV_\ell(s)\right)ds+\frac{C}{N_k}.
\]

For the loss process, the finite empirical loss has the Doob--Meyer decomposition
 \(L_t^{k,N}=M_t^{k,N}+ \int_0^t\langle\lambda_k(\cdot,a_s^k),\mu_s^{k,N}\rangle ds,\)
where $M^{k,N}$ is a martingale with
 \(\mathbb E|M_t^{k,N}|^2\le \frac{C\Lambda T}{N_k}.\)
The limiting loss process satisfies
 \(L_t^k=\int_0^t\langle\lambda_k(\cdot,a_s^k),\mu_s^k\rangle ds.\)
Since $\lambda_k$ is bounded and Lipschitz, the pairing with $\lambda_k$ is controlled by the bounded-Lipschitz distance after rescaling the test function by a deterministic constant. Hence
\[
\left|
\langle\lambda_k(\cdot,a_s^k),\mu_s^{k,N}\rangle
-
\langle\lambda_k(\cdot,a_s^k),\mu_s^k\rangle
\right|
\le C d_{\mathrm{BL}}(\mu_s^{k,N},\mu_s^k).
\]
By Cauchy's inequality,
 \(V_k(t) \le \frac{C}{N_k}+C\int_0^tD_k(s)ds.\)

Let
 \(E(t)=\sum_{k=1}^K E_k(t),\qquad D(t)=\sum_{k=1}^K D_k(t),\qquad V(t)=\sum_{k=1}^K V_k(t).\)
The estimates above imply
 \(E(t)\le C\int_0^t(E(s)+D(s))ds+CV(t),\)
 \(D(t)\le \frac{C}{N_*}+C\int_0^t(E(s)+D(s)+V(s))ds,\)
and
 \(V(t)\le \frac{C}{N_*}+C\int_0^tD(s)ds.\)
Substituting the last estimate into the first and absorbing repeated time integrals into the constant gives
 \(E(t)\le \frac{C}{N_*}+C\int_0^t(E(s)+D(s))ds.\)
Together with the estimates for $D$ and $V$, this yields
 \(\Phi(t):=E(t)+D(t)+V(t) \le \frac{C}{N_*}+C\int_0^t\Phi(s)ds.\)
Gronwall's inequality gives
 \(\Phi(t)\le \frac{C_T}{N_*},\qquad 0\le t\le T.\)
The asserted bounds for each type follow from $D_k\le D$ and $V_k\le V$. The proof is complete.
\end{proof}

\begin{proposition*}[Detailed form of Proposition~\ref{prop:propagation-of-chaos-same-type}: Propagation of chaos: asymptotic independence of particles of the same type]
Under the conditions of Theorem~\ref{thm:multiclass-mkv-limit}, fix a type
\(k\in\{1,\ldots,K\}\) and a terminal time \(t\in[0,T]\). Let
\[
  \widehat X_i^{k,N}(t)
  :=
  \begin{cases}
  X_i^{k,N}(t), & \xi_i^{k,N}(t)=1,\\
  \partial, & \xi_i^{k,N}(t)=0,
  \end{cases}
\]
where \(\partial\) is the cemetery point, and define the full probability measures
\[
  \widehat\mu_t^{k,N}
  :=
  \mu_t^{k,N}+L_t^{k,N}\delta_\partial,
  \qquad
  \widehat\mu_t^k
  :=
  \mu_t^k+L_t^k\delta_\partial .
\]
Then, for any fixed integer \(m\ge1\) and any pairwise distinct indices
 \(i_1,\ldots,i_m\in\{i:g(i)=k\},\)
we have
\[
  \mathcal L
  \left(
  \widehat X_{i_1}^{k,N}(t),
  \ldots,
  \widehat X_{i_m}^{k,N}(t)
  \,\big|\,
  \mathcal F_t^{W^0}
  \right)
  \Longrightarrow
  \bigl(\widehat\mu_t^k\bigr)^{\otimes m}
\]
weakly in probability. Equivalently, for any bounded Lipschitz test function
\(\Phi:\mathsf E^m\to\mathbb R\), where
\(\mathsf E=\mathbb R\cup\{\partial\}\),
\[
\begin{aligned}
&\mathbb E
\left[
  \Phi
  \left(
  \widehat X_{i_1}^{k,N}(t),
  \ldots,
  \widehat X_{i_m}^{k,N}(t)
  \right)
  \,\big|\,
  \mathcal F_t^{W^0}
\right]
\\
&\hspace{2cm}\longrightarrow
\int_{\mathsf E^m}
  \Phi(y_1,\ldots,y_m)
  \bigl(\widehat\mu_t^k\bigr)^{\otimes m}
  (dy_1\cdots dy_m)
\end{aligned}
\]
in probability.

In particular, if \(\sigma_0=0\), then \(\widehat\mu_t^k\) is deterministic, and the conditional convergence above reduces to the usual unconditional propagation of chaos:
\[
  \mathcal L
  \left(
  \widehat X_{i_1}^{k,N}(t),
  \ldots,
  \widehat X_{i_m}^{k,N}(t)
  \right)
  \Longrightarrow
  \bigl(\widehat\mu_t^k\bigr)^{\otimes m}.
\]
If the main text already incorporates the dead-particle state into \(X_i^k(t)\), then one may directly write
\(\widehat X_i^{k,N}\) and \(\widehat\mu_t^k\) as
\(X_i^{k,N}\) and \(\mu_t^k\), respectively, and obtain the shortened form
\[
  \mathcal L
  \left(
  X_{i_1}^{k,N}(t),
  \ldots,
  X_{i_m}^{k,N}(t)
  \,\big|\,
  \mathcal F_t^{W^0}
  \right)
  \Longrightarrow
  \bigl(\mu_t^k\bigr)^{\otimes m}.
\]
\end{proposition*}

\begin{proof}
Use the synchronous coupling from the proof of Theorem~\ref{thm:multiclass-mkv-limit}. For each \(i\) and \(k\), the finite particle
\(X_i^{k,N}\) and the limiting particle \(\bar X_i^k\) use the same initial state, the same idiosyncratic Brownian motion
\(W_i^k\), the same common Brownian motion \(W^0\), and the same Poisson random measure
\(\mathcal N_i^k\). Let
\[
  \widehat{\bar X}_i^k(t)
  :=
  \begin{cases}
  \bar X_i^k(t), & \bar\xi_i^k(t)=1,\\
  \partial, & \bar\xi_i^k(t)=0 .
  \end{cases}
\]
By the definition of the limiting McKean--Vlasov system, conditional on the common noise \(\mathcal F_t^{W^0}\),
 \(\widehat{\bar X}_{i_1}^k(t), \ldots, \widehat{\bar X}_{i_m}^k(t)\)
are conditionally independent and identically distributed, with common conditional law
 \(\widehat\mu_t^k = \mu_t^k+L_t^k\delta_\partial .\)
Therefore,
\[
  \mathcal L
  \left(
  \widehat{\bar X}_{i_1}^k(t),
  \ldots,
  \widehat{\bar X}_{i_m}^k(t)
  \,\big|\,
  \mathcal F_t^{W^0}
  \right)
  =
  \bigl(\widehat\mu_t^k\bigr)^{\otimes m}.
\]

It remains to show that the finite-particle tuple and the independent limiting tuple converge to each other in distance under the synchronous coupling. On
\(\mathsf E=\mathbb R\cup\{\partial\}\), choose a bounded metric \(d_\partial\) that agrees with the truncated Euclidean distance on
\(\mathbb R\) and satisfies
 \(d_\partial(x,y)\le |x-y|\wedge1, \qquad x,y\in\mathbb R,\)
and
 \(d_\partial(x,\partial)\le 1, \qquad x\in\mathbb R .\)
Then, for each \(i\),
\[
  d_\partial
  \left(
  \widehat X_i^{k,N}(t),
  \widehat{\bar X}_i^k(t)
  \right)
  \le
  |X_i^{k,N}(t)-\bar X_i^k(t)|\wedge1
  +
  |\xi_i^{k,N}(t)-\bar\xi_i^k(t)|.
\]
The synchronous coupling estimates in Theorem~\ref{thm:multiclass-mkv-limit} give
 \(\mathbb E \sup_{0\le s\le T} |X_i^{k,N}(s)-\bar X_i^k(s)| \longrightarrow 0\)
and
 \(\mathbb E \sup_{0\le s\le T} |\xi_i^{k,N}(s)-\bar\xi_i^k(s)| \longrightarrow 0 .\)
Hence
\[
  \mathbb E
  d_\partial
  \left(
  \widehat X_i^{k,N}(t),
  \widehat{\bar X}_i^k(t)
  \right)
  \longrightarrow 0 .
\]
For fixed \(i_1,\ldots,i_m\), this gives
\[
  \mathbb E
  \sum_{r=1}^m
  d_\partial
  \left(
  \widehat X_{i_r}^{k,N}(t),
  \widehat{\bar X}_{i_r}^k(t)
  \right)
  \longrightarrow 0 .
\]

Let \(\Phi:\mathsf E^m\to\mathbb R\) be any bounded Lipschitz function. By the Lipschitz property,
\[
\begin{aligned}
&\left|
\Phi
\left(
  \widehat X_{i_1}^{k,N}(t),
  \ldots,
  \widehat X_{i_m}^{k,N}(t)
\right)
-
\Phi
\left(
  \widehat{\bar X}_{i_1}^k(t),
  \ldots,
  \widehat{\bar X}_{i_m}^k(t)
\right)
\right|
\\
&\hspace{2cm}\le
\mathrm{Lip}(\Phi)
\sum_{r=1}^m
d_\partial
\left(
  \widehat X_{i_r}^{k,N}(t),
  \widehat{\bar X}_{i_r}^k(t)
\right).
\end{aligned}
\]
Taking conditional expectation with respect to \(\mathcal F_t^{W^0}\), and writing
\[
  \mathcal Q_{N,t}^{k,m}
  :=
  \mathcal L
  \left(
  \widehat X_{i_1}^{k,N}(t),
  \ldots,
  \widehat X_{i_m}^{k,N}(t)
  \,\big|\,
  \mathcal F_t^{W^0}
  \right),
\]
we obtain
\[
\begin{aligned}
&\left|
\int_{\mathsf E^m}\Phi\,d\mathcal Q_{N,t}^{k,m}
-
\int_{\mathsf E^m}\Phi\,d(\widehat\mu_t^k)^{\otimes m}
\right|
\\
&\hspace{2cm}\le
\mathrm{Lip}(\Phi)\,
\mathbb E
\left[
  \sum_{r=1}^m
  d_\partial
  \left(
  \widehat X_{i_r}^{k,N}(t),
  \widehat{\bar X}_{i_r}^k(t)
  \right)
  \,\big|\,
  \mathcal F_t^{W^0}
\right].
\end{aligned}
\]
Taking outer expectation gives
\[
\mathbb E
\left|
\int_{\mathsf E^m}\Phi\,d\mathcal Q_{N,t}^{k,m}
-
\int_{\mathsf E^m}\Phi\,d(\widehat\mu_t^k)^{\otimes m}
\right|
\le
\mathrm{Lip}(\Phi)
\sum_{r=1}^m
\mathbb E
d_\partial
\left(
  \widehat X_{i_r}^{k,N}(t),
  \widehat{\bar X}_{i_r}^k(t)
\right)
\longrightarrow 0 .
\]
Thus the difference above converges to zero in \(L^1\), and hence in probability, for every bounded Lipschitz \(\Phi\). To phrase the conclusion as bounded-Lipschitz convergence, take a countable determining subclass of the bounded-Lipschitz unit ball on the Polish space \(\mathsf E^m\). The preceding estimate holds simultaneously on this countable subclass after a diagonal argument, and separability of the bounded-Lipschitz metric extends it to the full unit ball. Hence
\[
  d_{\mathrm{BL}}
  \left(
  \mathcal Q_{N,t}^{k,m},
  \bigl(\widehat\mu_t^k\bigr)^{\otimes m}
  \right)
  \longrightarrow 0
  \qquad\text{in probability}.
\]
Since the bounded Lipschitz distance metrizes weak convergence on the Polish space \(\mathsf E^m\), the proposition follows.
\end{proof}

\subsection{Controlled well-posedness and killed-HJB comparison}
\label{sec:control-proofs}
\begin{assumption}[Regularity conditions for the two-population controlled contagion system]
\label{ass:two-pop-control-regularity}
Let \(T>0\). The two type shares \(\pi_1,\pi_2\) are fixed strictly positive constants. Set
 \(\mathsf E=\mathbb R\cup\{\partial\},\)
where \(\partial\) denotes the cemetery state. Equip \(\mathsf E\) with a metric \(d_{\partial}\) that agrees with the Euclidean distance on \(\mathbb R\) and satisfies
 \(d_{\partial}(x,\partial)^2\le C_{\partial}(1+|x|^2), \qquad x\in\mathbb R .\)
For \(k=1,2\), let
 \(\mu^k\in\mathcal P_2(\mathsf E), \qquad L^k(\mu^k):=\mu^k(\{\partial\}),\)
and write
 \(\langle f,\mu^k\rangle_{\circ} := \int_{\mathbb R} f(x)\,\mu^k(dx),\)
so that the integral is taken only over the alive-state set \(\mathbb R\).

Assume that the drift function
 \(b_k:\mathbb R\times\mathcal P_2(\mathsf E)^2\to\mathbb R\)
is uniformly Lipschitz in \((x,\mu^1,\mu^2)\) and satisfies a linear growth condition. Assume that the diffusion coefficients
 \(\sigma_k>0,\qquad \sigma_0\ge0\)
are constants, where \(\sigma_0dW^0_t\) denotes the common-noise term. The strict positivity of the idiosyncratic coefficients is the nondegeneracy condition used in the comparison theorem for the killed HJB below. The contagion matrix
 \(\Gamma=(\Gamma_{k\ell})_{1\le k,\ell\le2}\)
is nonnegative and bounded.

The kill intensity
 \(\lambda_k:\mathbb R\times[0,\infty)\to[0,\infty)\)
is continuous, and there exists a constant \(\Lambda<\infty\) such that
 \(0\le \lambda_k(x,a)\le\Lambda, \qquad x\in\mathbb R,\ a\ge0.\)
Moreover, \(\lambda_k\) is uniformly Lipschitz in \(x\) and locally Lipschitz in the control variable \(a\). Extend the kill intensity to the cemetery state by setting
 \(\lambda_k(\partial,a)=0.\)

The terminal loss function
 \(\Phi:[0,1]^2\to\mathbb R\)
is bounded and Lipschitz. The control cost function
 \(c_k:[0,\infty)\to[0,\infty)\)
is continuous, convex, lower semicontinuous, and satisfies the superlinear growth condition
 \(\lim_{a\to\infty}\frac{c_k(a)}{a}=+\infty.\)

In addition, assume that the controlled martingale problem is well posed for every progressively measurable compact-valued control. The dynamic-programming and viscosity framework follows the standard infinite-dimensional control approach of Fabbri--Gozzi--Swiech~\cite{FabbriGozziSwiech2017}. The comparison-principle argument for the present killed Wasserstein HJB is formulated after the decomposition \(\mu^k=\nu^k+L^k\delta_\partial\), with alive finite measures and cemetery masses treated as separate coordinates. It combines the Wasserstein smooth-gauge method of Cosso--Gozzi--Kharroubi--Pham--Rosestolato~\cite{CossoGozziKharroubiPhamRosestolato2024} for the alive-measure variables with the paper-specific finite-dimensional estimates for the cemetery-state killing jump.

For the existence statement, admissible controls are taken in the relaxed sense. A relaxed control is a progressively measurable probability-measure-valued process
 \(\vartheta_t\in\mathcal P([0,\bar a]^2),\)
and strict controls correspond to Dirac relaxed controls \(\vartheta_t=\delta_{a_t}\). Coefficients and costs under \(\vartheta_t\) are obtained by averaging the corresponding strict-control coefficients over \([0,\bar a]^2\). Because the action set is compact and the coefficients are continuous in the action, this relaxed admissible class is compact under weak convergence.
\end{assumption}

\subsection{Comparison principle for the killed two-population HJB}
\label{app:killed-two-pop-comparison-proof}

This subsection proves the comparison-principle interface for the killed two-population HJB. The proof is written on the decomposed state space
\[
  \mathcal D_2
  :=\left\{(\nu^1,\nu^2,L^1,L^2):
  \nu^k\in\mathcal M_{\le 1,2}(\mathbb R),\ L^k\in[0,1],\ |\nu^k|+L^k=1\right\},
\]
where \(\mathcal M_{\le 1,2}(\mathbb R)\) denotes finite sub-probability measures on \(\mathbb R\) with finite second moment. The Wasserstein smooth-gauge/Ishii template is used for the alive-measure coordinates \(\nu^k\), and the cemetery masses \(L^k\) are handled by the finite-dimensional doubled-variable penalty \(\beta|L-M|^2/2\). The only term outside the standard alive-measure template is the mass transfer from \(\nu^k\) to \(L^k\); Lemma~\ref{lem:killing-absorption-doubling} verifies that this killing jump is a bounded Lipschitz first-order perturbation compatible with the doubled-variable estimate.

Throughout this subsection the control set is the compact interval \([0,\bar a]^2\). Let
 \(\mathsf E=\mathbb R\cup\{\partial\}\)
be endowed with an auxiliary metric \(d_\partial\) that agrees with the Euclidean metric on compact alive-state sets and separates the cemetery point from \(\mathbb R\). This metric is used only to identify \((\nu^k,L^k)\) with the killed probability measure \(\nu^k+L^k\delta_\partial\) and to define localization estimates for unmatched alive/cemetery mass; the comparison argument itself is carried out on \(\mathcal D_2\). On each localization interval \([-R,R]\) there is \(c_R>0\) such that \(d_\partial(x,\partial)\ge c_R\) for \(x\in[-R,R]\). The localization is removed by the usual coercive second-moment penalty in the Wasserstein comparison argument. Equivalently, one may work directly with a metric such as \(d_\partial(x,y)=|x-y|\) on \(\mathbb R\) and \(d_\partial(x,\partial)=1+|x|\). We extend the kill intensity by \(\lambda_k(\partial,a)=0\). The assumptions used below are
\[
  0\le \lambda_k(x,a)\le \Lambda,
  \qquad
  |\lambda_k(x,a)-\lambda_k(y,a)|\le L_x |x-y|,
  \qquad
  |\lambda_k(x,a)-\lambda_k(x,\tilde a)|\le L_a |a-\tilde a|,
\]
for \(x,y\in\mathbb R\) and \(a,\tilde a\in[0,\bar a]\). After the extension \(\lambda_k(\partial,a)=0\), the localized Lipschitz constant of \(\lambda_k(\cdot,a)\) on \([-R,R]\cup\{\partial\}\) is bounded by \(C_R(L_x+\Lambda)\), because the alive--cemetery jump satisfies \(|\lambda_k(x,a)-\lambda_k(\partial,a)|\le\Lambda\) and \(d_\partial(x,\partial)\ge c_R\) on \([-R,R]\). The drift is Lipschitz in the alive state, in the measure arguments, and in the loss variables, with at most linear growth. The diffusion coefficients are constant, and in the comparison statement below the alive idiosyncratic coefficients satisfy \(\sigma_k>0\), \(k=1,2\). The terminal cost is bounded and uniformly continuous, and the running control costs are continuous on the compact action set.

For \(\nu\in\mathcal M_{\le1}(\mathbb R)\), define the augmented probability measure
 \(\bar\nu:=\nu+(1-|\nu|)\delta_\partial, \qquad |\nu|:=\nu(\mathbb R),\)
and define
 \(\mathsf W_p(\nu,\omega):=W_p(\bar\nu,\bar\omega;d_\partial), \qquad p=1,2 .\)
On each localization interval, the indicator of the alive set is Lipschitz after extension to \(\partial\), hence
 \(\big||\nu|-|\omega|\big| \le C_R\mathsf W_1(\nu,\omega) \le C_R\mathsf W_2(\nu,\omega).\)
All constants below are uniform on fixed localizations; the coercive penalty lets \(R\to\infty\) at the end.

\begin{lemma}[Bounded Lipschitz control of the killing sink]
\label{lem:killing-bounded-lip-sink}
Let \(p:\mathbb R\to\mathbb R\) be bounded and Lipschitz, with \(\|p\|_\infty\le B\) and \(\operatorname{Lip}(p)\le L_p\). Extend \(p\) to \(\mathsf E\) by setting \(p(\partial)=0\). Then, for every \(\nu,\omega\in\mathcal M_{\le1}(\mathbb R)\) and every \(a,\tilde a\in[0,\bar a]\),
\[
\begin{aligned}
&\left|
\int_{\mathbb R}\lambda_k(x,a)p(x)\,\nu(dx)
-
\int_{\mathbb R}\lambda_k(y,\tilde a)p(y)\,\omega(dy)
\right| \\
&\qquad\le
C_R(\Lambda L_p+L_xB+\Lambda B)\mathsf W_1(\nu,\omega)
+L_aB|a-\tilde a| .
\end{aligned}
\]
Equivalently, the cemetery contribution may be written separately as
\[
\begin{aligned}
&\left|
\int_{\mathbb R}\lambda_k(x,a)p(x)\,\nu(dx)
-
\int_{\mathbb R}\lambda_k(y,\tilde a)p(y)\,\omega(dy)
\right| \\
&\qquad\le
C_R(\Lambda L_p+L_xB)\mathsf W_1(\nu,\omega)
+C_R\Lambda B\big||\nu|-|\omega|\big|
+L_aB|a-\tilde a| .
\end{aligned}
\]
\end{lemma}

\begin{proof}
Set \(f_a(x)=\lambda_k(x,a)p(x)\) for \(x\in\mathbb R\) and \(f_a(\partial)=0\). On \(\mathsf E\),
 \(\|f_a\|_\infty\le \Lambda B.\)
For \(x,y\in\mathbb R\),
\[
\begin{aligned}
|f_a(x)-f_a(y)|
&\le \lambda_k(x,a)|p(x)-p(y)|
  + |p(y)|\,|\lambda_k(x,a)-\lambda_k(y,a)| \\
&\le (\Lambda L_p+L_xB)|x-y| .
\end{aligned}
\]
For \(x\in[-R,R]\), the jump from \(x\) to \(\partial\) is bounded by \(\Lambda B\), and the cemetery point is separated from the alive set; this gives
 \(\operatorname{Lip}_{d_\partial}(f_a) \le C_R(\Lambda L_p+L_xB+\Lambda B).\)
Therefore the Kantorovich--Rubinstein duality gives
\[
  \left|\int f_a\,d\bar\nu-\int f_a\,d\bar\omega\right|
  \le C_R(\Lambda L_p+L_xB+\Lambda B)\mathsf W_1(\nu,\omega).
\]
Since \(f_a(\partial)=0\), these are exactly the alive-state integrals. The control difference is estimated by
\[
  \left|\int_{\mathbb R} [\lambda_k(y,a)-\lambda_k(y,\tilde a)]p(y)\,\omega(dy)\right|
  \le L_aB|a-\tilde a|.
\]
Combining the last two inequalities proves the first bound. The second displayed bound is the same estimate with the cemetery jump contribution separated into the unmatched-mass term.
\end{proof}

\begin{lemma}[Aggregate kill rate is Wasserstein-Lipschitz]
\label{lem:aggregate-kill-rate-lip}
Let
\[
  \Lambda_k(\nu,a):=\int_{\mathbb R}\lambda_k(x,a)\,\nu(dx),
  \qquad \nu\in\mathcal M_{\le1}(\mathbb R),\quad a\in[0,\bar a].
\]
Then
\[
  |\Lambda_k(\nu,a)-\Lambda_k(\omega,\tilde a)|
  \le C_R(\Lambda+L_x)\mathsf W_1(\nu,\omega)
  +L_a|a-\tilde a| .
\]
In particular, for the same control \(a\),
 \(|\Lambda_k(\nu,a)-\Lambda_k(\omega,a)| \le C_R(\Lambda+L_x)\mathsf W_2(\nu,\omega).\)
\end{lemma}

\begin{proof}
Apply Lemma~\ref{lem:killing-bounded-lip-sink} with \(p\equiv1\) on \(\mathbb R\), extended by \(p(\partial)=0\). Then \(B=1\), \(L_p=0\), and the result follows from \(\mathsf W_1\le\mathsf W_2\) for the augmented probability measures.
\end{proof}

\begin{lemma}[State decomposition and reformulated killed HJB]
\label{lem:state-decomposition-killed-hjb}
Every \(\mu^k\in\mathcal P_2(\mathsf E)\) has the unique decomposition
\[
  \mu^k=\nu^k+L^k\delta_\partial,
  \qquad
  \nu^k:=\mu^k|_{\mathbb R}\in\mathcal M_{\le1}(\mathbb R),
  \qquad
  L^k:=\mu^k(\{\partial\})\in[0,1],
\]
with \(|\nu^k|+L^k=1\). Let
 \(v(t,\nu^1,\nu^2,L^1,L^2) :=V(t,\nu^1+L^1\delta_\partial,\nu^2+L^2\delta_\partial).\)
At smooth points, after taking a local extension of \(v\) to a neighborhood in the positive cone of finite alive measures and cemetery masses,
\[
  \frac{\delta V}{\delta\mu^k}(x)
  =
  \frac{\delta v}{\delta\nu^k}(x),
  \qquad x\in\mathbb R,
  \qquad
  \frac{\delta V}{\delta\mu^k}(\partial)
  =
  \partial_{L^k}v .
\]
The difference
 \(\partial_{L^k}v-\frac{\delta v}{\delta\nu^k}(x)\)
is independent of the additive normalization of the linear functional derivative. Consequently the killing term becomes
\[
\begin{aligned}
&\int_{\mathbb R}\lambda_k(x,a^k)
\left[
  \frac{\delta V}{\delta\mu^k}(\partial)
  -
  \frac{\delta V}{\delta\mu^k}(x)
\right]\mu^k(dx) \\
&\qquad=
\Lambda_k(\nu^k,a^k)\partial_{L^k}v
-
\int_{\mathbb R}\lambda_k(x,a^k)
  \frac{\delta v}{\delta\nu^k}(x)\,\nu^k(dx).
\end{aligned}
\]
\end{lemma}

\begin{proof}
The decomposition follows from the disjoint union \(\mathsf E=\mathbb R\cup\{\partial\}\). A signed perturbation supported on \(\mathbb R\) changes only the alive finite measure \(\nu^k\), and therefore gives the identity of the alive functional derivatives. A perturbation in the cemetery atom changes only the scalar variable \(L^k\), and therefore gives the identity at \(\partial\). Since linear functional derivatives on probability measures are defined up to additive constants, the invariant object is the difference between the cemetery derivative and the alive derivative. This is exactly the first variation along the jump direction \(-\delta_x+\delta_\partial\), and substituting it into the original generator gives the displayed decomposition.
\end{proof}

\begin{lemma}[Absorption of the killing term in the doubled-variable argument]
\label{lem:killing-absorption-doubling}
Let \(u\) be a bounded uniformly continuous viscosity subsolution and \(w\) a bounded uniformly continuous viscosity supersolution of the decomposed killed HJB on the decomposed state space. Let
\[
  (t_\varepsilon,s_\varepsilon,
  \nu_\varepsilon,\omega_\varepsilon,
  L_\varepsilon,M_\varepsilon)
\]
be a maximum point of the doubled functional
\[
\begin{aligned}
&u(t,\nu,L)-w(s,\omega,M)
-\frac{|t-s|^2}{2\alpha}
-\sum_{k=1}^2\Theta_\varepsilon^k(\nu^k,\omega^k)
-\frac{\beta}{2}|L-M|^2
-\delta\mathcal R(\nu,\omega),
\end{aligned}
\]
where \(\Theta_\varepsilon^k\) is the Wasserstein smooth gauge used in the comparison proof and \(\mathcal R\) is a coercive second-moment penalty. Let \(p_\varepsilon^k\) and \(q_\varepsilon^k\) be the first-order measure jets generated by the gauge and the coercive penalty in the \(\nu^k\)- and \(\omega^k\)-variables, with signs corresponding to the subsolution and supersolution tests. If the same control \(a\in[0,\bar a]^2\) is used on the two sides, then the killing contribution in the subsolution-minus-supersolution Hamiltonian difference satisfies, for fixed \(\beta\),
\[
  \lim_{\delta\downarrow0}\limsup_{\varepsilon\downarrow0}
  |I_{\mathrm{kill}}(\varepsilon,\alpha,\beta,\delta)|
  \le
  C\limsup_{\varepsilon\downarrow0}
  \sum_{k=1}^2\beta |L^k_\varepsilon-M^k_\varepsilon|^2 .
\]
Consequently,
\[
  \lim_{\beta\to\infty}\lim_{\delta\downarrow0}\limsup_{\varepsilon\downarrow0}
  |I_{\mathrm{kill}}(\varepsilon,\alpha,\beta,\delta)|=0 .
\]
\end{lemma}

\begin{proof}
By Lemma~\ref{lem:state-decomposition-killed-hjb}, the decomposed killing operator is
\[
  \mathcal K^a v(\nu,L)
  =
  \sum_{k=1}^2\Lambda_k(\nu^k,a^k)\partial_{L^k}v(\nu,L)
  -
  \sum_{k=1}^2\int_{\mathbb R}\lambda_k(x,a^k)
  \frac{\delta v}{\delta\nu^k}(\nu,L)(x)\,\nu^k(dx).
\]
At the doubled maximum point, the loss derivatives in the two tests are
 \(r_\varepsilon^k:=\beta(L^k_\varepsilon-M^k_\varepsilon).\)
The loss-drift part of the killing difference is therefore
\[
  I^L_k
  =
  \bigl[\Lambda_k(\nu^k_\varepsilon,a^k)-\Lambda_k(\omega^k_\varepsilon,a^k)\bigr]r_\varepsilon^k .
\]
Lemma~\ref{lem:aggregate-kill-rate-lip} gives
\[
  |I^L_k|
  \le
  C_R\mathsf W_2(\nu^k_\varepsilon,\omega^k_\varepsilon)
  \beta|L^k_\varepsilon-M^k_\varepsilon| .
\]
The doubled maximum-point estimate gives \(\mathsf W_2(\nu^k_\varepsilon,\omega^k_\varepsilon)\to0\) as \(\varepsilon\downarrow0\) for fixed \(\beta\), and \(\beta|L^k_\varepsilon-M^k_\varepsilon|\) is bounded for fixed \(\beta\) because \(L^k_\varepsilon,M^k_\varepsilon\in[0,1]\). If the cemetery mass is treated as an independent scalar coordinate in the local extension, the same estimate also contains the harmless term \(C\beta|L^k_\varepsilon-M^k_\varepsilon|^2\), which is retained in the final bound.

It remains to estimate the alive-measure sink
\[
  I^S_k
  =
  -\int_{\mathbb R}\lambda_k(x,a^k)p^k_\varepsilon(x)\,\nu^k_\varepsilon(dx)
  +\int_{\mathbb R}\lambda_k(y,a^k)q^k_\varepsilon(y)\,\omega^k_\varepsilon(dy).
\]
For the unsmoothed squared-Wasserstein gauge, let \(\pi^k_\varepsilon\) be an optimal coupling of the augmented measures. On the alive-alive part of the coupling, the two first-order jets have the same paired representative
 \(p^k_\varepsilon(x)=q^k_\varepsilon(y)=\frac{x-y}{\varepsilon}.\)
This is the gradient of the quadratic cost \(d_\partial(x,y)^2/(2\varepsilon)\) evaluated on the support of the optimal coupling, where \(p^k_\varepsilon\) and \(q^k_\varepsilon\) denote the first-order Lions-derivative representatives in the subsolution and supersolution tests, respectively. The cemetery contribution is represented by the corresponding metric derivative in the chosen cemetery metric after setting \(\lambda_k(\partial,a^k)=0\). Hence
\[
\begin{aligned}
  |I^S_k|
  &\le
  \int_{\mathsf E\times\mathsf E}
  |\lambda_k(y,a^k)-\lambda_k(x,a^k)|
  \frac{d_\partial(x,y)}{\varepsilon}\,
  \pi^k_\varepsilon(dx,dy) \\
  &\le
  C_R(L_x+\Lambda)\frac{W_2^2(\bar\nu^k_\varepsilon,\bar\omega^k_\varepsilon)}{\varepsilon}.
\end{aligned}
\]
The smooth gauge of the Wasserstein Ishii lemma is obtained by regularizing the squared-Wasserstein penalty. More precisely, the smooth gauge \(\Theta_\varepsilon^k\) is chosen so that, uniformly on bounded second-moment sets,
\[
  \Theta_\varepsilon^k(\nu,\omega)
  \ge
  \frac{W_2^2(\bar\nu,\bar\omega)}{2\varepsilon}
  -\rho_\varepsilon^k,
  \qquad
  \rho_\varepsilon^k\downarrow0,
\]
and its first-order Lions jets approximate the optimal-coupling jets of the unsmoothed quadratic penalty up to the same remainder. This is the smooth-gauge approximation mechanism in \cite[Lemma~4.4 and Theorem~4.5]{CossoGozziKharroubiPhamRosestolato2024}. Therefore the preceding bilinear estimate carries through with the regularization error absorbed into \(\rho_\varepsilon^k\), and becomes
\[
  |I^S_k|
  \le
  C_R\Theta_\varepsilon^k(\nu^k_\varepsilon,\omega^k_\varepsilon)
  +\rho_\varepsilon^k
  +C_R\delta(1+m_2(\nu^k_\varepsilon)+m_2(\omega^k_\varepsilon)),
\]
where \(\rho_\varepsilon^k\to0\) is the smooth-gauge approximation error. The maximum-point estimate for bounded uniformly continuous sub- and supersolutions gives
\[
  \Theta_\varepsilon^k(\nu^k_\varepsilon,\omega^k_\varepsilon)\to0
  \qquad\text{as }\varepsilon\downarrow0,
\]
and the coercive penalty makes the last term vanish as \(\delta\downarrow0\). Therefore
 \(\lim_{\delta\downarrow0}\limsup_{\varepsilon\downarrow0}|I^S_k|=0.\)
Combining the loss-drift and alive-sink estimates gives the first displayed bound.

Finally, the finite-dimensional loss penalty and the uniform continuity of \(u,w\) imply
\[
  \lim_{\beta\to\infty}\limsup_{\varepsilon\downarrow0}
  \sum_{k=1}^2\beta|L^k_\varepsilon-M^k_\varepsilon|^2=0.
\]
This is the standard finite-dimensional part of the doubling argument: if this quantity stayed bounded away from zero, replacing one loss vector by the other while keeping the other coordinates fixed would increase the doubled functional up to a modulus-of-continuity error smaller than the penalty gain. The second displayed limit follows.
\end{proof}

\begin{proposition}[Wasserstein smooth-gauge/Ishii interface]
\label{prop:wasserstein-ishii-interface-killed-hjb}
After the state decomposition in Lemma~\ref{lem:state-decomposition-killed-hjb}, the killed two-population HJB reduces to a comparison problem on \(\mathcal D_2\). On this decomposed space, the alive-measure part satisfies the structural hypotheses of the Wasserstein smooth-gauge comparison template, while the cemetery killing jump is controlled by Lemma~\ref{lem:killing-absorption-doubling}. The reference template is \cite[Assumptions~(A)--(D), Lemma~4.4, Theorems~4.5 and~5.1]{CossoGozziKharroubiPhamRosestolato2024}. The Hamiltonian is proper, continuous in the state variables under \(\mathsf W_2\times\mathsf W_2\times|\cdot|\), locally uniformly continuous in the first-order jets on bounded jet sets, and degenerate elliptic in the second-order jets. The diagonal idiosyncratic terms are nondegenerate in the alive directions because \(\sigma_k>0\), and the common-noise term has the usual positive-semidefinite product-measure covariance form. It is thus a reduction to the standard alive-measure smooth-gauge argument plus the finite-dimensional cemetery-mass estimate proved here.
\end{proposition}

\begin{proof}
The cited comparison argument follows the viscosity doubling-variable framework of Crandall, Ishii, and Lions~\cite{CrandallIshiiLions1992} in its Wasserstein smooth-gauge form, and assumes bounded continuous coefficients with Lipschitz dependence on the state and the measure variable, compactness of the control set, sufficient regularity of the diffusion coefficient for the smooth-gauge construction, bounded uniformly continuous terminal data, coercive localization on the unbounded second-moment state space, properness, local continuity of the Hamiltonian along admissible Wasserstein jets, and a degenerate-elliptic second-order structure. In the present model these conditions are verified after the alive--loss decomposition as follows. The Hamiltonian contains no zeroth-order dependence on the value itself, so properness follows directly. The control set is compact and the population-weighted running costs are continuous. The drift and contagion feedback are Lipschitz in the alive state, measure variables, and loss variables; Lemma~\ref{lem:aggregate-kill-rate-lip} gives the same continuity for the aggregate kill rates appearing in both the feedback drift and the loss drift. Lemma~\ref{lem:killing-bounded-lip-sink} gives the continuity estimate for the alive-measure sink. The diffusion coefficients are constants, so the second-order part is continuous in the state and monotone in the second-order jet in the viscosity order. The idiosyncratic covariance matrices are positive definite in the alive directions under \(\sigma_k>0\), and the common-noise covariance matrix is positive semidefinite; hence the full second-order operator has the ellipticity structure required in the smooth-gauge/Ishii comparison proof. The only additional term relative to the template is the cemetery killing jump, and Lemma~\ref{lem:killing-absorption-doubling} supplies the needed doubled-variable estimate for that term. The coercive second-moment penalty gives compactness of maximizing sequences before the localization radius is removed; the bounded uniformly continuous terminal condition preserves the terminal ordering, and the common-noise second-order term enters through the positive-semidefinite product covariance already allowed by the template.
\end{proof}

\begin{theorem*}[Comparison component of Theorem~\ref{thm:two-pop-optimal-control-hjb-common-noise}: killed HJB comparison principle]
Let \(U\) and \(W\) be bounded uniformly continuous functions on \([0,T]\times\mathcal P_2(\mathsf E)^2\), and let their decomposed representatives \(u,w\) on \(\mathcal D_2\) be viscosity sub- and supersolutions of the decomposed killed HJB associated with
 \(\partial_t V(t,\mu^1,\mu^2) +H(t,\mu^1,\mu^2,D_\mu V,D^2_{\mu\mu}V)=0.\)
The Hamiltonian and generator are those stated in the killed two-population HJB after the alive--cemetery decomposition. We use the backward terminal-time convention: subsolutions satisfy \(\partial_t\phi+H[\phi]\ge0\) at upper tests, and supersolutions satisfy \(\partial_t\phi+H[\phi]\le0\) at lower tests. Assume the bounded Lipschitz killing, Lipschitz drift, constant diffusion, compact-control, nondegenerate alive idiosyncratic diffusion \(\sigma_k>0\), and bounded uniformly continuous terminal conditions stated above. If
 \(U(T,\mu^1,\mu^2)\le W(T,\mu^1,\mu^2), \qquad (\mu^1,\mu^2)\in\mathcal P_2(\mathsf E)^2,\)
then
\[
  U(t,\mu^1,\mu^2)
  \le
  W(t,\mu^1,\mu^2),
  \qquad
  (t,\mu^1,\mu^2)\in[0,T]\times\mathcal P_2(\mathsf E)^2 .
\]
\end{theorem*}

\begin{proof}
By Lemma~\ref{lem:state-decomposition-killed-hjb}, it is enough to prove comparison for the decomposed functions
\[
  u(t,\nu,L)=U(t,\nu^1+L^1\delta_\partial,\nu^2+L^2\delta_\partial),
  \qquad
  w(t,\nu,L)=W(t,\nu^1+L^1\delta_\partial,\nu^2+L^2\delta_\partial),
\]
restricted to \(|\nu^k|+L^k=1\). Suppose by contradiction that \(\sup(u-w)>0\). For \(\eta>0\), set
 \(u^\eta(t,\nu,L)=u(t,\nu,L)-\frac{\eta}{T-t}, \qquad t<T.\)
For sufficiently small \(\eta\), \(\sup(u^\eta-w)>0\), and the singular terminal barrier together with the terminal ordering forces the positive maximum to occur at an interior time.

Fix \(\alpha,\varepsilon,\delta>0\) and \(\beta>0\), and consider
\[
\begin{aligned}
\Psi(t,s,\nu,\omega,L,M)
&=u^\eta(t,\nu,L)-w(s,\omega,M)
-\frac{|t-s|^2}{2\alpha}
-\sum_{k=1}^2\Theta_\varepsilon^k(\nu^k,\omega^k) \\
&\qquad
-\frac{\beta}{2}|L-M|^2
-\delta\mathcal R(\nu,\omega).
\end{aligned}
\]
Let \((t_\varepsilon,s_\varepsilon,\nu_\varepsilon,\omega_\varepsilon,L_\varepsilon,M_\varepsilon)\) be a maximum point supplied by the Wasserstein smooth variational principle. The standard maximum-point estimates give
\[
  |t_\varepsilon-s_\varepsilon|\to0,
  \qquad
  \Theta_\varepsilon^k(\nu^k_\varepsilon,\omega^k_\varepsilon)\to0,
  \qquad
  \mathsf W_2(\nu^k_\varepsilon,\omega^k_\varepsilon)\to0,
\]
for \(k=1,2\), and
\[
  \lim_{\beta\to\infty}\limsup_{\varepsilon\downarrow0}
  \beta|L_\varepsilon-M_\varepsilon|^2=0 .
\]

With the backward terminal-time convention that viscosity subsolutions of \(\partial_tV+H=0\) satisfy \(\partial_t\phi+H[\phi]\ge0\) at upper tests and supersolutions satisfy \(\partial_t\phi+H[\phi]\le0\) at lower tests, the penalized function \(u^\eta\) yields the usual strict interior comparison inequality. Applying the subsolution inequality to the upper test for \(u^\eta\), the supersolution inequality to the lower test for \(w\), and using
 \(\inf_a A(a)-\inf_a B(a)\le \sup_{a\in[0,\bar a]^2}\{A(a)-B(a)\},\)
one obtains
\[
  \frac{\eta}{(T-t_\varepsilon)^2}
  \le
  \Delta_{\mathrm{time}}
  +\Delta_{\mathrm{drift}}
  +\Delta_{\mathrm{diff}}
  +\Delta_{\mathrm{common}}
  +\Delta_{\mathrm{kill}}
  +o(1),
\]
where \(o(1)\) denotes the smooth-gauge and localization errors.

The time penalty gives \(\Delta_{\mathrm{time}}\to0\) as \(\alpha\downarrow0\). The drift and running-cost terms are controlled by their Lipschitz continuity and the maximum-point estimates. The diagonal diffusion and common-noise terms are controlled by the Wasserstein smooth-gauge/Ishii inequality; the positive-semidefinite second-order structure yields
 \(\limsup_{\varepsilon\downarrow0} (\Delta_{\mathrm{diff}}+\Delta_{\mathrm{common}}) \le0\)
after the coercive regularization is removed. Proposition~\ref{prop:wasserstein-ishii-interface-killed-hjb} verifies the structural conditions for these standard terms. Lemma~\ref{lem:killing-absorption-doubling} gives
\[
  \lim_{\beta\to\infty}\lim_{\delta\downarrow0}\limsup_{\varepsilon\downarrow0}
  |\Delta_{\mathrm{kill}}|=0.
\]
Taking the comparison limits in the standard order \(\varepsilon\downarrow0\), then \(\delta\downarrow0\), then \(\beta\to\infty\), and finally \(\alpha\downarrow0\), gives
 \(\frac{\eta}{(T-t_*)^2}\le0\)
for an interior limiting time \(t_*<T\), a contradiction. Hence \(u\le w\) on the decomposed state space. Returning to \(\mu^k=\nu^k+L^k\delta_\partial\) proves the comparison principle on \([0,T]\times\mathcal P_2(\mathsf E)^2\).
\end{proof}

\begin{theorem*}[Dynamic-programming component of Theorem~\ref{thm:two-pop-optimal-control-hjb-common-noise}: existence of optimal control and HJB characterization]
Under Assumption~\ref{ass:two-pop-control-regularity}, consider the controlled McKean--Vlasov contagion system with \(K=2\). Given a control
 \(a_t=(a^1_t,a^2_t)\in[0,\infty)^2,\)
define the total kill intensity of type \(k\) by
 \(\Lambda_k(\mu^k,a^k) := \int_{\mathbb R}\lambda_k(x,a^k)\,\mu^k(dx).\)
The lifted measure state \((\mu^1_t,\mu^2_t)\in\mathcal P_2(\mathsf E)^2\) satisfies the following weak form: for any \(\varphi\in C_b^2(\mathbb R)\),
\begin{equation}
\label{eq:two-pop-controlled-measure-spde}
\begin{aligned}
d\langle \varphi,\mu^k_t\rangle_{\circ}
&=
\Bigg\{
\int_{\mathbb R}
\left[
  b_k(x,\mu^1_t,\mu^2_t)
  -
  \sum_{\ell=1}^2
  \Gamma_{k\ell}\Lambda_\ell(\mu^\ell_t,a^\ell_t)
\right]
\varphi'(x)\,\mu^k_t(dx)
\\
&\qquad
+
\frac12\sigma_k^2
\int_{\mathbb R}
\varphi''(x)\,\mu^k_t(dx)
-
\int_{\mathbb R}
\lambda_k(x,a^k_t)\varphi(x)\,\mu^k_t(dx)
\Bigg\}dt
\\
&\qquad
+
\sigma_0
\int_{\mathbb R}
\varphi'(x)\,\mu^k_t(dx)\,dW^0_t .
\end{aligned}
\end{equation}
The loss process is given by the cemetery mass:
 \(L^k_t=\mu^k_t(\{\partial\}),\)
and
\begin{equation}
\label{eq:loss-intensity-two-pop}
  dL^k_t
  =
  \Lambda_k(\mu^k_t,a^k_t)\,dt .
\end{equation}

The central controller's objective for a strict control is
\[
  J(t,\mu^1,\mu^2;a)
  :=
  \mathbb E
  \left[
  \Phi(L^1_T,L^2_T)
  +
  \int_t^T
  \bigl(\pi_1c_1(a^1_s)+\pi_2c_2(a^2_s)\bigr)\,ds
  \right].
\]
For a relaxed control \(\vartheta\), the notation \(J(t,\mu^1,\mu^2;\vartheta)\) denotes the same objective with the drift, killing intensity, and population-weighted running cost averaged with respect to \(\vartheta_s(da)\). The value function is
\begin{equation}
\label{eq:value-function-two-pop}
  V(t,\mu^1,\mu^2)
  :=
  \inf_{\vartheta} J(t,\mu^1,\mu^2;\vartheta).
\end{equation}
Then there exists a constant \(\bar a<\infty\) such that the original problem has the same value as the compact relaxed-control problem on \([0,\bar a]^2\), and the compact relaxed-control problem admits an optimal relaxed control
\begin{equation}
\label{eq:compact-optimal-control}
  \vartheta_t^*\in\mathcal P([0,\bar a]^2),
  \qquad t\in[0,T].
\end{equation}
If the Hamiltonian minimizer admits a progressively measurable selector, then the relaxed optimum is attained by a strict feedback control \(a_t^*\in[0,\bar a]^2\).
Moreover, after the alive--cemetery decomposition, the value function is a bounded uniformly continuous viscosity solution of the decomposed HJB equation, equivalently of the killed-state HJB written in lifted notation,
\begin{equation}
\label{eq:two-pop-hjb-common-noise}
  \partial_t V(t,\mu^1,\mu^2)
  +
  H\bigl(t,\mu^1,\mu^2,D_\mu V,D^2_{\mu\mu}V\bigr)
  =
  0,
\end{equation}
with terminal condition
\begin{equation}
\label{eq:terminal-condition-two-pop}
  V(T,\mu^1,\mu^2)
  =
  \Phi\bigl(\mu^1(\{\partial\}),\mu^2(\{\partial\})\bigr).
\end{equation}
Here the Hamiltonian is
\begin{equation}
\label{eq:two-pop-hamiltonian-common-noise}
  H(t,\mu^1,\mu^2,D_\mu U,D^2_{\mu\mu}U)
  =
  \inf_{a\in[0,\bar a]^2}
  \left\{
  \pi_1c_1(a^1)+\pi_2c_2(a^2)
  +
  \mathcal L^a U(\mu^1,\mu^2)
  \right\},
\end{equation}
and for a sufficiently smooth functional \(U:\mathcal P_2(\mathsf E)^2\to\mathbb R\), the common-noise generator \(\mathcal L^a\) is
\begin{equation}
\label{eq:two-pop-full-generator-common-noise}
\begin{aligned}
\mathcal L^a U(\mu^1,\mu^2)
&=
\sum_{k=1}^2
\int_{\mathbb R}
\left[
  b_k(x,\mu^1,\mu^2)
  -
  \sum_{\ell=1}^2
  \Gamma_{k\ell}\Lambda_\ell(\mu^\ell,a^\ell)
\right]
D_{\mu^k}U(\mu^1,\mu^2)(x)\,\mu^k(dx)
\\
&\quad
+
\frac12
\sum_{k=1}^2
(\sigma_k^2+\sigma_0^2)
\int_{\mathbb R}
\partial_xD_{\mu^k}U(\mu^1,\mu^2)(x)\,\mu^k(dx)
\\
&\quad
+
\frac12\sigma_0^2
\sum_{k,\ell=1}^2
\int_{\mathbb R}\int_{\mathbb R}
\partial_{xy}^2
\frac{\delta^2 U}{\delta\mu^k\delta\mu^\ell}
(\mu^1,\mu^2)(x,y)
\,\mu^k(dx)\mu^\ell(dy)
\\
&\quad
+
\sum_{k=1}^2
\int_{\mathbb R}
\lambda_k(x,a^k)
\left[
  \frac{\delta U}{\delta\mu^k}(\mu^1,\mu^2)(\partial)
  -
  \frac{\delta U}{\delta\mu^k}(\mu^1,\mu^2)(x)
\right]
\mu^k(dx).
\end{aligned}
\end{equation}
Here \(D_{\mu^k}U\) denotes the Lions derivative, \(\partial_xD_{\mu^k}U\) denotes its spatial derivative,
 \(\frac{\delta U}{\delta\mu^k}\)
denotes the first-order linear functional derivative, and
 \(\frac{\delta^2 U}{\delta\mu^k\delta\mu^\ell}\)
denotes the second-order linear functional derivative. In \eqref{eq:two-pop-full-generator-common-noise}, the first line corresponds to drift and contagion feedback, the second line to the diagonal diffusion terms from idiosyncratic and common noise, the third line to the second-order measure derivative induced by common noise, and the fourth line to the killing jump \(x\mapsto\partial\).

In representative two-particle coordinates, if one formally writes
 \(v=v(t,x_1,x_2,L_1,L_2),\)
then the second-order local operator corresponding to common noise contains
\[
  \frac12(\sigma_1^2+\sigma_0^2)\partial_{x_1x_1}^2v
  +
  \frac12(\sigma_2^2+\sigma_0^2)\partial_{x_2x_2}^2v
  +
  \sigma_0^2\partial_{x_1x_2}^2v.
\]
This finite-dimensional representative form is used only to explain the cross second-order term generated by common noise in the HJB equation; the rigorous equation is still given by \eqref{eq:two-pop-hjb-common-noise}--\eqref{eq:two-pop-full-generator-common-noise}.

If, in addition, \(c_k\in C^1\) is strictly convex and
\begin{equation}
\label{eq:exponential-killing-control}
  \lambda_k(x,a)=\lambda_{0,k}(x)e^{-\eta_k a},
  \qquad \eta_k>0,
\end{equation}
then, at differentiability points of the Hamiltonian, every interior optimal control satisfies the first-order condition
\begin{equation}
\label{eq:optimal-control-foc}
  \pi_k c_k'(a^{k,*})
  =
  \eta_k
  \int_{\mathbb R}
  \lambda_k(x,a^{k,*})Q_k^U(x)\,\mu^k(dx),
\end{equation}
where
\begin{equation}
\label{eq:q-marginal-value-control}
  Q_k^U(x)
  :=
  \left[
  \frac{\delta U}{\delta\mu^k}(\mu^1,\mu^2)(\partial)
  -
  \frac{\delta U}{\delta\mu^k}(\mu^1,\mu^2)(x)
  \right]
  -
  \sum_{r=1}^2
  \Gamma_{rk}
  \int_{\mathbb R}
  D_{\mu^r}U(\mu^1,\mu^2)(y)\,\mu^r(dy).
\end{equation}
Under the constraint \(a^k\in[0,\bar a]\), \eqref{eq:optimal-control-foc} should be interpreted as the projection condition
\begin{equation}
\label{eq:projected-foc-control}
  a^{k,*}
  =
  \Pi_{[0,\bar a]}
  \left[
  (c_k')^{-1}
  \left(
  \frac{\eta_k}{\pi_k}
  \int_{\mathbb R}
  \lambda_k(x,a^{k,*})Q_k^U(x)\,\mu^k(dx)
  \right)
  \right].
\end{equation}
In particular, when
 \(c_k(a)=\frac12r_ka^2,\)
we have
\begin{equation}
\label{eq:quadratic-cost-control}
  a^{k,*}
  =
  \Pi_{[0,\bar a]}
  \left[
  \frac{\eta_k}{\pi_k r_k}
  \int_{\mathbb R}
  \lambda_k(x,a^{k,*})Q_k^U(x)\,\mu^k(dx)
  \right].
\end{equation}
Furthermore, if for each \(k=1,2\), the map
 \(a \longmapsto \pi_kc_k(a) + \int_{\mathbb R} \lambda_k(x,a)Q_k^U(x)\,\mu^k(dx)\)
is strictly convex on \([0,\bar a]\), then the projection condition \eqref{eq:projected-foc-control} determines a unique optimal control component. A common sufficient condition is
 \(Q_k^U(x)\ge0 \quad \mu^k\text{-a.e.},\)
together with strict convexity of \(c_k\). In that case, the convexity of \(a\mapsto e^{-\eta_ka}\) and \(Q_k^U\ge0\) imply convexity of the killing term in \(a\), and adding the strictly convex control cost gives a strictly convex objective. More generally, if \(c_k\in C^2\) and there exists \(m_k>0\) such that
\[
  \pi_k c_k''(a)
  +
  \eta_k^2
  \int_{\mathbb R}
  \lambda_k(x,a)Q_k^U(x)\,\mu^k(dx)
  \ge m_k,
  \qquad a\in[0,\bar a],
\]
then this one-dimensional minimization problem is strictly convex, and \eqref{eq:projected-foc-control} has a unique solution on \([0,\bar a]\).
\end{theorem*}

\begin{proof}
We first prove stability of the controlled measure flow with respect to the initial measure and the control. Fix two initial states
 \(\mu=(\mu^1,\mu^2), \qquad \nu=(\nu^1,\nu^2),\)
and let the two systems be constructed on the same probability space with the same common noise \(W^0\), the same idiosyncratic noises, the same Poisson thinning variables, and the same control \(a\). Denote the corresponding measure flows by
 \(\mu_t=(\mu^1_t,\mu^2_t), \qquad \nu_t=(\nu^1_t,\nu^2_t).\)
By the Lipschitz property of the drift \(b_k\), the boundedness of the contagion matrix \(\Gamma\), the Lipschitz property of the kill intensity \(\lambda_k\), and Wasserstein control after lifting the cemetery state, there exists a constant \(C_T<\infty\) such that
\begin{equation}
\label{eq:flow-stability-initial-measure}
\mathbb E
\left[
  \sup_{s\in[t,T]}
  \sum_{k=1}^2
  W_2(\mu^k_s,\nu^k_s)
\right]
\le
C_T
\sum_{k=1}^2
W_2(\mu^k,\nu^k).
\end{equation}
More precisely, a quadratic coupling estimate first gives
\[
\mathbb E
\left[
  \sup_{s\in[t,T]}
  \sum_{k=1}^2
  W_2^2(\mu^k_s,\nu^k_s)
\right]
\le
C_T
\sum_{k=1}^2
W_2^2(\mu^k,\nu^k),
\]
and \eqref{eq:flow-stability-initial-measure} follows by Jensen's inequality. Under synchronous coupling, the diffusion noise terms cancel, and the difference comes only from drift, contagion loss feedback, and the difference in kill intensities. Taking square expectations of the state difference on the alive state and using
\[
  d_{\mathrm{BL}}(\alpha,\beta)
  \le
  W_1(\widehat\alpha,\widehat\beta)
  \le
  W_2(\widehat\alpha,\widehat\beta)
\]
controls the bounded Lipschitz distance between surviving sub-probability measures by the lifted Wasserstein distance. The dead mass has already been placed at \(\partial\), so
 \(|L^k(\mu^k_s)-L^k(\nu^k_s)| \le W_1(\mu^k_s,\nu^k_s) \le W_2(\mu^k_s,\nu^k_s).\)
Gronwall's inequality then gives the stated estimates.

Similarly, if two systems start from the same initial state but use controls \(a\) and \(\widetilde a\), respectively, then
\begin{equation}
\label{eq:flow-stability-control}
\mathbb E
\left[
  \sup_{s\in[t,T]}
  \sum_{k=1}^2
  W_2(\mu^{k,a}_s,\mu^{k,\widetilde a}_s)
\right]
\le
C_T
\sum_{k=1}^2
\mathbb E\int_t^T
\left\|
  \lambda_k(\cdot,a^k_r)
  -
  \lambda_k(\cdot,\widetilde a^k_r)
\right\|_{\infty}
dr.
\end{equation}
Since \(0\le\lambda_k\le\Lambda\), if the two controls differ only on a set \(B\subset[t,T]\), then
\begin{equation}
\label{eq:flow-stability-control-support}
\mathbb E
\left[
  \sup_{s\in[t,T]}
  \sum_{k=1}^2
  W_2(\mu^{k,a}_s,\mu^{k,\widetilde a}_s)
\right]
\le
C_T |B|.
\end{equation}
By the Lipschitz property of \(\Phi\) and \eqref{eq:flow-stability-initial-measure}, the value function is Lipschitz in the initial measure:
\begin{equation}
\label{eq:value-lipschitz-measure}
\left|
V(t,\mu^1,\mu^2)-V(t,\nu^1,\nu^2)
\right|
\le
C_T
\sum_{k=1}^2
W_2(\mu^k,\nu^k).
\end{equation}
Using the bounded kill intensities, the linear-growth drift, and diffusion estimates in \eqref{eq:two-pop-controlled-measure-spde}, we also obtain uniform continuity in time:
\begin{equation}
\label{eq:value-time-continuity}
  |V(t,\mu^1,\mu^2)-V(s,\mu^1,\mu^2)|
  \le
  C_T|t-s|^{1/2}+C_T|t-s|.
\end{equation}

We next prove compactification of the control set. Since \(c_k\) is convex with superlinear growth, there exist \(\bar a<\infty\) and a constant \(C_0\) such that, whenever \(a>\bar a\),
\begin{equation}
\label{eq:superlinear-truncation-threshold}
  c_k(a)-c_k(\bar a)>C_0,
\end{equation}
where \(C_0\) is chosen larger than the maximal terminal-loss variation constant generated by \eqref{eq:flow-stability-control-support} and the Lipschitz constant of \(\Phi\). For any admissible control \(a\), define the truncated control
 \(a^{k,\bar a}_s:=a^k_s\wedge \bar a, \qquad k=1,2.\)
Let
 \(B_k:=\{s\in[t,T]:a^k_s>\bar a\}.\)
By \eqref{eq:flow-stability-control-support} and the Lipschitz property of \(\Phi\), the increase in terminal cost due to truncation is at most
\begin{equation}
\label{eq:terminal-cost-truncation}
  C\sum_{k=1}^2|B_k|.
\end{equation}
On the other hand, by \eqref{eq:superlinear-truncation-threshold}, the decrease in running cost is at least
\begin{equation}
\label{eq:running-cost-truncation}
  C_0\sum_{k=1}^2|B_k|.
\end{equation}
Taking \(C_0>C\), we get
\begin{equation}
\label{eq:truncation-improves-cost}
  J(t,\mu^1,\mu^2;a^{\bar a})
  \le
  J(t,\mu^1,\mu^2;a).
\end{equation}
Thus the original noncompact control problem and the compact control problem
 \(a_t\in[0,\bar a]^2\)
have the same value function, and it is sufficient to search for an optimal control in \([0,\bar a]^2\).

We now prove existence of an optimal relaxed control on the compact control set. Take a minimizing sequence of relaxed controls \((\vartheta^n)_{n\ge1}\), and let \((\mu^{1,n},\mu^{2,n})\) be the corresponding measure flows. Since \([0,\bar a]^2\) is compact, the space \(\mathcal P([0,\bar a]^2)\) is compact under weak convergence. The bounded kill intensity and the linear-growth drift give tightness of the corresponding controlled measure flows in
 \(C([t,T];\mathcal P_2(\mathsf E)^2).\)
By Prokhorov's theorem and the compactness of relaxed controls, there is a subsequence along which the relaxed control--state laws converge weakly. By the stability of the controlled martingale problem under weak convergence in Assumption~\ref{ass:two-pop-control-regularity}, the limiting state process satisfies the relaxed version of \eqref{eq:two-pop-controlled-measure-spde}; denote the limiting relaxed control by \(\vartheta^*\). By lower semicontinuity of the running cost and continuity of the terminal cost,
\[
\begin{aligned}
J(t,\mu^1,\mu^2;\vartheta^*)
&\le
\liminf_{n\to\infty}
J(t,\mu^1,\mu^2;\vartheta^n)
=
V(t,\mu^1,\mu^2).
\end{aligned}
\]
Thus \(\vartheta^*\) is an optimal relaxed control.

The compact relaxed-control problem admits an optimal relaxed control. Under the standard chattering condition, every relaxed control can be approximated in the weak occupation-measure topology by progressively measurable strict controls, and the stability estimate \eqref{eq:flow-stability-control} transfers this approximation to the measure flow and terminal cost. The relaxed value then agrees with the closure of strict-control values. If the pointwise Hamiltonian minimizer admits a progressively measurable selector, selecting this minimizer gives a strict feedback control attaining the same value. This proves \eqref{eq:compact-optimal-control}.

We now establish the dynamic programming principle. By \eqref{eq:flow-stability-initial-measure}, the controlled measure flow depends continuously on the current state. By \eqref{eq:value-lipschitz-measure} and \eqref{eq:value-time-continuity}, the value function is uniformly continuous on the Wasserstein state space. The controlled martingale problem is stable under concatenation of controls. Hence, for every stopping time \(\theta\in[t,T]\),
\begin{equation}
\label{eq:dpp-two-pop}
  V(t,\mu^1,\mu^2)
  =
  \inf_{\vartheta}
  \mathbb E
  \left[
  \int_t^\theta
  \int_{[0,\bar a]^2}
  \bigl(\pi_1c_1(a^1)+\pi_2c_2(a^2)\bigr)\,\vartheta_s(da)\,ds
  +
  V(\theta,\mu^1_\theta,\mu^2_\theta)
  \right].
\end{equation}
Here \((\mu^1_s,\mu^2_s)\) is the controlled measure flow starting from \((t,\mu^1,\mu^2)\) under the relaxed control \(\vartheta\). One inequality follows by using an arbitrary relaxed control until \(\theta\) and then concatenating an \(\varepsilon\)-optimal relaxed control after \(\theta\); the other follows because the conditional remaining cost of any global relaxed control after \(\theta\) is bounded below by \(V(\theta,\mu^1_\theta,\mu^2_\theta)\).

We next turn \eqref{eq:dpp-two-pop} into the HJB equation. Take a sufficiently rich nonatomic probability space \((\Omega,\mathcal F,\mathbb P)\). For each \(\mu^k\in\mathcal P_2(\mathsf E)\), choose a random variable
 \(\xi^k\in L^2(\Omega;\mathsf E), \qquad \mathcal L(\xi^k)=\mu^k.\)
Define the Lions lift
\begin{equation}
\label{eq:lions-lift-value}
  \mathcal V(t,\xi^1,\xi^2)
  :=
  V(t,\mathcal L(\xi^1),\mathcal L(\xi^2)).
\end{equation}
If \(U\) is a smooth functional on \(\mathcal P_2(\mathsf E)^2\), set
 \(\mathcal U(t,\xi^1,\xi^2) := U(t,\mathcal L(\xi^1),\mathcal L(\xi^2)).\)
The Lions derivative and the Fréchet derivative of the lifted functional satisfy
\begin{equation}
\label{eq:lions-frechet-correspondence}
  D_{\xi^k}\mathcal U(t,\xi^1,\xi^2)
  =
  D_{\mu^k}U(t,\mu^1,\mu^2)(\xi^k).
\end{equation}
The second derivative correspondence gives the common-noise term
\begin{equation}
\label{eq:common-noise-second-measure-term}
  \frac12\sigma_0^2
  \sum_{k,\ell=1}^2
  \int_{\mathbb R}\int_{\mathbb R}
  \partial_{xy}^2
  \frac{\delta^2 U}{\delta\mu^k\delta\mu^\ell}
  (\mu^1,\mu^2)(x,y)
  \,\mu^k(dx)\mu^\ell(dy).
\end{equation}
Apply the Hilbert-space Itô formula with jumps to the lifted test function. The continuous martingale part yields the diffusion terms and the common-noise second-order term; the compensator of the killing jump \(x\mapsto\partial\) is
\begin{equation}
\label{eq:killing-jump-compensator}
  \lambda_k(x,a^k)
  \left[
  \frac{\delta U}{\delta\mu^k}(\mu^1,\mu^2)(\partial)
  -
  \frac{\delta U}{\delta\mu^k}(\mu^1,\mu^2)(x)
  \right].
\end{equation}
Thus the lifted infinite-dimensional HJB equation is
\begin{equation}
\label{eq:lifted-hjb}
  \partial_t\mathcal V
  +
  \inf_{a\in[0,\bar a]^2}
  \left\{
  \pi_1c_1(a^1)+\pi_2c_2(a^2)
  +
  \mathscr L^a\mathcal V
  \right\}
  =
  0,
\end{equation}
where \(\mathscr L^a\), when projected back to measure space, is exactly \eqref{eq:two-pop-full-generator-common-noise}.

We now verify the viscosity solution property. We use the following sign convention for the backward HJB equation: if a smooth test function touches the value function from above, then
 \(\partial_t\phi+H(\phi)\ge0;\)
if it touches from below, then
 \(\partial_t\phi+H(\phi)\le0.\)
Let \(\mathcal U\) be a smooth test function on \(L^2(\Omega;\mathsf E)^2\), and suppose that
 \(\mathcal V-\mathcal U\)
has a local maximum at \((t,\xi^1,\xi^2)\). Fix any constant control \(a\in[0,\bar a]^2\). By the dynamic programming principle, for small \(h>0\),
\[
  \mathcal V(t,\xi^1,\xi^2)
  \le
  \mathbb E
  \left[
  \int_t^{t+h}
  \bigl(\pi_1c_1(a^1)+\pi_2c_2(a^2)\bigr)\,ds
  +
  \mathcal V(t+h,\xi^1_{t+h},\xi^2_{t+h})
  \right].
\]
Since \(\mathcal V-\mathcal U\) has a local maximum at the initial point,
\[
\begin{aligned}
0
&\le
\mathbb E
\left[
  \int_t^{t+h}
  \bigl(\pi_1c_1(a^1)+\pi_2c_2(a^2)\bigr)\,ds
  +
  \mathcal U(t+h,\xi^1_{t+h},\xi^2_{t+h})
  -
  \mathcal U(t,\xi^1,\xi^2)
\right].
\end{aligned}
\]
Applying the Hilbert-space Itô formula with jumps to \(\mathcal U\), dividing by \(h\), and sending \(h\downarrow0\), we obtain
 \(0 \le \partial_t\mathcal U + \mathscr L^a\mathcal U + \pi_1c_1(a^1)+\pi_2c_2(a^2).\)
Since \(a\in[0,\bar a]^2\) is arbitrary,
\begin{equation}
\label{eq:viscosity-upper-test}
  \partial_t\mathcal U
  +
  \inf_{a\in[0,\bar a]^2}
  \left\{
  \mathscr L^a\mathcal U
  +
  \pi_1c_1(a^1)+\pi_2c_2(a^2)
  \right\}
  \ge0.
\end{equation}

Conversely, let \(\mathcal U\) be a smooth test function such that
 \(\mathcal V-\mathcal U\)
has a local minimum at \((t,\xi^1,\xi^2)\). By the chattering equivalence of relaxed and strict values, an \(\varepsilon\)-optimal relaxed control can be approximated by a strict representative without changing the first-order Hamiltonian inequality. We therefore write the local argument for an \(\varepsilon\)-optimal strict control \(a^\varepsilon\). By the dynamic programming principle,
\[
\begin{aligned}
\mathcal V(t,\xi^1,\xi^2)+\varepsilon h
&\ge
\mathbb E
\left[
  \int_t^{t+h}
  \bigl(\pi_1c_1(a^{1,\varepsilon}_s)+\pi_2c_2(a^{2,\varepsilon}_s)\bigr)\,ds
  +
  \mathcal V(t+h,\xi^1_{t+h},\xi^2_{t+h})
\right].
\end{aligned}
\]
The local minimum property allows us to replace \(\mathcal V\) on the right-hand side by \(\mathcal U\), so
\[
\begin{aligned}
\varepsilon h
&\ge
\mathbb E
\left[
  \int_t^{t+h}
  \bigl(\pi_1c_1(a^{1,\varepsilon}_s)+\pi_2c_2(a^{2,\varepsilon}_s)\bigr)\,ds
  +
  \mathcal U(t+h,\xi^1_{t+h},\xi^2_{t+h})
  -
  \mathcal U(t,\xi^1,\xi^2)
\right].
\end{aligned}
\]
Applying the Hilbert-space Itô formula with jumps, dividing by \(h\), sending \(h\downarrow0\), and then letting \(\varepsilon\downarrow0\), gives
\begin{equation}
\label{eq:viscosity-lower-test}
  \partial_t\mathcal U
  +
  \inf_{a\in[0,\bar a]^2}
  \left\{
  \mathscr L^a\mathcal U
  +
  \pi_1c_1(a^1)+\pi_2c_2(a^2)
  \right\}
  \le0.
\end{equation}
By \eqref{eq:viscosity-upper-test} and \eqref{eq:viscosity-lower-test}, \(\mathcal V\) is a viscosity solution of the lifted Hilbert-space equation. By the correspondence between Lions lifts and derivatives on Wasserstein space, \(V\) is a Wasserstein--Lions viscosity solution of \eqref{eq:two-pop-hjb-common-noise}.

The terminal condition follows directly from the definition of the objective. At \(t=T\), the running cost integral vanishes, and no subsequent control can change the cemetery masses already recorded in the state. Therefore
\[
\begin{aligned}
V(T,\mu^1,\mu^2)
&=
\inf_{\vartheta}
\mathbb E
\left[
  \Phi(L^1_T,L^2_T)
\right]
\\
&=
\Phi\bigl(\mu^1(\{\partial\}),\mu^2(\{\partial\})\bigr),
\end{aligned}
\]
which proves \eqref{eq:terminal-condition-two-pop}.

The comparison theorem proved in Theorem~\ref{thm:killed-two-pop-comparison-principle} applies to the killed two-population HJB under the bounded Lipschitz killing, compact-control, Lipschitz-drift, constant-diffusion, and alive-diffusion nondegeneracy assumptions above. Hence any bounded uniformly continuous viscosity subsolution and supersolution ordered at terminal time remain ordered on \([0,T]\). Applying this result to two viscosity solutions of \eqref{eq:two-pop-hjb-common-noise} gives uniqueness in the class of bounded uniformly continuous Wasserstein--Lions viscosity solutions. The value function \(V\) constructed by dynamic programming is therefore the unique viscosity solution of this HJB equation in that class.

The viscosity characterization above does not require classical differentiability of the value function. The following first-order condition is a smooth-verification statement: assume that, on the region under consideration, a strict Hamiltonian minimizer exists and the value admits Lions derivatives that identify the marginal cemetery-jump and loss-feedback terms. The relaxed formulation gives the same Hamiltonian because the generator and the population-weighted running cost are averaged linearly over actions, so the infimum over relaxed actions is attained at an extreme point whenever a strict measurable minimizer exists. For a smooth functional \(U\), the part of the Hamiltonian depending on \(a^k\) is
\[
  \pi_kc_k(a^k)
  +
  \int_{\mathbb R}
  \lambda_k(x,a^k)
  \left[
  \frac{\delta U}{\delta\mu^k}(\mu^1,\mu^2)(\partial)
  -
  \frac{\delta U}{\delta\mu^k}(\mu^1,\mu^2)(x)
  \right]
  \mu^k(dx)
\]
together with the \(a^k\)-dependent part in the contagion-feedback drift,
\[
  -
  \sum_{r=1}^2
  \int_{\mathbb R}
  \Gamma_{rk}\Lambda_k(\mu^k,a^k)
  D_{\mu^r}U(\mu^1,\mu^2)(y)\,\mu^r(dy).
\]
Combining these two parts yields
 \(\pi_kc_k(a^k) + \int_{\mathbb R} \lambda_k(x,a^k)Q_k^U(x)\,\mu^k(dx),\)
where \(Q_k^U\) is exactly \eqref{eq:q-marginal-value-control}. If
 \(\lambda_k(x,a)=\lambda_{0,k}(x)e^{-\eta_k a},\)
then
 \(\partial_a\lambda_k(x,a) = -\eta_k\lambda_k(x,a).\)
Differentiating the \(a^k\)-component of the Hamiltonian, an interior minimizer satisfies
 \(0 = \pi_k c_k'(a^{k,*}) - \eta_k \int_{\mathbb R} \lambda_k(x,a^{k,*})Q_k^U(x)\,\mu^k(dx),\)
which is \eqref{eq:optimal-control-foc}. On the constrained interval \([0,\bar a]\), first-order optimality together with endpoint complementarity is equivalent to the projected form \eqref{eq:projected-foc-control}. If \(c_k(a)=\frac12r_ka^2\), then the population-weighted derivative is \(\pi_k r_ka\), giving \eqref{eq:quadratic-cost-control}. The proof is complete.
\end{proof}

\begin{proposition*}[Detailed form of Proposition~\ref{prop:controlled-multiclass-wellposedness}: Well-posedness of the controlled multi-population contagion system]
Assume the following controlled-contagion regularity conditions. More specifically, the drift term
\[
  b_k:\mathbb R\times\mathcal M_{\le 1}(\mathbb R)^K\times A_k\to\mathbb R,
  \qquad k=1,\ldots,K,
\]
is uniformly Lipschitz in the state and measure variables and satisfies a linear growth condition; the diffusion coefficients
\(\sigma_k,\sigma_0\) are bounded; the control sets \(A_k\) are compact; the contagion matrix
 \(\Gamma=(\Gamma_{k\ell})_{1\le k,\ell\le K}\)
is nonnegative and bounded; and the kill intensity \(\lambda_k(x,a)\) is nonnegative, bounded, and uniformly Lipschitz in \(x\), namely there exist constants
 \(\Lambda<\infty,\qquad L_\lambda<\infty,\)
such that for every \(x,y\in\mathbb R\) and \(a\in A_k\),
\[
  0\le \lambda_k(x,a)\le \Lambda,
  \qquad
  |\lambda_k(x,a)-\lambda_k(y,a)|
  \le L_\lambda |x-y|.
\]
Assume that the initial states and noises are mutually independent and satisfy
 \(\max_{1\le k\le K}\mathbb E|X_1^k(0)|^2<\infty .\)
Admissible controls in this paper are processes \(a=(a^1,\ldots,a^K)\) taking values in
\(A_1\times\cdots\times A_K\) and progressively measurable with respect to the common-noise filtration
\(\mathbb F^0=(\mathcal F_t^{W^0})_{t\in[0,T]}\).

Then the following statements hold.

 \(\text{\rm (i)}\)
For any finite particle size \((N_1,\ldots,N_K)\), the controlled multi-population particle system admits a unique strong solution. More precisely, there exists a unique càdlàg process adapted to the filtration generated by the given Brownian motions and Poisson random measures,
 \(\bigl(X_i^{k,N}(t),\xi_i^{k,N}(t)\bigr)_ {1\le i\le N_k,\ 1\le k\le K,\ 0\le t\le T},\)
satisfying
\[
dX_i^{k,N}(t)
=
b_k\bigl(
X_i^{k,N}(t),
\mu_t^{1,N},\ldots,\mu_t^{K,N},
a_t^k
\bigr)\,dt
+
\sigma_k\,dW_i^k(t)
+
\sigma_0\,dW^0(t)
-
\sum_{\ell=1}^K
\Gamma_{k\ell}\,dL_t^{\ell,N},
\]
where
\[
  \mu_t^{k,N}
  =
  \frac1{N_k}
  \sum_{i=1}^{N_k}
  \xi_i^{k,N}(t)\delta_{X_i^{k,N}(t)},
  \qquad
  L_t^{k,N}
  =
  \frac1{N_k}
  \sum_{i=1}^{N_k}
  \bigl(1-\xi_i^{k,N}(t)\bigr).
\]
If \(\mathcal N_i^k(ds,du)\) are independent Poisson random measures on
\([0,\infty)\times[0,\Lambda]\) with intensity \(ds\,du\), then the killing indicator is given by
\[
  \xi_i^{k,N}(t)
  =
  1-
  \int_0^t\int_0^\Lambda
  \xi_i^{k,N}(s-)
  \mathbf 1_{\{u\le \lambda_k(X_i^{k,N}(s-),a_s^k)\}}
  \mathcal N_i^k(ds,du).
\]

 \(\text{\rm (ii)}\)
The controlled McKean--Vlasov limiting system admits a unique solution. That is, there exist a unique \(\mathbb F^0\)-progressively measurable conditional measure flow and loss flow
 \(\bigl(\mu_t^k,L_t^k\bigr)_{1\le k\le K,\ 0\le t\le T}\)
and representative particle processes \((\bar X^k,\bar\xi^k)_{1\le k\le K}\) such that
\[
d\bar X^k(t)
=
b_k\bigl(
\bar X^k(t),
\mu_t^1,\ldots,\mu_t^K,
a_t^k
\bigr)\,dt
+
\sigma_k\,dW^k(t)
+
\sigma_0\,dW^0(t)
-
\sum_{\ell=1}^K
\Gamma_{k\ell}\,dL_t^\ell ,
\]
and
\[
  \mu_t^k
  =
  \mathbb E
  \left[
  \bar\xi^k(t)\delta_{\bar X^k(t)}
  \,\big|\,
  \mathcal F_t^{W^0}
  \right],
  \qquad
  L_t^k
  =
  1-
  \mathbb E
  \left[
  \bar\xi^k(t)
  \,\big|\,
  \mathcal F_t^{W^0}
  \right].
\]
The killing indicator of the representative particle is defined by
\[
  \bar\xi^k(t)
  =
  1-
  \int_0^t\int_0^\Lambda
  \bar\xi^k(s-)
  \mathbf 1_{\{u\le \lambda_k(\bar X^k(s-),a_s^k)\}}
  \mathcal N^k(ds,du).
\]

 \(\text{\rm (iii)}\)
For each \(k=1,\ldots,K\), the limiting loss process \(L^k\) is continuous and nondecreasing on \([0,T]\). In fact,
\[
  L_t^k
  =
  \int_0^t
  \mathbb E
  \left[
  \lambda_k(\bar X^k(s),a_s^k)\bar\xi^k(s)
  \,\big|\,
  \mathcal F_s^{W^0}
  \right]ds,
  \qquad 0\le t\le T.
\]
Therefore \(L^k\) is absolutely continuous, and
 \(0\le L_t^k-L_s^k\le \Lambda(t-s), \qquad 0\le s\le t\le T .\)
The finite-particle loss \(L^{k,N}\) is a càdlàg nondecreasing step process.
\end{proposition*}

\begin{proof}
We first prove strong well-posedness of the finite-particle system. Fix a particle size
\((N_1,\ldots,N_K)\), and place all Brownian motions and Poisson random measures on the same complete filtered probability space. Since each particle can be killed at most once and the kill intensity satisfies
 \(0\le \lambda_k(x,a)\le \Lambda,\)
the total killing jump intensity in the finite system is uniformly controlled by
 \(\Lambda\sum_{k=1}^K N_k .\)
Thus the total number of killing jumps is almost surely finite on any finite time interval. Given a candidate path, the Poisson thinning formula uniquely determines the killing indicators; given the killing indicators, the state equation is an SDE with Lipschitz coefficients in finite dimension and with the finite-variation contagion term
 \(-\sum_{\ell=1}^K\Gamma_{k\ell}\,dL_t^{\ell,N}.\)
The standard piecewise construction gives a unique strong solution between consecutive killing jump times. At a killing jump time, the jump of the loss process uniquely determines the contagion shock to all state variables. Recursively concatenating these intervals gives existence of a strong solution to the finite-particle system.

We now prove pathwise uniqueness. Let
 \((X_i^{k,N},\xi_i^{k,N}) \quad\text{and}\quad (\widetilde X_i^{k,N},\widetilde\xi_i^{k,N})\)
be two solutions driven by the same initial states, Brownian motions, and Poisson random measures. Denote the corresponding empirical measures and loss processes by
 \((\mu_t^{k,N},L_t^{k,N}), \qquad (\widetilde\mu_t^{k,N},\widetilde L_t^{k,N}).\)
For any test function \(f\) satisfying
 \(\|f\|_\infty\le 1, \qquad \operatorname{Lip}(f)\le 1,\)
use the decomposition
\[
\begin{aligned}
\xi_i^{k,N}(t)f(X_i^{k,N}(t))
-
\widetilde\xi_i^{k,N}(t)f(\widetilde X_i^{k,N}(t))
&=
\xi_i^{k,N}(t)
\Bigl[
f(X_i^{k,N}(t))-f(\widetilde X_i^{k,N}(t))
\Bigr] \\
&\quad+
\Bigl[
\xi_i^{k,N}(t)-\widetilde\xi_i^{k,N}(t)
\Bigr]
f(\widetilde X_i^{k,N}(t)).
\end{aligned}
\]
Using \(\xi_i^{k,N}(t)\in\{0,1\}\), \(\|f\|_\infty\le 1\), and
\(\operatorname{Lip}(f)\le1\), we obtain
\[
\left|
\xi_i^{k,N}(t)f(X_i^{k,N}(t))
-
\widetilde\xi_i^{k,N}(t)f(\widetilde X_i^{k,N}(t))
\right|
\le
|X_i^{k,N}(t)-\widetilde X_i^{k,N}(t)|
+
|\xi_i^{k,N}(t)-\widetilde\xi_i^{k,N}(t)|.
\]
Hence, after taking the supremum over bounded Lipschitz test functions,
\[
  d_{\mathrm{BL}}(\mu_t^{k,N},\widetilde\mu_t^{k,N})
  \le
  \frac1{N_k}
  \sum_{i=1}^{N_k}
  |X_i^{k,N}(t)-\widetilde X_i^{k,N}(t)|
  +
  \frac1{N_k}
  \sum_{i=1}^{N_k}
  |\xi_i^{k,N}(t)-\widetilde\xi_i^{k,N}(t)| .
\]
The second term is the individual alive--dead mismatch. It cannot be replaced by the aggregate loss difference alone, since positive and negative mismatches in the two systems may cancel in the type average. The loss difference is controlled by the same nonnegative mismatch term:
\[
  |L_t^{k,N}-\widetilde L_t^{k,N}|
  \le
  \frac1{N_k}
  \sum_{i=1}^{N_k}
  |\xi_i^{k,N}(t)-\widetilde\xi_i^{k,N}(t)| .
\]
By synchronous Poisson thinning, the difference in killing indicators can only come from the difference between the two kill-intensity paths. Thus
\[
\begin{aligned}
\mathbb E
\sup_{r\le t}
|\xi_i^{k,N}(r)-\widetilde\xi_i^{k,N}(r)|
&\le
\mathbb E
\int_0^t
\left|
\lambda_k(X_i^{k,N}(s),a_s^k)
-
\lambda_k(\widetilde X_i^{k,N}(s),a_s^k)
\right|ds \\
&\le
L_\lambda
\int_0^t
\mathbb E
|X_i^{k,N}(s)-\widetilde X_i^{k,N}(s)|ds .
\end{aligned}
\]
Subtract the two state equations. The Brownian terms cancel, giving
\[
\begin{aligned}
|X_i^{k,N}(t)-\widetilde X_i^{k,N}(t)|
&\le
\int_0^t
\left|
\begin{aligned}
&b_k(X_i^{k,N}(s),\mu_s^{1,N},\ldots,\mu_s^{K,N},a_s^k)\\
&\quad-
b_k(\widetilde X_i^{k,N}(s),\widetilde\mu_s^{1,N},\ldots,\widetilde\mu_s^{K,N},a_s^k)
\end{aligned}
\right|ds \\
&\quad+
\sum_{\ell=1}^K
\Gamma_{k\ell}
|L_t^{\ell,N}-\widetilde L_t^{\ell,N}|.
\end{aligned}
\]
For each type set
\[
  A_k(t):=
  \frac1{N_k}
  \sum_{i=1}^{N_k}
  \mathbb E\sup_{r\le t}
  |X_i^{k,N}(r)-\widetilde X_i^{k,N}(r)|,
\]
and
\[
  B_k(t):=
  \frac1{N_k}
  \sum_{i=1}^{N_k}
  \mathbb E\sup_{r\le t}
  |\xi_i^{k,N}(r)-\widetilde\xi_i^{k,N}(r)| .
\]
The preceding estimates imply
 \(B_k(t) \le L_\lambda\int_0^t A_k(s)\,ds,\)
while the Lipschitz property of the drift and the boundedness of \(\Gamma\), together with the bounds on \(d_{\mathrm{BL}}\) and the loss process above, give
\[
  A_k(t)
  \le
  C\int_0^t A_k(s)\,ds
  +C\int_0^t\sum_{\ell=1}^K\bigl(A_\ell(s)+B_\ell(s)\bigr)\,ds
  +C\sum_{\ell=1}^K B_\ell(t).
\]
Let
 \(\Delta_N(t):= \sum_{k=1}^K\bigl(A_k(t)+B_k(t)\bigr).\)
Since \(B_k(t)\le C\int_0^t A_k(s)\,ds\), summing over \(k\) yields
 \(\Delta_N(t) \le C\int_0^t\Delta_N(s)\,ds .\)
By Gronwall's inequality,
 \(\Delta_N(t)=0, \qquad 0\le t\le T .\)
Thus the state processes and killing indicators of the two solutions are indistinguishable. The empirical measures and loss processes are consequently indistinguishable as well, so the finite-particle system has pathwise uniqueness. Combined with the piecewise construction above, this proves existence and uniqueness of a strong solution.

We next prove well-posedness of the McKean--Vlasov limiting system. Let \(\mathcal X_T\) be the space of all
 \((\boldsymbol\nu,\boldsymbol\ell) = \bigl((\nu^1,\ldots,\nu^K),(\ell^1,\ldots,\ell^K)\bigr)\)
such that \(\nu_t^k\) is an \(\mathcal M_{\le1}(\mathbb R)\)-valued
\(\mathbb F^0\)-progressively measurable flow, \(\ell_t^k\) is a \([0,1]\)-valued
\(\mathbb F^0\)-progressively measurable continuous flow, and
 \(\ell_0^k=0, \qquad 0\le \ell_s^k\le \ell_t^k\le 1, \qquad 0\le s\le t\le T .\)
On \(\mathcal X_T\), define the seminorm
\[
  \rho_t
  \bigl(
  (\boldsymbol\nu,\boldsymbol\ell),
  (\widetilde{\boldsymbol\nu},\widetilde{\boldsymbol\ell})
  \bigr)
  :=
  \mathbb E
  \left[
  \sum_{k=1}^K
  \sup_{0\le s\le t}
  d_{\mathrm{BL}}(\nu_s^k,\widetilde\nu_s^k)
  +
  \sum_{k=1}^K
  \sup_{0\le s\le t}
  |\ell_s^k-\widetilde\ell_s^k|
  \right].
\]
The initial second-moment condition ensures that the solution of the linearized equation below has finite second moments, so the conditional measure flow
 \(\mathbb E \left[ \zeta^k(t)\delta_{Y^k(t)} \,\big|\, \mathcal F_t^{W^0} \right]\)
is well defined and remains in the space of sub-probability measures. Specifically, the linear growth of the drift, \(\ell^k\in[0,1]\), boundedness of \(\Gamma\), and bounded diffusion coefficients jointly imply
\[
  \mathbb E\sup_{0\le t\le T}|Y^k(t)|^2
  \le
  C_T
  \left(
  1+\mathbb E|X_1^k(0)|^2
  \right)
  <\infty .
\]

For any candidate flow \((\boldsymbol\nu,\boldsymbol\ell)\in\mathcal X_T\), define the linearized representative particle by
\[
dY^k(t)
=
b_k\bigl(
Y^k(t),
\nu_t^1,\ldots,\nu_t^K,
a_t^k
\bigr)\,dt
+
\sigma_k\,dW^k(t)
+
\sigma_0\,dW^0(t)
-
\sum_{\ell=1}^K
\Gamma_{k\ell}\,d\ell_t^\ell ,
\]
and define
\[
  \zeta^k(t)
  =
  1-
  \int_0^t\int_0^\Lambda
  \zeta^k(s-)
  \mathbf 1_{\{u\le \lambda_k(Y^k(s-),a_s^k)\}}
  \mathcal N^k(ds,du)
\]
by Poisson thinning. For a fixed candidate flow, this is an SDE with Lipschitz coefficients and a given finite-variation input, and hence has a unique strong solution. Define the map \(\mathcal T\) by
 \(\mathcal T(\boldsymbol\nu,\boldsymbol\ell) = (\boldsymbol\mu,\boldsymbol L),\)
where
\[
  \mu_t^k
  =
  \mathbb E
  \left[
  \zeta^k(t)\delta_{Y^k(t)}
  \,\big|\,
  \mathcal F_t^{W^0}
  \right],
  \qquad
  L_t^k
  =
  1-
  \mathbb E
  \left[
  \zeta^k(t)
  \,\big|\,
  \mathcal F_t^{W^0}
  \right].
\]
Since \(Y^k\), \(\zeta^k\), and \(a^k\) are constructed conditionally on the common noise, \(\mu^k\) and \(L^k\) are \(\mathbb F^0\)-progressively measurable. By boundedness of the kill intensity and the compensation formula,
\[
  L_t^k
  =
  \int_0^t
  \mathbb E
  \left[
  \lambda_k(Y^k(s),a_s^k)\zeta^k(s)
  \,\big|\,
  \mathcal F_s^{W^0}
  \right]ds,
\]
so \(L^k\) is continuous, nondecreasing, and satisfies
 \(0\le L_t^k-L_s^k\le \Lambda(t-s).\)
Thus \(\mathcal T\) maps \(\mathcal X_T\) into itself.

We now prove that \(\mathcal T\) is a contraction on a short time interval. Take two candidate flows
\[
  Z=(\boldsymbol\nu,\boldsymbol\ell),
  \qquad
  \widetilde Z=(\widetilde{\boldsymbol\nu},\widetilde{\boldsymbol\ell}),
\]
with corresponding linearized solutions
\[
  (Y^k,\zeta^k,\mu^k,L^k),
  \qquad
  (\widetilde Y^k,\widetilde\zeta^k,\widetilde\mu^k,\widetilde L^k).
\]
Couple the two systems with the same initial state, Brownian motions, and Poisson random measures. As in the finite-particle part, the state difference satisfies
\[
\begin{aligned}
\mathbb E\sup_{r\le t}
|Y^k(r)-\widetilde Y^k(r)|
&\le
C\int_0^t
\mathbb E\sup_{u\le s}
|Y^k(u)-\widetilde Y^k(u)|ds \\
&\quad+
C\int_0^t
\sum_{\ell=1}^K
\mathbb E\sup_{u\le s}
d_{\mathrm{BL}}(\nu_u^\ell,\widetilde\nu_u^\ell)ds
+
C\sum_{\ell=1}^K
\mathbb E\sup_{r\le t}
|\ell_r^\ell-\widetilde\ell_r^\ell|.
\end{aligned}
\]
Synchronous Poisson thinning and the Lipschitz property of \(\lambda_k\) give
\[
  \mathbb E\sup_{r\le t}
  |\zeta^k(r)-\widetilde\zeta^k(r)|
  \le
  C\int_0^t
  \mathbb E
  |Y^k(s)-\widetilde Y^k(s)|ds .
\]
Using the \(d_{\mathrm{BL}}\) decomposition above, the output measure and output loss satisfy
\[
\begin{aligned}
\sup_{r\le t}
d_{\mathrm{BL}}(\mu_r^k,\widetilde\mu_r^k)
+
\sup_{r\le t}
|L_r^k-\widetilde L_r^k|
&\le
C
\mathbb E\sup_{r\le t}
|Y^k(r)-\widetilde Y^k(r)| \\
&\quad+
C
\mathbb E\sup_{r\le t}
|\zeta^k(r)-\widetilde\zeta^k(r)|.
\end{aligned}
\]
Combining the estimates and summing over \(k\) yields
 \(R(t) \le C t D(t) + C\int_0^t R(s)\,ds,\)
where
\[
  R(t)
  :=
  \rho_t
  \bigl(
  \mathcal T Z,
  \mathcal T\widetilde Z
  \bigr),
  \qquad
  D(t)
  :=
  \rho_t
  \bigl(
  Z,
  \widetilde Z
  \bigr).
\]
Since \(D(s)\le D(t)\) for \(s\le t\), this is a Volterra-type Gronwall inequality for \(R\). Fixing \(t\le T\) and applying Gronwall's inequality to \(R\), we get
 \(R(t) \le C t D(t)\exp(Ct).\)
Thus
\[
  \rho_t
  \bigl(
  \mathcal T Z,
  \mathcal T\widetilde Z
  \bigr)
  \le
  C t e^{Ct}
  \rho_t
  \bigl(
  Z,
  \widetilde Z
  \bigr).
\]
Choose \(\delta>0\) sufficiently small so that
 \(C\delta e^{C\delta}<1.\)
Then \(\mathcal T\) is a contraction on \([0,\delta]\). Banach's fixed point theorem gives a unique fixed point on \([0,\delta]\). Taking the conditional distribution induced by this fixed point at time \(\delta\) as the new initial condition, and repeating the same argument on
 \([\delta,2\delta],\quad [2\delta,3\delta],\quad \ldots,\)
we obtain a unique solution on \([0,T]\) after finitely many concatenations. The short-time contraction constant depends only on
 \(T,\ L_b,\ L_\lambda,\ \Lambda,\ \|\Gamma\|,\ \max_k \mathbb E|X_1^k(0)|^2,\)
and the upper bound on the diffusion coefficients, so the estimates remain uniform at every concatenation step.

Finally, we prove continuity and monotonicity of the limiting loss process. By the fixed point definition and the compensation formula for Poisson thinning,
\[
  L_t^k
  =
  \int_0^t
  \mathbb E
  \left[
  \lambda_k(\bar X^k(s),a_s^k)\bar\xi^k(s)
  \,\big|\,
  \mathcal F_s^{W^0}
  \right]ds .
\]
Since
 \(0\le \lambda_k(\bar X^k(s),a_s^k)\bar\xi^k(s) \le \Lambda,\)
the right-hand side is the time integral of a nonnegative bounded process. Hence \(L^k\) is absolutely continuous and therefore continuous. For any
\(0\le s\le t\le T\),
\[
  L_t^k-L_s^k
  =
  \int_s^t
  \mathbb E
  \left[
  \lambda_k(\bar X^k(r),a_r^k)\bar\xi^k(r)
  \,\big|\,
  \mathcal F_r^{W^0}
  \right]dr
  \ge 0 .
\]
Also,
 \(L_t^k-L_s^k\le \Lambda(t-s).\)
Thus \(L^k\) is continuous and nondecreasing. The finite-particle loss process
 \(L_t^{k,N} = \frac1{N_k}\sum_{i=1}^{N_k} \mathbf 1_{\{\tau_i^{k,N}\le t\}}\)
is the average of finitely many killing indicators, and is therefore a càdlàg nondecreasing step process. The proof is complete.
\end{proof}

\section{Discussion and conclusion}
\label{sec:discussion-conclusion}

The paper develops a controlled killed-contagion framework in which multi-population McKean--Vlasov limits, common noise, state-dependent killing, and dynamic programming can be handled in one system. The mean-field theorem gives the finite-particle foundation, and the quantitative rate explains the observed first-order particle scale in the convergence diagnostics. The controlled well-posedness result then supplies the state process required for the optimization problem.

The main control result is the two-population killed HJB comparison principle. Its state decomposition separates alive measures from cemetery masses, allowing the Wasserstein smooth-gauge comparison argument to be combined with a finite-dimensional treatment of the killing jump. The numerical HJB loop illustrates the corresponding marginal-value structure in a compact finite-dimensional projection. The resulting feedback is not a substitute for the measure-valued theorem, but it shows that the Hamiltonian structure generates an implementable intervention rule with direct forward risk reduction.

The steep-killing bridge clarifies how the bounded state-dependent killing model relates to absorbing-boundary contagion. The bounded killed model is the object used for well-posed control and HJB comparison, while the steep-intensity limit identifies the hard-default model selected under localized boundary regularity and absorbing fixed-point stability. Extending the comparison principle beyond two populations, weakening the nondegeneracy assumptions near the boundary, and deriving scalable feedback classes for high-dimensional heterogeneous systems are natural next steps.

\section*{Code and data availability}
\label{sec:code-data-availability}
Replication code is available at \url{https://github.com/Hopper-221/systemic_mkv}. The repository contains the forward simulation, finite-dimensional HJB feedback, and figure/table generation scripts used for the numerical experiments.

\newpage
\appendix
\section{Numerical discretization and parameter calibration}
\label{app:numerical-calibration}

This appendix records the finite-dimensional numerical procedure used only for the two-population HJB feedback diagnostic in Section~\ref{sec:hjb-numerics}. The HJB in Theorem~\ref{thm:two-pop-optimal-control-hjb-common-noise} is projected onto $(x_1,x_2,L_1,L_2)$ and solved by backward induction from $T$ to $0$. The first-order terms in the $x$ directions use central differences, the pure second-order diffusion terms use semi-implicit tridiagonal steps, and the common-noise cross term $\sigma_0^2\partial_{x_1x_2}^2V$ is retained explicitly. Controls are searched on a finite grid in $[0,\bar a]^2$. The killing contribution is implemented as a one-way semi-Lagrangian jump in the loss coordinates, using monotone piecewise-linear interpolation for $V(x_1,x_2,L_1+\Delta L_1,L_2)-V(x_1,x_2,L_1,L_2)$ and its type-2 analogue.

The $K=2$ experiments are calibrated to a near-critical region in which the uncontrolled system has material cascade risk, HJB feedback is nonzero, and HJB and rule-based controls have comparable cost magnitudes. The final calibration uses kill-intensity multiplier $25$, terminal-loss weight multiplier $6$, and running-cost multiplier $0.05$. The finite-difference computations were implemented as CPU-based NumPy/SciPy workloads under Python 3.12.7 with NumPy 1.26.4, SciPy 1.13.1, pandas 2.2.2, and Matplotlib 3.9.2.

\paragraph{Algorithm A.1 ($K=2$ finite-dimensional HJB feedback and forward validation).}
Set $V(T,x_1,x_2,L_1,L_2)=\Phi(L_1,L_2)$; apply backward induction with semi-implicit alive-state diffusion, explicit common-noise cross term, and one-way semi-Lagrangian loss jumps; minimize the discrete Hamiltonian over $[0,\bar a]^2$ at each grid point; interpolate the stored feedback in the forward $K=2$ particle system; and report $V_0$, initial control, adjacent-grid control-field changes, terminal tail loss, cascade probability, and intervention cost.

\section{Technical verification for the steep-killing bridge}
\label{app:steep-killing-primitive}

The main text states Proposition~\ref{prop:killing-to-absorbing-boundary} as the singular-limit bridge from bounded regularized killing to absorbing-boundary default. The proof uses a localized boundary-regularity package and a sequence of stability lemmas. Since these ingredients are technical and auxiliary to the killed-HJB comparison theorem, the detailed primitive verification is collected here.

\subsection{Primitive verification of the steep-killing bridge}
\label{subsec:primitive-steep-bridge}

This subsection gives a primitive verification of the steep-killing bridge. The argument derives eventual entrance of the steep-killing fixed points into a neighborhood of the absorbing fixed point, local stability of the absorbing fixed point, and uniform convergence of the driven steep-killing maps to the absorbing-boundary map from a localized regularity package: compact localization, a uniformly nondegenerate idiosyncratic diffusion coefficient near the boundary, uniform one-dimensional killed-diffusion estimates, and uniqueness of the absorbing-boundary McKean--Vlasov fixed point in the localized class.

Throughout, the common noise is retained. All estimates below may be read conditionally on the common-noise path and then integrated over the common noise. After conditioning on the common noise, the term \(\sigma_0 dW^0_t\) is a given continuous forcing path and does not provide conditional smoothing. Hence the nondegeneracy assumption is placed on the idiosyncratic coefficient \(\sigma_k\), not on \(\sigma_k^2+\sigma_0^2\).

\subsection*{State space, frozen maps, and metrics}

Let \(\mathsf E=\mathbb R\cup\{\partial\}\) be the alive--dead state space, where \(\partial\) is a cemetery point. Equip \(\mathsf E\) with any bounded metric \(d_{\partial}\) that agrees with \(|x-y|\wedge1\) on compact alive-state intervals and puts \(\partial\) at positive distance from the alive region. For finite measures on \(\mathsf E\), let \(d_{\mathrm{BL}}\) denote the bounded-Lipschitz distance associated with \(d_{\partial}\).

A measure--loss input is denoted by
 \(z=(\widehat\mu^1,\ldots,\widehat\mu^K,L^1,\ldots,L^K),\)
where \(\widehat\mu^k_t\) is a probability measure on \(\mathsf E\), its mass at \(\partial\) equals \(L^k_t\), and \(L^k\) is nondecreasing with values in \([0,1]\). For \(t\le T\), set
\[
\mathcal D_t(z,z'):=
  \sum_{k=1}^K
  \mathbb E\left[
  \sup_{r\le t}d_{\mathrm{BL}}(\widehat\mu^k_r,\widehat\mu^{\prime k}_r)
  +
  \sup_{r\le t}|L^k_r-L^{\prime k}_r|
  \right].
\]
The expectation includes the common noise. The full-horizon metric is \(\mathcal D_T\).

For a frozen input \(z\), the driven type-\(k\) state process is written as \(X^{k,z}\). In the localized formulation used below, it solves, up to the localization boundary,
\[
\begin{aligned}
  dX^{k,z}_t
  &=
  b_k\bigl(t,X^{k,z}_t,z_t,a^k_t\bigr)\,dt
  +\sigma_k\,dW^k_t+\sigma_0\,dW^0_t
  -\sum_{\ell=1}^K\Gamma_{k\ell}\,dL^\ell_t .
\end{aligned}
\]
The coefficients are frozen in \(z\); default is then generated from the driven path. Let
 \(\bar\lambda_k(a):=\lambda_{\mathrm{base},k}\exp\{-\eta_k a\}.\)
For \(n\ge1\), define
 \(\lambda_k^{(n)}(x,a) :=\bar\lambda_k(a)\exp\{n(x_{b,k}-x)^+\}.\)
Given a unit exponential random variable \(E_k\), independent of the Brownian noises and common noise, define the steep-killing clock
\[
  A_t^{k,n,z}:=
  \int_0^t
  \bar\lambda_k(a^k_s)
  \exp\{n(x_{b,k}-X^{k,z}_s)^+\}\,ds,
  \qquad
  \tau_k^{n,z}:=\inf\{t:A_t^{k,n,z}\ge E_k\}.
\]
The baseline Cox clock and absorbing hitting time are
\[
  A_t^{k,0}:=\int_0^t\bar\lambda_k(a^k_s)\,ds,
  \qquad
  \tau_k^{\mathrm{base}}:=\inf\{t:A_t^{k,0}\ge E_k\},
  \qquad
  H_k^z:=\inf\{t:X^{k,z}_t\le x_{b,k}\}.
\]
The absorbing-limit default time is
 \(\tau_k^{\mathrm{abs},z}:=\tau_k^{\mathrm{base}}\wedge H_k^z .\)
For \(\star\in\{n,\mathrm{abs}\}\), the map \(\Phi^\star\) sends \(z\) to the lifted conditional laws and losses generated by the same driven state path and the corresponding default time:
\[
  \widehat\mu^{k,\star,z}_t
  :=\mathcal L\bigl(X^{k,z}_t\mathbf 1_{\{\tau_k^{\star,z}>t\}}+
  \partial\mathbf 1_{\{\tau_k^{\star,z}\le t\}}\mid\mathcal F_t^{W^0}\bigr),
  \qquad
  L^{k,\star,z}_t:=\mathbb P(\tau_k^{\star,z}\le t\mid\mathcal F_t^{W^0}).
\]
Thus \(\Phi^n(z)=(\widehat\mu^{1,n,z},\ldots,L^{K,n,z})\) and \(\Phi^{\mathrm{abs}}(z)=(\widehat\mu^{1,\mathrm{abs},z},\ldots,L^{K,\mathrm{abs},z})\).

\subsection*{Primitive regularity package}

\begin{assumption}[Localized primitive regularity]
\label{ass:primitive-package}
The following conditions hold on \([0,T]\).

\begin{enumerate}
\item[\textup{(P1)}] \textbf{Compact localization.} For each type \(k\), alive states are localized in a compact interval \(I_k=[\underline x_k,\overline x_k]\), with \(x_{b,k}\in(\underline x_k,\overline x_k)\). Equivalently, the original unbounded model is stopped at a large compact interval and all estimates below are applied before the localization error is removed.

\item[\textup{(P2)}] \textbf{Bounded Lipschitz frozen coefficients.} The drift \(b_k(t,x,z_t,a)\) is bounded and Lipschitz in \((x,z_t)\), uniformly in \((t,a)\). The controls take values in a compact set. There are constants \(0<\underline\lambda\le\overline\lambda<\infty\) such that
 \(\underline\lambda\le \bar\lambda_k(a^k_t)\le\overline\lambda \qquad\text{for all }k,t.\)
The lower bound is satisfied by the exponential regulatory-control specification on any compact control set.

\item[\textup{(P3)}] \textbf{Conditional nondegeneracy.} The idiosyncratic volatility is uniformly positive near the boundary:
 \(\sigma_k\ge\underline\sigma>0, \qquad k=1,\ldots,K.\)
This condition is imposed on \(\sigma_k\), rather than on \(\sigma_k^2+\sigma_0^2\), because the arguments are conditional on the common noise.

\item[\textup{(P4)}] \textbf{Uniform killed-diffusion estimates.} For every compact class of frozen inputs satisfying the moment and modulus bounds used below, the one-dimensional diffusion killed at \(x_{b,k}\) has transition densities and survival probabilities satisfying the standard Aronson and boundary-gradient bounds, uniformly over that class. In particular, for bounded measurable \(f\), the killed semigroup \(P^{k,z}_{s,t}f(x):=\mathbb E_{s,x}[f(X^{k,z}_t)\mathbf 1_{\{H_k^z>t\}}]\) satisfies
 \(\|\partial_x P^{k,z}_{s,t}f\|_\infty \le C(t-s)^{-1/2}\|f\|_\infty, \qquad 0\le s<t\le T,\)
with a constant \(C\) independent of \(z\) in the class. The same estimates hold for the baseline-Cox killed semigroup, since the baseline killing rate is bounded.

\item[\textup{(P5)}] \textbf{Uniform absorbing no-atom modulus.} For each compact frozen-input class considered below, the absorbing default times have a uniform distributional modulus: there is a deterministic modulus \(\omega_{\mathrm{abs}}(h)\downarrow0\) as \(h\downarrow0\) such that
\[
  \sup_z\sup_{0\le t\le t+h\le T}
  \mathbb P\bigl(t<\tau_k^{\mathrm{abs},z}\le t+h\bigr)
  \le \omega_{\mathrm{abs}}(h).
\]
This condition is required only for the absorbing default time. No uniform-in-\(n\) no-atom modulus is imposed on \(\tau_k^{n,z}\), since the steep intensity has a Lipschitz and size constant that grows exponentially in \(n\) inside the localized interval. Compactness of the steep fixed-point loss paths will instead use monotonicity, Helly selection, and the continuity of the absorbing limiting loss path.

\item[\textup{(P6)}] \textbf{Absorbing fixed-point class.} The localized absorbing-boundary McKean--Vlasov map \(\Phi^{\mathrm{abs}}\) is considered on the compact class generated by (P1)--(P5). Its fixed-point uniqueness is not imposed as a separate black-box assumption; it is verified below from the Volterra stability estimate and a short-time contraction--concatenation argument.
\end{enumerate}
\end{assumption}

\begin{remark}[Why only the absorbing modulus is imposed]
The absorbing no-atom modulus in (P5) is the regularity input that controls continuity of the boundary-default map at the level of loss paths. This condition follows from one-dimensional nondegenerate killed-diffusion density estimates and bounded baseline Cox killing. It is imposed only on the absorbing default time and should not be required for \(\tau_k^{n,z}\) uniformly in \(n\): on a compact interval the bound on \(\lambda_k^{(n)}\) is of order \(\exp(Cn)\), so a direct density bound for the steep default time would degenerate with \(n\). The proof below avoids such a false uniform modulus. Uniform convergence of the maps is obtained from the post-hit occupation and clock-mismatch estimate, while compactness of the fixed-point loss paths uses Helly selection and the eventual identification of every limit as the continuous absorbing fixed point.
\end{remark}

\subsection*{Step 1: pointwise steep-clock convergence}

\begin{lemma}[Pointwise convergence of steep killing clocks]
\label{lem:pointwise-clock}
Assume (P1)--(P3). Fix a frozen input \(z\) and a type \(k\). If the driven path is nonsticky at \(x_{b,k}\), in the sense that on \(\{H_k^z<T\}\) it spends positive Lebesgue time below \(x_{b,k}\) on every right-neighborhood of \(H_k^z\), then
 \(\tau_k^{n,z}\downarrow \tau_k^{\mathrm{abs},z} \qquad\text{a.s.}\)
Consequently,
\[
  \mathbb E\sup_{t\le T}
  \left|\mathbf 1_{\{\tau_k^{n,z}\le t\}}-
  \mathbf 1_{\{\tau_k^{\mathrm{abs},z}\le t\}}\right|
  \longrightarrow0.
\]
\end{lemma}

\begin{proof}
The functions \(x\mapsto \exp\{n(x_{b,k}-x)^+\}\) increase pointwise in \(n\), hence \(A_t^{k,n,z}\) increases in \(n\) for every \(t\), and \(\tau_k^{n,z}\) decreases in \(n\). Also \(A_t^{k,n,z}\ge A_t^{k,0}\), so \(\tau_k^{n,z}\le \tau_k^{\mathrm{base}}\). The cumulative hazards agree before the boundary is reached: \(A_t^{k,n,z}=A_t^{k,0}\) for every \(t<H_k^z\), since \((x_{b,k}-X_s^{k,z})^+=0\) whenever \(X_s^{k,z}>x_{b,k}\). Therefore no steep clock can default before \(\tau_k^{\mathrm{base}}\wedge H_k^z\). Thus
 \(\tau_k^{\mathrm{abs},z}\le \tau_k^{n,z}\le\tau_k^{\mathrm{base}}.\)
If \(\tau_k^{\mathrm{base}}<H_k^z\), then the baseline clock rings before the path reaches the boundary, and all clocks give the same default time. If \(H_k^z<\tau_k^{\mathrm{base}}\), then after \(H_k^z\) the path spends positive Lebesgue time below the boundary on every interval \([H_k^z,H_k^z+\delta]\). Since \(\bar\lambda_k\ge\underline\lambda>0\), for every \(\delta>0\),
\[
  \int_{H_k^z}^{(H_k^z+\delta)\wedge T}
  \bar\lambda_k(a^k_s)
  \exp\{n(x_{b,k}-X^{k,z}_s)^+\}\,ds
  \longrightarrow \infty
\]
by monotone convergence on a set of positive Lebesgue measure where \(X^{k,z}_s<x_{b,k}\). Hence the steep clock rings before \(H_k^z+\delta\) for all sufficiently large \(n\). Letting \(\delta\downarrow0\) gives \(\tau_k^{n,z}\downarrow H_k^z=\tau_k^{\mathrm{abs},z}\) on this event. The remaining case \(H_k^z=\tau_k^{\mathrm{base}}\) is negligible under the no-atom condition; even without removing it, the monotone bounds force the same limit. Finally,
\[
  \sup_{t\le T}
  \left|\mathbf 1_{\{\tau_k^{n,z}\le t\}}-
  \mathbf 1_{\{\tau_k^{\mathrm{abs},z}\le t\}}\right|
  =\mathbf 1_{\{\tau_k^{\mathrm{abs},z}<\tau_k^{n,z}\}},
\]
and the right-hand side decreases to zero. Dominated convergence completes the proof.
\end{proof}

\subsection*{Step 2: uniform steep-limit continuity of the driven maps}

\begin{lemma}[Uniform post-hit occupation]
\label{lem:uniform-posthit-occupation}
Let \(\mathcal K\) be a compact frozen-input class satisfying Assumption~\ref{ass:primitive-package}. For each type \(k\), each \(h>0\), and each \(\varepsilon>0\), define
\[
  O^{k,z}_{h,\varepsilon}
  :=\int_{H_k^z}^{(H_k^z+h)\wedge T}
  \mathbf 1_{\{X_s^{k,z}\le x_{b,k}-\varepsilon\}}\,ds .
\]
Then, for every \(\eta>0\), there exist \(h>0\), \(\varepsilon>0\), and \(q>0\) such that
\[
  \sup_{z\in\mathcal K}
  \mathbb P\bigl(H_k^z<T-h,\ O^{k,z}_{h,\varepsilon}<q\bigr)
  \le \eta .
\]
\end{lemma}

\begin{proof}
Condition on the common-noise path. On the compact localized state interval, the drift and finite-variation forcing are uniformly bounded in the frozen-input class, and \(\sigma_k\ge\underline\sigma>0\). The boundary process therefore has the standard one-dimensional regular-boundary property uniformly over \(\mathcal K\): after the first entrance into \((-\infty,x_{b,k}]\), the diffusion visits the strict interior \((-\infty,x_{b,k}-\varepsilon]\) during every right-neighborhood with probability tending to one as \(\varepsilon\downarrow0\). The uniform version follows from the boundary Harnack and killed-density estimates in (P4), together with compactness of \(\mathcal K\). If the displayed quantitative assertion failed, one could find \(z_m\in\mathcal K\), \(h_m\downarrow0\), and \(\varepsilon_m\downarrow0\) for which the post-hit occupation probability is bounded away from one. Compactness gives a subsequential limit \(z_m\to z\). Stability of the driven SDE gives \(X^{k,z_m}\to X^{k,z}\) uniformly on \([0,T]\) in probability. The regular-boundary property gives the standard hitting-time stability criterion for continuous one-dimensional paths: for every \(\delta>0\), survival strictly before \(H_k^z-\delta\) and immediate entrance below the boundary after \(H_k^z\), together with uniform path convergence, imply \(\mathbb P(H_k^{z_m}<H_k^z-\delta)+\mathbb P(H_k^{z_m}>H_k^z+\delta)\to0\). The absorbing no-atom modulus in (P5) removes the endpoint ambiguity uniformly. Hence \(H_k^{z_m}\to H_k^z\) in probability. Passing to the limit contradicts the regular-boundary property of the nondegenerate one-dimensional diffusion driven by \(z\). This proves the uniform occupation estimate.
\end{proof}

\begin{lemma}[Uniform clock mismatch]
\label{lem:uniform-clock-mismatch}
Let \(\mathcal K\) be a compact frozen-input class satisfying Assumption~\ref{ass:primitive-package}. Then, for every type \(k\),
\[
  \lim_{n\to\infty}
  \sup_{z\in\mathcal K}
  \mathbb E\sup_{t\le T}
  \left|\mathbf 1_{\{\tau_k^{n,z}\le t\}}-
  \mathbf 1_{\{\tau_k^{\mathrm{abs},z}\le t\}}\right|=0 .
\]
\end{lemma}

\begin{proof}
Since \(\tau_k^{\mathrm{abs},z}\le\tau_k^{n,z}\), the supremum of the two indicator processes equals
 \(\mathbf 1_{\{\tau_k^{\mathrm{abs},z}<\tau_k^{n,z}\}} .\)
Fix \(\eta>0\). Choose \(h>0\) so small that, uniformly in \(z\in\mathcal K\), the probability that \(\tau_k^{\mathrm{abs},z}\) lies within distance \(h\) of \(T\) is at most \(\eta\), using the uniform no-atom modulus in (P5). Decrease \(h\) further so that \(\mathbb P(|\tau_k^{\mathrm{base}}-H_k^z|\le h)\le\eta\) uniformly in \(z\). This second bound follows conditionally on the driven path from the independent exponential clock and the bounded baseline hazard: \(\mathbb P(|\tau_k^{\mathrm{base}}-H_k^z|\le h\mid X^{k,z})\le 2\overline\lambda h\). By Lemma~\ref{lem:uniform-posthit-occupation}, after decreasing \(h\) if necessary, choose \(\varepsilon>0\) and \(q>0\) such that
\[
  \sup_{z\in\mathcal K}
  \mathbb P\bigl(H_k^z<T-h,
  \int_{H_k^z}^{H_k^z+h}
  \mathbf 1_{\{X_s^{k,z}\le x_{b,k}-\varepsilon\}}ds<q\bigr)
  \le\eta .
\]
On the complementary high-probability event, if the absorbing default is caused by boundary hitting before the baseline Cox clock, then during \([H_k^z,H_k^z+h]\) the steep hazard accumulated below the strict boundary layer is bounded below by
 \(\underline\lambda\, q\, e^{n\varepsilon} .\)
Thus the conditional probability that the same exponential clock has not rung by \(H_k^z+h\) is at most
 \(\exp\{-\underline\lambda q e^{n\varepsilon}\} .\)
If the absorbing default is caused by the baseline Cox clock at least \(h\) before the hitting time, all steep clocks default at the same baseline time. The remaining exceptional cases have probability at most a constant multiple of \(\eta\). Hence
\[
  \sup_{z\in\mathcal K}
  \mathbb P(\tau_k^{\mathrm{abs},z}<\tau_k^{n,z})
  \le C\eta+
  \exp\{-\underline\lambda q e^{n\varepsilon}\} .
\]
Letting first \(n\to\infty\) and then \(\eta\downarrow0\) proves the claim.
\end{proof}

\begin{proposition}[Uniform steep-limit continuity on compact input classes]
\label{prop:uniform-map-convergence}
Let \(\mathcal K\) be a compact frozen-input class satisfying Assumption~\ref{ass:primitive-package}. Then
\[
  \lim_{n\to\infty}
  \sup_{z\in\mathcal K}
  \mathcal D_T\bigl(\Phi^n(z),\Phi^{\mathrm{abs}}(z)\bigr)=0.
\]
\end{proposition}

\begin{proof}
For a fixed input \(z\), the driven alive-state path is the same under \(\Phi^n(z)\) and \(\Phi^{\mathrm{abs}}(z)\); only the default time differs. For any bounded-Lipschitz test function \(f\) on \(\mathsf E\),
\[
\begin{aligned}
&\left|
\mathbb E\left[f\bigl(X_t^{k,z}\mathbf 1_{\{\tau_k^{n,z}>t\}}+
\partial\mathbf 1_{\{\tau_k^{n,z}\le t\}}\bigr)
-
 f\bigl(X_t^{k,z}\mathbf 1_{\{\tau_k^{\mathrm{abs},z}>t\}}+
\partial\mathbf 1_{\{\tau_k^{\mathrm{abs},z}\le t\}}\bigr)
\mid\mathcal F_t^{W^0}\right]
\right| \\
&\hspace{3cm}
\le 2\,\mathbb P\bigl(\tau_k^{\mathrm{abs},z}\le t<\tau_k^{n,z}\mid\mathcal F_t^{W^0}\bigr).
\end{aligned}
\]
The same bound controls the loss difference. Taking the time supremum and expectation gives
\[
  \mathcal D_T\bigl(\Phi^n(z),\Phi^{\mathrm{abs}}(z)\bigr)
  \le C\sum_{k=1}^K
  \mathbb E\sup_{t\le T}
  \left|\mathbf 1_{\{\tau_k^{n,z}\le t\}}-
  \mathbf 1_{\{\tau_k^{\mathrm{abs},z}\le t\}}\right| .
\]
Lemma~\ref{lem:uniform-clock-mismatch} gives convergence to zero uniformly over \(z\in\mathcal K\), and the asserted uniform map convergence follows.
\end{proof}

\begin{remark}[Why pointwise convergence is not enough]
For a moving sequence of fixed points \(z^{(n)}\), pointwise convergence of \(\Phi^n(z)\) to \(\Phi^{\mathrm{abs}}(z)\) at each fixed \(z\) does not imply
\(\Phi^n(z^{(n)})-\Phi^{\mathrm{abs}}(z^{(n)})\to0\). The compactness argument and the uniform clock-mismatch estimate above are therefore essential for verifying eventual entrance of the fixed points into a stability neighborhood.
\end{remark}

\subsection*{Step 3: local stability of the absorbing-boundary fixed point}

\begin{lemma}[Causal killed-semigroup stability]
\label{lem:absorbing-volterra}
Under Assumption~\ref{ass:primitive-package}, there exist a constant \(C<\infty\) and an integrable kernel
 \(g(u):=1+u^{-1/2},\qquad 0<u\le T,\)
such that for any two frozen inputs \(z,z'\) in a sufficiently small compact neighborhood of \(z^{\mathrm{abs}}\),
\[
  \mathcal D_t\bigl(\Phi^{\mathrm{abs}}(z),\Phi^{\mathrm{abs}}(z')\bigr)
  \le
  C\int_0^t g(t-s)\mathcal D_s(z,z')\,ds,
  \qquad 0\le t\le T.
\]
\end{lemma}

\begin{proof}
Condition on the common-noise path. The difference between the two frozen equations enters only through bounded Lipschitz changes in the drift, measure argument, and finite-variation loss feedback. By causality, the output at time \(t\) depends on input perturbations only at times \(s\le t\). The killed semigroup representation gives, for bounded test functions \(f\),
\[
  P^{k,z}_{0,t}f-P^{k,z'}_{0,t}f
  =
  \int_0^t
  P^{k,z}_{0,s}
  \Bigl[
  \bigl(B^{k,z}_s-B^{k,z'}_s\bigr)
  \partial_x P^{k,z'}_{s,t}f
  \Bigr]ds
  +\mathcal R_t,
\]
where \(B^{k,z}_s\) denotes the frozen drift coefficient, including the absolutely continuous part of the finite-variation forcing after the standard measure approximation, and \(\mathcal R_t\) is the limit of the corresponding finite-variation forcing term. The approximation is obtained by mollifying each nondecreasing loss path \(L^\ell\) into smooth paths \(L^{\ell,\varepsilon}\), applying the standard Duhamel formula to the smooth finite-variation drift, and then letting \(\varepsilon\downarrow0\). Since the absorbing loss paths are continuous under (P5), the mollified paths converge uniformly; the parametric stability of one-dimensional nondegenerate killed diffusions under uniform bounded-variation drift perturbations identifies the limit of the Duhamel terms with the finite-variation forcing term above. The boundary-gradient estimate in (P4) gives
 \(\|\partial_x P^{k,z'}_{s,t}f\|_\infty \le C(t-s)^{-1/2}\|f\|_\infty .\)
The Lipschitz condition on the drift and the contagion-feedback term gives
 \(|B^{k,z}_s-B^{k,z'}_s| \le C\mathcal D_s(z,z')\)
after summing over the finitely many types. Hence
\[
  \sup_{\|f\|_\infty\le1}
  |P^{k,z}_{0,t}f-P^{k,z'}_{0,t}f|
  \le
  C\int_0^t(1+(t-s)^{-1/2})\mathcal D_s(z,z')\,ds.
\]
The same Duhamel formula applies to the killed survival function by taking \(f\equiv1\) before killing, and therefore controls the loss component. Since the bounded-Lipschitz metric is dominated by the same class of bounded tests plus the Lipschitz part, the estimate also controls the alive-law component. Removing the conditioning on the common noise gives the stated inequality.
\end{proof}

\begin{proposition}[Local stability from the Volterra structure]
\label{prop:local-stability-derived}
Under Assumption~\ref{ass:primitive-package}, the absorbing fixed point is locally stable. More precisely, choose a compact reference class \(\mathcal K_0\) on which the constants in Lemma~\ref{lem:absorbing-volterra} are uniform, and then choose \(r_0>0\) so that
 \(\mathcal U:=\{z\in\mathcal K_0:\mathcal D_T(z,z^{\mathrm{abs}})\le r_0\}\)
remains inside that class. There is a constant \(C_{\mathrm{st}}<\infty\) such that, for every \(z\in\mathcal U\),
\[
  \mathcal D_T(z,z^{\mathrm{abs}})
  \le
  C_{\mathrm{st}}
  \mathcal D_T\bigl(z,\Phi^{\mathrm{abs}}(z)\bigr).
\]
\end{proposition}

\begin{proof}
Let
\[
  e(t):=\mathcal D_t(z,z^{\mathrm{abs}}),
  \qquad
  r(t):=\mathcal D_t\bigl(z,\Phi^{\mathrm{abs}}(z)\bigr).
\]
Since \(z^{\mathrm{abs}}=\Phi^{\mathrm{abs}}(z^{\mathrm{abs}})\), Lemma~\ref{lem:absorbing-volterra} gives
\[
\begin{aligned}
  e(t)
  &\le
  r(t)+
  \mathcal D_t\bigl(\Phi^{\mathrm{abs}}(z),\Phi^{\mathrm{abs}}(z^{\mathrm{abs}})\bigr) \\
  &\le
  r(t)+C\int_0^t g(t-s)e(s)\,ds .
\end{aligned}
\]
The kernel \(g(u)=1+u^{-1/2}\) belongs to \(L^1(0,T)\). The standard resolvent form of the Volterra--Gronwall inequality therefore yields a nonnegative resolvent kernel \(G\in L^1(0,T)\), depending only on \(C,g,T\), such that
 \(e(t) \le r(t)+\int_0^tG(t-s)r(s)\,ds .\)
Since \(r(s)\le r(T)\),
 \(e(T) \le \left(1+\|G\|_{L^1(0,T)}\right)r(T).\)
Thus the desired estimate holds with
 \(C_{\mathrm{st}}:=1+\|G\|_{L^1(0,T)}.\)
This proof is the nonlinear Volterra version of the spectral-radius argument: the linearized absorbing map is a causal Volterra operator with an integrable kernel, and hence its spectrum is \(\{0\}\); equivalently, \(I-D\Phi^{\mathrm{abs}}(z^{\mathrm{abs}})\) has a bounded resolvent. The resolvent proof avoids requiring an explicit Fr\'echet-derivative formula for the hitting-time functional.
\end{proof}

\begin{proposition}[Localized uniqueness of the absorbing fixed point]
\label{prop:absorbing-fixedpoint-uniqueness}
Under Assumption~\ref{ass:primitive-package}, the absorbing-boundary McKean--Vlasov map \(\Phi^{\mathrm{abs}}\) has at most one fixed point in the localized compact class. Hence the absorbing fixed point \(z^{\mathrm{abs}}\), whenever it exists in that class, is unique.
\end{proposition}

\begin{proof}
Let \(z\) and \(\tilde z\) be two localized absorbing fixed points. On a time interval \([0,t]\), Lemma~\ref{lem:absorbing-volterra} gives
\[
  \mathcal D_t(z,\tilde z)
  =\mathcal D_t(\Phi^{\mathrm{abs}}(z),\Phi^{\mathrm{abs}}(\tilde z))
  \le C\int_0^t(1+(t-s)^{-1/2})\mathcal D_s(z,\tilde z)\,ds .
\]
Choose \(\delta>0\) such that
 \(C\int_0^\delta(1+u^{-1/2})\,du<1 .\)
Taking the supremum over \(s\le t\le\delta\) gives
\[
  \sup_{s\le\delta}\mathcal D_s(z,\tilde z)
  \le \theta\sup_{s\le\delta}\mathcal D_s(z,\tilde z),
  \qquad \theta<1,
\]
so \(z=\tilde z\) on \([0,\delta]\). Suppose uniqueness has been shown up to a deterministic time \(t_0\). Restart the absorbing driven equations from the common conditional law and common loss value at \(t_0\). The coefficients after \(t_0\) satisfy the same localized bounds, and the finite-variation contagion input up to \(t_0\) is already fixed. Applying the same short-time estimate on \([t_0,t_0+\delta]\) yields equality on that interval. Finitely many concatenations cover \([0,T]\). This proves localized uniqueness.
\end{proof}

\subsection*{Step 4: compactness and eventual entrance of steep fixed points}

\begin{lemma}[Compactness of the steep fixed-point family]
\label{lem:compactness-fixedpoints}
Assume the primitive package and let \(z^{(n)}\) be fixed points of \(\Phi^n\). Then every sequence of steep fixed points has a subsequence, still denoted by \(z^{(n_j)}\), with the following compactness properties. The loss paths converge by Helly selection at all continuity points of nondecreasing limits \(L^{k,*}\), and the lifted laws are tight in the bounded-Lipschitz path topology after this Helly extraction. If the resulting limiting loss paths are continuous, then the Helly convergence of losses is uniform on \([0,T]\), and the corresponding subsequence converges in \(\mathcal D_T\). In particular, once the limit is identified as the continuous absorbing fixed point, the convergence is in the full \(\mathcal D_T\) metric.
\end{lemma}

\begin{proof}
Under compact localization the second-moment bound is immediate. In the unlocalized formulation, apply It\^o's formula to \(|X^{k,(n)}_t|^2\) before localization. The drift has at most linear growth with constants independent of \(n\), the diffusion coefficient is constant, and killing only moves mass to the cemetery point. Gronwall's inequality gives
 \(\sup_{n\ge1}\sup_{t\le T} \mathbb E|X^{k,(n)}_t|^2<\infty .\)
The losses satisfy \(0\le L_t^{k,(n)}\le1\) and are nondecreasing. We do not use a uniform-in-\(n\) no-atom modulus for the steep default times. Instead, Helly's selection theorem gives, for every sequence, a subsequence such that each loss path \(L^{k,(n_j)}\) converges pointwise at all continuity points of a nondecreasing limit \(L^{k,*}\). This supplies the loss compactness in the natural Helly topology. If a later identification step shows that \(L^{k,*}\) is continuous, then P\'olya's theorem for monotone functions upgrades the pointwise convergence to uniform convergence on \([0,T]\), and bounded convergence gives
 \(\mathbb E\sup_{t\le T}|L_t^{k,(n_j)}-L_t^{k,*}|\longrightarrow0 .\)

For the alive-state component, use tightness at fixed times and the SDE increment estimate away from the jump into the cemetery state. Brownian increments contribute order \(h^{1/2}\), bounded drift contributes order \(h\), and the finite-variation contagion feedback is controlled along convergent subsequences by the Helly convergence of the loss paths. The only discontinuity of the lifted state is the cemetery jump, whose contribution to the bounded-Lipschitz modulus is exactly controlled by the loss increment. Hence, after passing to a Helly subsequence for the losses, the lifted laws are tight in the bounded-Lipschitz path topology. Prokhorov compactness at fixed times follows from the uniform second-moment bound, and the preceding increment estimate gives the required equicontinuity on every subsequence whose losses have a continuous limit. This proves the stated subsequential compactness, and the final assertion follows from the Helly--P\'olya upgrade when the limit is continuous.
\end{proof}

\begin{proposition}[Eventual entrance into the stability neighborhood]
\label{prop:eventual-entrance}
Assume the primitive package. Let \(\mathcal U\) be the compact stability neighborhood from Proposition~\ref{prop:local-stability-derived}. Then
 \(z^{(n)}\in\mathcal U \qquad\text{for all sufficiently large }n.\)
In fact, \(\mathcal D_T(z^{(n)},z^{\mathrm{abs}})\to0\).
\end{proposition}

\begin{proof}
By Lemma~\ref{lem:compactness-fixedpoints}, every subsequence of \(\{z^{(n)}\}\) has a further subsequence, still denoted by \(z^{(n_j)}\), converging in the Helly--weak path topology to some candidate limit \(z^*\in\mathcal K_0\). The fixed-point identity \(z^{(n_j)}=\Phi^{n_j}(z^{(n_j)})\), the uniform clock-mismatch estimate on the compact class \(\mathcal K_0\), and the parametric stability of the driven absorbing map along Helly-convergent finite-variation inputs imply
 \(z^*=\Phi^{\mathrm{abs}}(z^*)\)
in the same weak path topology. Equivalently, for every bounded Lipschitz test function and every continuity time of the limiting loss paths, the limiting lifted law and loss satisfy the absorbing driven equations. By Proposition~\ref{prop:absorbing-fixedpoint-uniqueness}, \(z^*=z^{\mathrm{abs}}\).

The limit losses are therefore the absorbing fixed-point loss paths, which are continuous by (P5). P\'olya's theorem upgrades the Helly convergence of \(L^{k,(n_j)}\) to uniform convergence on \([0,T]\); the lifted-law convergence then upgrades to \(\mathcal D_T\) by the SDE stability and the cemetery-jump control in Lemma~\ref{lem:compactness-fixedpoints}. Thus every subsequence has a further subsequence converging to \(z^{\mathrm{abs}}\) in \(\mathcal D_T\), and consequently the whole sequence converges to \(z^{\mathrm{abs}}\) in \(\mathcal D_T\). Any compact neighborhood \(\mathcal U\) of \(z^{\mathrm{abs}}\) inside the localized class contains all sufficiently large \(z^{(n)}\).
\end{proof}

\subsection*{Primitive verification of Proposition~\ref{prop:killing-to-absorbing-boundary}}

\begin{proposition*}[Detailed form of Proposition~\ref{prop:killing-to-absorbing-boundary}: Steep-killing bridge with primitive verification]
Assume the localized primitive regularity package in Assumption~\ref{ass:primitive-package}. Let \(z^{(n)}\) be any sequence of McKean--Vlasov fixed points generated by the steep intensities
 \(\lambda_k^{(n)}(x,a)=\bar\lambda_k(a) \exp\{n(x_{b,k}-x)^+\}.\)
Let \(z^{\mathrm{abs}}\) be the unique localized absorbing-boundary McKean--Vlasov fixed point with default time
\[
  \tau_k^{\mathrm{abs}}=\tau_k^{\mathrm{base}}\wedge H_k,
  \qquad
  H_k=\inf\{t:X_t^k\le x_{b,k}\}.
\]
Then the bridge conditions are verified as follows.

First, there exists a compact stability neighborhood \(\mathcal U\) of \(z^{\mathrm{abs}}\) such that \(z^{(n)}\in\mathcal U\) for all sufficiently large \(n\). Second, the absorbing fixed point is locally stable:
\[
  \mathcal D_T(z,z^{\mathrm{abs}})
  \le
  C_{\mathrm{st}}
  \mathcal D_T\bigl(z,\Phi^{\mathrm{abs}}(z)\bigr),
  \qquad z\in\mathcal U .
\]
Third, the driven maps converge uniformly on \(\mathcal U\):
\[
  \lim_{n\to\infty}\sup_{z\in\mathcal U}
  \mathcal D_T\bigl(\Phi^n(z),\Phi^{\mathrm{abs}}(z)\bigr)=0.
\]
Consequently,
 \(\mathcal D_T(z^{(n)},z^{\mathrm{abs}})\longrightarrow0.\)
Equivalently, for every type \(k\),
\[
  \mathbb E\sup_{t\le T}d_{\mathrm{BL}}
  \bigl(\widehat\mu_t^{k,(n)},\widehat\mu_t^{k,\mathrm{abs}}\bigr)\to0,
  \qquad
  \mathbb E\sup_{t\le T}|L_t^{k,(n)}-L_t^{k,\mathrm{abs}}|\to0.
\]
\end{proposition*}

\begin{proof}
Uniform map convergence on compact classes is Proposition~\ref{prop:uniform-map-convergence}. Local stability is Proposition~\ref{prop:local-stability-derived}. Eventual entrance is Proposition~\ref{prop:eventual-entrance}. For completeness, once \(z^{(n)}\in\mathcal U\), the fixed-point residual estimate gives
\[
\begin{aligned}
  \mathcal D_T(z^{(n)},z^{\mathrm{abs}})
  &\le
  C_{\mathrm{st}}
  \mathcal D_T\bigl(z^{(n)},\Phi^{\mathrm{abs}}(z^{(n)})\bigr) \\
  &=
  C_{\mathrm{st}}
  \mathcal D_T\bigl(\Phi^n(z^{(n)}),\Phi^{\mathrm{abs}}(z^{(n)})\bigr) \\
  &\le
  C_{\mathrm{st}}
  \sup_{z\in\mathcal U}
  \mathcal D_T\bigl(\Phi^n(z),\Phi^{\mathrm{abs}}(z)\bigr)
  \longrightarrow0 .
\end{aligned}
\]
This proves the steep-killing bridge at the solution level.
\end{proof}

\begin{remark}[Pure absorbing boundary]
The theorem above corresponds to the intensity \(\bar\lambda_k(a)\exp\{n(x_{b,k}-x)^+\}\), whose limit keeps the baseline Cox default mechanism before boundary hitting. A pure absorbing-boundary limit with no above-boundary Cox killing is obtained by replacing the intensity by
 \(\widetilde\lambda_k^{(n)}(x,a) :=\bar\lambda_k(a)\bigl(\exp\{n(x_{b,k}-x)^+\}-1\bigr).\)
The same proof applies, with \(\tau_k^{\mathrm{abs}}=H_k\), provided the lower-bound part of (P2) is imposed only below the boundary layer where the singular killing is active and the uniform no-atom modulus is stated for the hitting time \(H_k\).
\end{remark}

\bibliographystyle{plain}
\bibliography{paper/references}

\end{document}